\documentclass[a4paper, 10pt, twoside, notitlepage]{amsart}

\usepackage{amsmath,amscd}
\usepackage{amssymb}
\usepackage{amsthm}
\usepackage{comment}
\usepackage{graphicx, xcolor}

\usepackage{mathrsfs}
\usepackage[ocgcolorlinks, linkcolor=blue]{hyperref}

\usepackage{bm}
\usepackage{bbm}
\usepackage{url}

\usepackage[utf8]{inputenc}
\usepackage{mathtools,amssymb}
\usepackage{esint}
\usepackage{tikz}
\usepackage{dsfont}
\usepackage{relsize}
\usepackage{url}
\usepackage{xcolor}
\usepackage{graphicx}
\usepackage{mathrsfs}
\usepackage[shortlabels]{enumitem}
\usepackage{lineno}
\usepackage{amsmath}
\usepackage{enumitem}
\usepackage{amsthm} 
\usepackage{verbatim}
\usepackage{dsfont}
\numberwithin{equation}{section}

\allowdisplaybreaks

\mathtoolsset{showonlyrefs}

\graphicspath{{images/}}

\newtheorem{theorem}{Theorem}[section]
\newtheorem{lemma}[theorem]{Lemma}
\newtheorem{definition}[theorem]{Definition}
\newtheorem{corollary}[theorem]{Corollary}

\newtheorem{claim}[theorem]{Claim}
\newtheorem{proposition}[theorem]{Proposition}
\newtheorem{question}{Question}
\newtheorem{remark}[theorem]{Remark}

\title[Inverse problem for nonlocal porous medium equations]{Unique determination of coefficients and kernel in nonlocal porous medium equations with absorption term}

\author[Y.-H. Lin]{Yi-Hsuan Lin}
\address{Department of Applied Mathematics, National Yang Ming Chiao Tung University, Hsinchu, Taiwan}
\email{yihsuanlin3@gmail.com}

\author[P. Zimmermann]{Philipp Zimmermann}
\address{Department of Mathematics, ETH Zurich, Z\"urich, Switzerland}
\email{philipp.zimmermann@math.ethz.ch}

\newcommand{\R}{{\mathbb R}}

\newcommand{\N}{{\mathbb N}}

\newcommand{\eps}{\epsilon}
\newcommand{\vareps}{\varepsilon}

\newcommand {\p} {\partial}

\newcommand{\LC}{\left(}
\newcommand{\RC}{\right)}
\newcommand{\wt}{\widetilde}

\newcommand{\tempered}{\mathscr{S}^{\prime}}

\newcommand{\test}{\mathscr{D}}
\newcommand{\distr}{\mathscr{D}^{\prime}}
\newcommand{\norm}[1]{\lVert #1 \rVert}
\DeclareMathOperator{\Div}{div} 
\DeclareMathOperator{\supp}{supp} 
\DeclareMathOperator{\dist}{dist} 

\newcommand{\weak}{\rightharpoonup}
\newcommand{\weakstar}{\overset{\ast}{\rightharpoonup}}

\begin{document}

	\maketitle
	\begin{abstract}
		The main purpose of this article is the study of an inverse problem for nonlocal porous medium equations (NPMEs) with a linear absorption term. More concretely, we show that under certain assumptions on the time-independent coefficients $\rho,q$ and the time-independent kernel $K$ of the nonlocal operator $L_K$, the (partial) Dirichlet-to-Neumann map uniquely determines the three quantities $(\rho,K,q)$ in the nonlocal porous medium equation $\rho \partial_tu+L_K(u^m)+qu=0$, where $m>1$.  In the first part of this work we adapt the Galerkin method to prove existence and uniqueness of nonnegative, bounded solutions to the homogenoeus NPME with regular initial and exterior conditions. Additionally, a comparison principle for solutions of the NPME is proved, whenever they can be approximated by sufficiently regular functions like the one constructed for the homogeneous NPME. These results are then used in the second part to prove the unique determination of the coefficients $(\rho,K,q)$ in the inverse problem. Finally, we show that the assumptions on the nonlocal operator $L_K$ in our main theorem are satisfied by the fractional conductivity operator $\mathcal{L}_{\gamma}$, whose kernel is $\gamma^{1/2}(x)\gamma^{1/2}(y)/|x-y|^{n+2s}$ up to a normalization constant.
		
		\medskip
		
		

	\end{abstract}

	\tableofcontents

	\section{Introduction}
	\label{sec: introduction}
	
	In recent years inverse problems for a wide class of nonlocal partial differential equations (PDEs) have been studied. The most classical example is the \emph{fractional Calder\'on problem}. In this problem one considers Dirichlet problem for the \emph{fractional Schr\"odinger equation}
	\begin{equation}
		\label{eq: fractional schroedinger}
		\begin{cases}
			((-\Delta)^s+q)u=0 &\text{ in }\Omega,\\
			u=\varphi &\text{ in }\Omega_e,
		\end{cases}
	\end{equation}
	where $0<s<1$, $(-\Delta)^s$ is the fractional Laplacian, $q$ is a given potential and $\Omega_e=\R^n\setminus\overline{\Omega}$ is the exterior of the bounded domain $\Omega\subset\R^n$. The fractional Calder\'on problem now asks to uniquely determine the potential $q$ from the (exterior) Dirichlet-to-Neumann (DN) map $\varphi\mapsto \Lambda_q(\varphi)=\left. (-\Delta)^su \right|_{\Omega_e}$. A first affirmative answer has been established in \cite{GSU20} for bounded potentials.

	Later, this work initiated many further developments in the field of nonlocal inverse problems, which includes determination of singular potentials, lower order local perturbations, higher order fractional Laplacians, single measurement results, generalizations to other nonlocal operators in place of the fractional Laplacian and unbounded domains (see \cite{bhattacharyya2021inverse,CMR20,CMRU20,GLX,CL2019determining,CLL2017simultaneously,cekic2020calderon,feizmohammadi2021fractional,harrach2017nonlocal-monotonicity,harrach2020monotonicity,GRSU18,GU2021calder,ghosh2021non,lin2020monotonicity,LL2020inverse,LL2022inverse,LLR2019calder,LLU2022calder,KLW2021calder,RS17,ruland2018exponential,RZ2022unboundedFracCald} and the references therein). Another type of inverse problems for nonlocal operators has been investigated for example in the articles \cite{RZ2022unboundedFracCald,RZ2022FracCondCounter,CRZ2022global,RZ-low-2022,CRTZ-2022,zimmermann2023inverse,GU2021calder,li2021inverse}, where the authors try to recover coefficients in a nonlocal operator from the DN map instead of determining lower order perturbations of a nonlocal operator. The solvability of the inverse problems in the above works strongly depend on the linear structure of the nonlocal operator, the unique continuation principle (UCP) and the Runge approximation. The last property is a consequence of the UCP, which in turn can be phrased as follows:\\
	
	\noindent\emph{\textbf{Unique continuation principle.}
		Let $X$ be a Banach space satisfying $C_c^{\infty}(\R^n)\hookrightarrow X\hookrightarrow\distr(\R^n)$. We say that an operator $L\colon X\to \distr(\R^n)$ has the UCP on $X$ if $Lu=u=0$ in some nonempty open set $\Omega$ implies $u=0$ in $\R^n$. 
	}\\
	
	In particular, in the article \cite{GSU20} it has been shown that the fractional Laplacian $(-\Delta)^s$, $0<s<1$, has the UCP on $X=H^t(\R^n)$ for any $t\in\R$. The same clearly remains true for local perturbations of fractional Laplacians. In \cite{GLX}, it has been proved that fractional powers of second order operators $L=-\Div(A\nabla \cdot)$, where $A\in C^{\infty}_b(\R^{n\times n};\R)$ is symmetric and uniformly elliptic, satisfy the UCP on $X=H^s(\R^n)$. Furthermore, in \cite{zimmermann2023inverse} it has been established by using the fractional Liouville reduction that the fractional conductivity operator $\mathcal{L}_{\gamma}$ has the UCP as well, at least when the coefficient $\gamma$ is sufficiently regular. The fractional conductivity operator $\mathcal{L}_{\gamma}$ belongs to a subclass of integro-differential operators of order $2s$, that is of operators of the form
	\begin{equation}
		\label{eq: integro diff op}
		L_Ku(x)=\mathrm{P.V.}\int_{\R^n}K(x,y)\frac{u(x)-u(y)}{|x-y|^{n+2s}}\,dy,
	\end{equation}
	where $K\colon \R^n\times\R^n\to \R$ is a given function and $\mathrm{P.V.}$ stands for the Cauchy principal value. In the rest of the article, we abuse the terminology and refer to $K$ as the kernel of $L_K$. Such an operator is said to belong to the class $\mathcal{L}_0$, whenever the kernel $K$ satisfies the following properties:
	\begin{enumerate}[(i)]
		\item $K$ is symmetric in the sense that
		\[
		K(x,y)=K(y,x)\ \text{for all} \ x,y\in\R^n.
		\]
		\item $K$ is uniformly elliptic, that is there holds
		\[
		\lambda \leq K(x,y)\leq \Lambda
		\]
		for all $x,y\in\R^n$ and some constants $0<\lambda \leq \Lambda<\infty$.
	\end{enumerate}
	Now, the fractional conductivity operator $\mathcal{L}_{\gamma}$ is the integro-differential operator of order $2s$ of the form \eqref{eq: integro diff op}, whose kernel is given by 
	\begin{equation}
		\label{eq: kernel fractional conductivity operator}
		K(x,y)=C_{n,s}\gamma^{1/2}(x)\gamma^{1/2}(y)
	\end{equation}
	for some uniformly elliptic function $\gamma\colon\R^n\to \R$ and hence $\mathcal{L}_{\gamma}\in \mathcal{L}_0$. Here $C_{n,s}$ is the constant given by 
	\begin{align}\label{C_ns}
		C_{n,s}:=\frac{4^s \Gamma(n/2+s)}{\pi^{n/2}|\Gamma(-s)|},
	\end{align}
    where $\Gamma$ is the Gamma function. Let us note that in the special case $\gamma=1$ the operator $\mathcal{L}_{\gamma}$ reduces to the fractional Laplacian $(-\Delta)^s$ and we remark that the constant $C_{n,s}$ is defined as in \eqref{C_ns} precisely that the Fourier symbol of $(-\Delta)^s$ is $|\xi|^{2s}$.
	
	Recently, the interplay between nonlocality and nonlinearity of an operator and its influence on the solvability of the related inverse problem has been studied in \cite{KRZ-2023} and \cite{KLZ-2022}. The underlying models in these articles are closely related, but exhibit dramatically different unique continuation properties. In the first article vector-valued generalizations of weighted fractional $p\,$-biharmonic operators
	\begin{equation}
		\label{eq: p biharm}
		(-\Delta)^{s/2}\LC \sigma |(-\Delta)^{s/2}u|^{p-2}(\Delta)^{s/2}u\RC 
	\end{equation}
	has been investigated, where one wants to determine the uniformly elliptic function $\sigma$ from the DN map. This has been achieved by monotonicity methods and the observation that the operators in \eqref{eq: p biharm} satisfy the UCP on the Bessel potential space $X=H^{s,p}(\R^n)$. The latter rests on the fact that the fractional Laplacian also satisfies the UCP on the Bessel potential spaces $X=H^{t,p}(\R^n)$ with $t\in\R$ and $1\leq p<\infty$. On the other hand, in \cite{KLZ-2022} the authors studied weighted fractional $p\,$-Laplacian, which are integro-differential operators of the form
	\begin{equation}
		\label{eq: weighted p lap}
		(-\Delta)^s_{\sigma,p}u(x)=\mathrm{P.V.}\int_{\R^n}\sigma(x,y) |u(x)-u(y)|^{p-2}\frac{u(x)-u(y)}{|x-y|^{n+sp}}\,dxdy,
	\end{equation}
	where $0<s<1$, $1<p<\infty$ and $\sigma\colon \R^n\times\R^n\to\R$ is a uniformly elliptic function. A disadvantage concerning inverse problems is the fact that its not known whether operators of the form \eqref{eq: weighted p lap} satisfy the UCP and hence an alternative approach has been established. In this work an exterior determination method in the spirit of the boundary determination result of Kohn and Vogelius \cite{KV84} has been introduced. A similar method was also used for the classical $p\,$-Calder\'on problem by Salo and Zhong \cite{Salo:Zhong:2012} and for the fractional conductivity problem in \cite{CRZ2022global,RZ-low-2022} or the parabolic fractional conductivity problem in \cite{LRZ2022calder}.
	
	In this article we follow the same line of research, namely we investigate an inverse problem for a nonlocal and nonlinear PDE. More concretely, we study an inverse problem for a class of \emph{nonlocal porous medium equations} (NPMEs), which are of the form
	\begin{equation}
		\label{eq: NPME}
		\rho\partial_t u+L_K(\Phi^m(u))+qu=0\ \text{in}\ \Omega_T.
	\end{equation}
	Here $\Omega\subset\R^n$ is a bounded Lipschitz domain, $T>0$, $\Omega_T=\Omega\times (0,T)$\footnote{Throughout the article we write $A_{\tau}$ to denote the space-time cylinder $A\times (0,\tau)$, whenever $\tau>0$ and $A\subset\R^n$.}, the function $\Phi^m\colon\R\to\R$ is given by
	\begin{equation}
		\label{eq: nonlinearity intro}
		\Phi^m(t)=|t|^{m-1}t\ \text{for some}\ m>1,
	\end{equation}
	$\rho\colon\R^n\to \R$ is uniformly elliptic, $L_K$ an integro-differential operator of order $2s$ in the class $\mathcal{L}_0$ and $q\colon\R^n\to\R$ a nonnegative potential. More concrete assumptions on the coefficients $\rho,q$ and the nonlocal operator $L_K$ will be given below (Section~\ref{subsec: nonlocal porous medium equation}). To motivate the investigation of the PDE \eqref{eq: NPME} and the related inverse problem, we first review in the next section the classical porous medium equation and then discuss natural nonlocal generalizations of it.
	
	\subsection{Local and nonlocal porous medium equations}\label{subsec: PME}
	
	The \emph{porous medium equation} (PME) is one of the simplest examples of a nonlinear evolution equation and its study at least dates back to 1957 \cite{oleinik1957equations}. One of its basic forms is
	\begin{equation}
		\label{eq: PME}
		\partial_t u-\Delta \Phi(u)=0\ \text{in}\ \Omega_T,
	\end{equation}
	where $\Omega\subset\R^n$ can be any domain and $\Phi$ is typically assumed to be of the form $\Phi=\Phi^m$ for some $m>1$ but not restricted to. Let us note that for $m>1$ the PME is degenerate parabolic, whereas for $m=1$ one would obtain the classical heat equation and for $m<1$ a singular PDE or fast diffusion equation. Physically speaking this equation models the gas flow through a porous medium and it is also used to model various phenomena in other fields, such as plasma physics \cite{rosenau1983thermal} and population dynamics \cite{namba1980density}. In the literature many generalizations of \eqref{eq: PME} had been studied like adding a forcing term $f(u,\nabla u)$ on the right hand side. For such a term the dependence on $\nabla u$ incorporates convection of the medium and the dependence on $u$ reaction or absorption effects. For a detailed account on the mathematical theory for the porous medium equation, we refer the reader to the monographs \cite{Adrian60,PME}.
	
	Furthermore, inverse problems for the PME \eqref{eq: PME} have been investigated in \cite{CGN21,CGU21}. In \cite{CGN21}, the authors determined two coefficients $(\rho,\gamma)$ for the PME in the form 
	\begin{equation}
		\label{PME abs}
		\rho \p_t u -\Div (\gamma \nabla \Phi^m(u))=0
	\end{equation}
	for $m>1$. In \cite{CGU21} this result has been generalized to porous medium equations with a possible nonlinear absorption term. They prove a unique determination result for the three parameters $(\rho,\gamma,q)$ in
	\begin{equation}
		\label{eq: gunter PME}
		\rho \p_t u -\Div (\gamma \nabla \Phi^m(u))+q \Phi^r(u)=0,
	\end{equation}
	where $m>1$ and $m^{-1}<r<\sqrt{m}$.

	An interesting, nonlocal generalization of the PME \eqref{eq: PME} is obtained by replacing the Laplacian by the fractional Laplacian, which leads to
	\begin{equation}
		\label{eq: basic fractional porous medium equation}
		\partial_tu+(-\Delta)^s(\Phi^m(u))=0 \text{ in }\Omega_T
	\end{equation}
	for $0<s<1$. This model describes anomalous diffusion through a porous medium and additional information on this equation as well as generalizations of it can be found for example in the articles \cite{RecentProgressFracPorous,BSV-2015,BV-2016,BFX-2017}.
	
	Hence, based on what we said a natural generalization to consider are nonlocal porous medium equations of the form
	\begin{equation}
		\label{eq: nonlinear NPME}
		\rho\partial_t u+L_K(\Phi^m(u))+q\Phi^r(u)=0\ \text{in}\ \Omega_T
	\end{equation}
	(see \eqref{eq: gunter PME} and \eqref{eq: basic fractional porous medium equation}). To simplify the presentation and the technicalities, we will restrict our considerations in this work to the linear case \eqref{eq: NPME} instead of \eqref{eq: nonlinear NPME}.

	\subsection{Inverse problem for nonlocal porous medium equations}
	\label{subsec: nonlocal porous medium equation}

		Next, let us consider the initial-exterior value problem for the NPME
		\begin{equation}\label{eq: main}
			\begin{cases}
				\rho\partial_tu+ L_K\LC \Phi^m (u) \RC +qu=0&  \text{ in }\Omega_T,\\
				u=\varphi &\text{ in }\LC\Omega_e\RC_T,\\
				u(0)=0 & \text{ in }\Omega,
			\end{cases}
		\end{equation}
		where $\Omega\subset\R^n$ is a bounded Lipschitz domain, $L_K\in \mathcal{L}_0$, $m>1$, $\rho\colon \R^n\to \R$ is uniformly elliptic, $q\colon\R^n\to\R$ is a nonnegative potential and $\varphi\colon \Omega_e\times (0,T)\to\R$ is a given exterior condition.

		In Section \ref{sec: forward problem}, it is established that under suitable assumptions on the coefficients $\rho$, $q$ and the exterior condition $\varphi$ the problem \eqref{eq: main} is well-posed. More concretely, we show in Theorem~\ref{Existence result with linear absorption term} and Theorem~\ref{Uniquenerss with lin absorption term} that there exists a unique, nonnegative, bounded solution of \eqref{eq: main}. Therefore, we can now introduce the DN map $\Lambda_{\rho,K,q}$ related to \eqref{eq: main}, which is strongly given by
		\begin{equation}\label{DN map}
			\Lambda_{\rho,K,q}\varphi=\left.L_K(\Phi^m(u))\right|_{(\Omega_e)_T}
		\end{equation}
		for suitable exterior conditions $\varphi$, where $u\colon\R^n\times (0,T)\to\R$ is the unique solution to \eqref{eq: main}. A more detailed account on the DN map is given in Section~\ref{subsec: DN maps}. Now, we can formulate our inverse problem.
		
		\begin{question}
        \label{inverse problem}
			Can one uniquely determine the coefficients $\rho,q$ in $\overline{\Omega}$ and the kernel $K$ in $\R^{2n}$ by the nonlocal DN map $\Lambda_{\rho,K,q}$?
		\end{question}
		
		Next, let us recall that for any $L_K\in\mathcal{L}_0$ one can show by the Lax--Milgram theorem and the fractional Poincar\'e inequality (see \eqref{eq: Poincare int1}) that for any $\varphi\in H^s(\R^n)$ there is a unique weak solution $u\in H^s(\R^n)$ of 
		\begin{equation}
			\label{eq: Dirichelt problem}
			\begin{cases}
				L_K u=0& \text{ in } \Omega,\\
				u=\varphi &\text{ in } \Omega_e.
			\end{cases}
		\end{equation}
		Furthermore, it is not difficult to prove that the exterior value to solution map of \eqref{eq: Dirichelt problem} is well-defined on the (abstract) trace space $Z=H^s(\R^n)/\widetilde{H}^s(\Omega)$. Hence, we can define the DN map related to \eqref{eq: Dirichelt problem} by
		\begin{equation}
			\label{eq: DN map nonlocal operatot}
			\Lambda_K\colon Z\to Z^*,\  \left\langle \Lambda_K \varphi,\psi\right\rangle\vcentcolon = B_K(u,\psi),
		\end{equation}
		where $u\in H^s(\R^n)$ is the unique solution to \eqref{eq: Dirichelt problem}, $\psi\in Z$ and $B_K\colon H^s(\R^n)\times H^s(\R^n)\to \R$ is the naturally bilinear form related to $L_K$ (see~\eqref{eq: bilinear form kernel}).
		
		In general, the determination of the kernel $K$ from the DN map \eqref{eq: DN map nonlocal operatot}
		is a highly nontrivial inverse problem. Nevertheless, it can be expected that if the linear elliptic nonlocal inverse problem is not uniquely solvable, then the same remains true for the inverse problem of the NPME. To rule out this possibility, we next introduce a suitable class of nonlocal operators, so called measurement equivalent operators. In a similar spirit, the authors of \cite{RZ2022unboundedFracCald} formulated general conditions under which local perturbations of linear elliptic nonlocal operators can be uniquely determined from the related DN map.
		
		\begin{definition}[Measurement equivalent operators]
			\label{def: nonlocal op for inverse problem}
			Let $\Omega\subset\R^n$ be a bounded domain and $0<s<1$. We say that two uniformly elliptic nonlocal operators $L_{K_1},L_{K_2}\in \mathcal{L}_0$ of order $2s$ are \emph{measurement equivalent}, written as $L_{K_1}\sim L_{K_2}$, if the condition
			\begin{equation}
				\label{eq: for measurement equivalent}
				\left. \Lambda_{K_1}\varphi\right|_{W_2}=\left. \Lambda_{K_2}\varphi\right|_{W_2},\ \text{for all}\ \varphi\in C_c^{\infty}(W_1)\ \text{with}\ \varphi\geq 0
			\end{equation}
			for some nonempty, open sets $W_1,W_2\subset \Omega_e$ with $W_1\cap W_2\neq \emptyset$, implies $K_1=K_2$. 
		\end{definition}
		
		\begin{remark}
			We assumed that the measurement sets $W_1,W_2\subset \Omega_e$  in the above definition are non-disjoint due to the counterexamples constructed in \cite{RZ2022FracCondCounter} and \cite{RZ-low-2022}. For example in \cite{RZ-low-2022}, it is proved that if the measurement sets are disjoint and have a positive distance to $\Omega$, then one can construct two different conductivities $\gamma_1,\gamma_2$ such that the related fractional conductivity operators $\mathcal{L}_{\gamma_j}$, $j=1,2$, satisfy \eqref{eq: for measurement equivalent} and the regularity assumptions in Proposition~\ref{prop: fractional conductivity operator} below.
		\end{remark}
		
		In Section \ref{subsec: meas equivalence}, we provide some examples of nonlocal operators that fulfill Definition \ref{def: nonlocal op for inverse problem}.
		Now, we are ready to state our main result of this article, which generalizes the recent work \cite{CGU21} on the local porous medium equation to the a nonlocal setting. 
  
		\begin{theorem}\label{thm: main}
			Let $\Omega\subset\R^n$ be a bounded Lipschitz domain, $T>0$, $0<s<\alpha\leq 1$, $m> 1$ and $W_1,W_2\Subset\Omega_e$ two nonempty open sets with $W_1\cap W_2\neq\emptyset$. Assume that for $j=1,2$ we have given nonlocal operators $L_{K_j}\in \mathcal{L}_0$ satisfying the UCP on $H^s(\R^n)$, and coefficients $\rho_j, q_j \in C_+^{1,\alpha}(\R^n)$ such that $\rho_j$ is uniformly elliptic. 
			If $L_{K_1}\sim L_{K_2}$ and there holds
			\begin{align}\label{same DN maps}
				\left.  \Lambda_{\rho_1,K_1,q_1}\varphi\right|_{(W_2)_T}=\left.  \Lambda_{\rho_2,K_2,q_2}\varphi \right|_{(W_2)_T}, 
			\end{align}  
			for any $0\leq \varphi \in C_c([0,T]\times W_1)$ with $ \Phi(\varphi)\in L^2(0,T;H^s(\R^n))$, then we have 
			\begin{align}
				\rho_1 =\rho_2, \ q_1 =q_2 \text{ in } \overline{\Omega} \text{ and } K_1=K_2 \text{ in }\R^{2n}.
			\end{align}
		\end{theorem}


        
        Finally, let us observe that if one can show that the fractional powers $(-\Div (A\nabla \cdot))^s$ induce measure equivalent operators, then Theorem~\ref{thm: main} implies that the related inverse problems for the NPME are uniquely solvable (see~Remark~\ref{remark: measurement equiv op}). 
		
		
		\subsection{Inverse problem for the porous medium equation}

		In this section, we will have a closer look at some aspects of the proof of the unique determination in the inverse problem for the porous medium equation \eqref{PME abs}, since some of these ideas also enter into the proof of Theorem \ref{thm: main}. For this purpose let us consider the initial-boundary value problem related \eqref{PME abs} for $r=1$, namely
		\begin{equation}
			\label{eq: PME with absorption}
			\begin{cases}
				\rho\partial_tu-\Div(\gamma\nabla   \Phi^m(u))+q  u=0 &\text{ in }\Omega_T,\\
				u=\varphi & \text{ on }\p\Omega_T,\\
				u(0)=0 &\text{ in }\Omega.
			\end{cases}
		\end{equation}
		

		For any regular boundary value $\varphi\geq 0$, one can prove under suitable regularity assumptions on $\rho$, $\gamma$ and $q$ the existence of a unique nonnegative solution to \eqref{eq: PME with absorption} (see \cite{CGU21}). Hence, the natural measurements related to \eqref{eq: PME with absorption} can be encoded in the DN map 
		\begin{equation}
        \label{eq: DN map local case}
	       	\varphi\mapsto \Lambda_{\rho,\gamma,q}\varphi\vcentcolon = \left.\gamma\partial_{\nu}\Phi^m(u)\right|_{\partial\Omega_T}.
		\end{equation}
		In \cite{CGU21}, they showed that if one has
        \begin{equation}
        \label{eq: equal dn maps local case}
            \Lambda_{\rho_1,\gamma_1,q_1}\varphi=\Lambda_{\rho_2,\gamma_2,q_2}\varphi
        \end{equation}
        for all sufficiently regular boundary data $\varphi\geq 0$, then there holds $\rho_1=\rho_2$, $\gamma_1=\gamma_2$ and $q_1=q_2$ in $\Omega$. Furthermore, under the additional assumption $\gamma_1=\gamma_2=1$ in $\Omega$, the authors of \cite{CGU21} also obtained a partial data uniqueness result. Actually, these results were established for $r>0$ satisfying $m^{-1}<r<m^{1/2}$. 
		
		Next, we describe the main idea of the uniqueness proof for the inverse problem related to \eqref{eq: PME with absorption}. In the first step, one introduces the new function $v:=\Phi^m(u)=u^m$, which solves the problem 
		\begin{equation}
			\label{eq: PME with absorption power}
			\begin{cases}
				\rho\partial_tv^{1/m}-\Div(\gamma\nabla   v)+q v^{\frac{1}{m}}=0   & \text{ in }\Omega\times (0,T),\\
				v=f  & \text{ in }\partial\Omega\times (0,T),\\
				v(0)=0 &\text{ in }\Omega,
			\end{cases}
		\end{equation}
		with $f=\varphi^m$, whenever $u$ is a nonnegative solution to \eqref{eq: PME with absorption}. The DN map $\Lambda^{red.1}_{\rho,\gamma,q}$ related to \eqref{eq: PME with absorption power} satisfies 
        \begin{equation}
            \Lambda_{\rho,\gamma,q}\varphi=\Lambda^{red.1}_{\rho,\gamma,q}\varphi^m.
        \end{equation}
        Next consider large, time homogeneous boundary values of the form $$f(x,t)=h t^mg(x)$$ with $g\geq 0$ on $\p \Omega$ and assume $h\gg 1$ is a fixed large parameter. Now the essential observation is that the time-integral transform 
		\[
		V(x)=\int_0^T(T-t)^{\alpha}v(x,t)\,dt,
		\]
		gives a solution to Dirichlet problem 
		\begin{equation}
			\label{eq: reduction 2}
			\begin{cases}
				\Div (\gamma\nabla V)=\mathcal{N}_t+\mathcal{N}_a & \text{ in }\quad\Omega,\\
				V=h T^{1+\alpha+m}\frac{\Gamma(1+\alpha)\Gamma(1+m)}{\Gamma(2+\alpha+m)}g & \text{ on }\quad\partial\Omega,
			\end{cases}
		\end{equation}
		where the parameters $T>0$ and $\alpha>0$ are determined later. The DN map of \eqref{eq: reduction 2} is given by 
        \begin{equation}
            \Lambda_{\rho,\gamma,q}^{red. 2}g=\left.h^{-1}\gamma\partial_{\nu}V\right|_{\partial\Omega}
        \end{equation}
        and can be computed from the DN map $\Lambda_{\rho,\gamma,q}^{red. 1}$. The approach  to establish the unique determination result can be briefly summarized as follows:
		\begin{enumerate}[(i)]
			\item\label{item: asymptotic 1} The Neumann data $\Lambda_{\rho,\gamma,q}^{red. 2}g$ converges (up to a constant) to $\Lambda_{\gamma}g=\left.\gamma\partial_{\nu}V_0\right|_{\partial\Omega}$ as $h\to\infty$, where $\Lambda_{\gamma}$ is the DN map of the conductivity equation and $V_0$ is defined via the expansion $$V(x)=h T^{1+\alpha+m}\frac{\Gamma(1+\alpha)\Gamma(1+m)}{\Gamma(2+\alpha+m)}V_0(x)+R_1(x).$$ 
			
			\item\label{item: asymptotic 2} In addition, making use of the more delicate asymptotic ansatz 
			$$
			V(x)=h V_0(x)+h^{1/m}V_t(x)+R_2(x),
			$$ 
			for suitable functions $V_t$ and $R_2$, one can see that the Neumann data $\Lambda_{\rho,\gamma,q}^{red. 2}g$ has the asymptotic expansion
			\begin{equation}
				\label{eq: asymptotic expansion 2}
				\begin{split}
					\Lambda_{\rho,\gamma,q}^{red. 2}g&=T^{1+\alpha+m}\frac{\Gamma(1+\alpha)\Gamma(1+m)}{\Gamma(2+\alpha+m)}V_0(x) \\
					&\quad +\left. h^{1/m-1}\gamma\partial_{\nu}V_2\right|_{\partial\Omega}+\mathcal{O}(h^{1/m^2-1})
				\end{split}
			\end{equation}
			as $h\to\infty$.
			
			\item Finally, if $\left\{\LC \rho_j,\gamma_j,q_j\RC\right\}_{j=1,2}$ is a triple of coefficients such that the related DN maps $\Lambda_{\rho_j,\gamma_j,q_j}$ coincide, then the property \ref{item: asymptotic 1} yields that $\gamma_1=\gamma_2$ and \ref{item: asymptotic 2} that $\rho_1=\rho_2$, $q_1=q_2$ in $\Omega$.
		\end{enumerate}
        In fact, based on the above ideas, we also prove Theorem \ref{thm: main} in Section \ref{sec: inverse}.
		
		\subsection{Structure of the paper}	The paper is organized as follows. In Section \ref{sec: Preliminaries}, we review several function spaces and integro-differential operators. In Section \ref{sec: forward problem}, we study the forward problem \eqref{eq: main}. More specifically, we prove the existence of a unique, nonnegative, bounded solution of \eqref{eq: main} under suitable sign and regularity assumptions on the data. Moreover, we demonstrate a novel comparison principle for NPMEs in Section \ref{sec: comparison principle}, which will be utilized to establish Theorem \ref{thm: main}. 
		In Section \ref{sec: DN maps and measurment equivalent operators}, we define the DN map rigorously, and provide some examples of measurement equivalent operators. 
		Our main result Theorem \ref{thm: main} will be proved in Section \ref{sec: inverse}. Finally, some compactness and density properties will be addressed in Appendix \ref{sec: appendix compact} and \ref{sec: appendix density}.
		
		\section{Preliminaries}
		\label{sec: Preliminaries}
		
		Let us emphasize that $\Omega\subset \R^n$ always denotes an open set and $\Omega_e\vcentcolon = \R^n\setminus\overline{\Omega}$ is called the exterior of $\Omega$ throughout this article. 
		
		\subsection{Sobolev and H\"older spaces}
		
		The classical Lebesgue and Sobolev spaces are denoted by $L^p(\Omega)$ and $W^{k,p}(\Omega)$, respectively. As usual we will write $H^k(\Omega)$ for $W^{k,p}(\Omega)$, whenever $p=2$.
		
		Next, we recall the definition of fractional Sobolev spaces. For $1\leq p < \infty$ and $0<s<1$, the \emph{fractional Sobolev space} $W^{s,p}(\Omega)$ is defined by
		\[
		W^{s,p}(\Omega)\vcentcolon =\left\{\,u\in L^p(\Omega)\,;\, [u]_{W^{s,p}(\Omega)}<\infty\, \right\}.
		\]
		which are naturally endowed with the norm
		\[
		\|u\|_{W^{s,p}(\Omega)}\vcentcolon =\left(\|u\|_{L^p(\Omega)}^p+[u]_{W^{s,p}(\Omega)}^p\right)^{1/p},
		\]
        where 
        \[
            [u]_{W^{s,p}(\Omega)}\vcentcolon =\left(\int_{\Omega}\int_{\Omega}\frac{|u(x)-u(y)|^p}{|x-y|^{n+s p}}\,dxdy\right)^{1/p}.
        \]
		Next, we introduce a closed subspace of $W^{s,p}(\R^n)$, which can be regarded as consisting of functions with zero exterior value, namely:
		\[
		\widetilde{W}^{s,p}(\Omega)\vcentcolon =\text{closure of }C_c^{\infty}(\Omega) \text{ with respect to } \|\cdot\|_{W^{s,p}(\R^n)}.
		\]
        Let us note that like the classical Sobolev spaces, the spaces $W^{s,p}(\R^n)$ or $\widetilde{W}^{s,p}(\Omega)$ are separable for $1\leq p<\infty$ and reflexive for $1<p<\infty$ (see~\cite[Section~7]{SobolevSpacesCompact}). 
        We remark that it is known that $\widetilde{W}^{s,p}(\Omega)$ coincides with the set of all functions $u\in W^{s,p}(\R^n)$ such that $u=0$ a.e. in $\Omega^c$, when $\partial\Omega\in C^0$, and with 
		\[
		W^{s,p}_0(\Omega)\vcentcolon =\text{closure of }C_c^{\infty}(\Omega) \text{ with respect to }\|\cdot\|_{W^{s,p}(\Omega)},
		\]
		whenever $\Omega\Subset\R^n$ has a Lipschitz boundary (see~\cite[Section~2]{WienerCriterion}).\\
		
		An important advantage of the spaces $\widetilde{W}^{s,p}(\Omega)$ is that they support a \emph{fractional Poincar\'e inequality}, that is, for any bounded set $\Omega\subset\R^n$, $0<s<1$ and $1<p<\infty$ there holds
		\begin{equation}
			\label{eq: Poincare int1}
			\|u\|_{L^p(\R^n)}\leq C[u]_{W^{s,p}(\R^n)}
		\end{equation}
		for all $u\in\widetilde{W}^{s,p}(\Omega)$. 
		
		As for the classical Sobolev spaces, we write $H^s(\R^n)$ and $\widetilde{H}^s(\Omega)$ instead of $W^{s,p}(\R^n)$ and $\widetilde{W}^{s,p}(\Omega)$, when $p=2$. Finally, let us point out that we denote the dual space of $\widetilde{H}^s(\Omega)$ by $H^{-s}(\Omega)$. This notation is justified by the fact that if one defines the latter via Fourier analytic methods, that is as a Bessel potential space, then they are isomorphic. Throughout the article, we will denote the duality pairing between $\widetilde{H}^s(\Omega)$ and $H^{-s}(\Omega)$ by $\langle\cdot, \cdot \rangle$.
		
		
		Now, we introduce the used notation for H\"older continuous functions. For all $0<\alpha\leq 1$, the space $C^{0,\alpha}(\Omega)$ consists of all continuous functions $u\in C(\Omega)$ such that the norm
		\[
		\|u\|_{C^{0,\alpha}(\Omega)}\vcentcolon = \|u\|_{L^{\infty}(\Omega)}+[u]_{C^{0,\alpha}(\Omega)}
		\]
		is finite, where
		\[
		[u]_{C^{0,\alpha}(\Omega)}\vcentcolon =\sup_{x\neq y\in\Omega}\frac{|u(x)-u(y)|}{|x-y|^{\alpha}}.
		\]  
		
		Finally, we recall some standard function spaces for time-dependent PDEs. If $X$ is a given Banach space and $(a,b)\subset\R$, then we write $L^p(a,b\,;X)$ ($1\leq p<\infty$) for the space of measurable functions $u\colon (a,b)\to X$ such that 
		\begin{equation}
			\label{eq: Bochner spaces}
			\|u\|_{L^p(a,b\,;X)}\vcentcolon = \left(\int_a^b\|u(t)\|_{X}^p\,dt\right)^{1/p}<\infty
		\end{equation} 
		and $L^{\infty}(a,b\,;X)$ for the space of measurable functions $u\colon (a,b)\to X$ such that 
		\begin{equation}
			\label{eq: infty spaces}
			\|u\|_{L^{\infty}(a,b\,;X)}\vcentcolon = \inf\{\,M\,;\,\|u(t)\|_X\leq M\,\text{a.e.}\,\}<\infty.
		\end{equation}
        Additionally to these Bochner--Lebesgue spaces, we will also make use of the Sobolev spaces $H^1(0,T;X)$ and generalized Sobolev spaces
        \[
            W_T(X,Y)=\left\{u\in L^2(0,T;X)\,;\,\partial_t u\in L^2(0,T;Y)\,\right\},
        \]
        where $X,Y$ are Banach spaces such that $X\hookrightarrow Y$, $T>0$ and $\partial_t$ denotes the distributional time derivative. These spaces carry the natural norms 
        \begin{equation}
            \|u\|_{H^1(0,T;X)}=\LC \|u\|_{L^2(0,T;X)}^2+\|\partial_t u\|_{L^2(0,T;X)}^2\RC^{1/2}
        \end{equation}
        and 
        \begin{equation}
            \|u\|_{W_T(X,Y)}=\LC \|u\|_{L^2(0,T;X)}^2+\|\partial_t u\|_{L^2(0,T;Y)}^2\RC^{1/2}.
        \end{equation}
        For more details, we refer to the monograph \cite{DautrayLionsVol5} or the article \cite{Simon}.


		\subsection{Nonlocal operators}
        \label{subsec: nonlocal operators}
		
		As explained in Section~\ref{sec: introduction}, in this article we consider uniformly elliptic integro-differential operators of order $2s$. That is, we assume our operators are given by
		\begin{equation}
			\label{eq: strong definition of operators}
			L_K u(x)=\mathrm{P.V.}\int_{\R^{n}}K(x,y)\frac{u(x)-u(y)}{|x-y|^{n+2s}}\,dxdy,
		\end{equation}
		whenever $u\colon \R^n\to\R$ is sufficiently regular and the function $K\colon\R^n\times\R^n\to\R$ satisfies the following properties
		\begin{enumerate}[(i)]
			\item\label{symmetry cond} $K(x,y)=K(y,x)$ for all $x,y\in\R^n$,
			\item\label{ellipticity cond} $\lambda\leq K(x,y)\leq\Lambda$ for some $0<\lambda \leq \Lambda <\infty$.
		\end{enumerate}
		Recall that if the kernel $K$ satisfies the above conditions, then we say $L_K$ belongs to the class $\mathcal{L}_0$. The natural bilinear form related to $L_K$ on $H^s(\R^n)$, will be denoted by $B_K$ and is given by
		\begin{equation}
			\label{eq: bilinear form kernel}
			B_K(u,v)=\int_{\R^{2n}}K(x,y)\frac{(u(x)-u(y))(v(x)-v(y))}{|x-y|^{n+2s}}\,dxdy
		\end{equation}
		for all $u,v\in H^s(\R^n)$. One easily sees that a nonlocal operator $L_K$ in the ellipticity class $\mathcal{L}_0$ defines via \eqref{eq: bilinear form kernel} a bounded linear operator from $\widetilde{H}^s(\Omega)$ to $H^{-s}(\Omega)$.
		
		Finally, let us point out that the following operators belong to the class $\mathcal{L}_0$ (see Section~\ref{sec: introduction} and \cite[eq: (2.15)-(2.16)]{GLX}):
		\begin{enumerate}[(i)]
			\item \emph{Fractional Laplacians} $(-\Delta)^s$,
			\item \emph{Fractional conductivity operators} $\mathcal{L}_{\gamma}$,
			\item \emph{Fractional powers of second order divergence form operators} $L=-\Div(A\nabla \cdot)$, where $A\in C^{\infty}_b(\R^n;\R^{n\times n})$ is symmetric and uniformly elliptic.
		\end{enumerate}

		\section{The forward problem}
		\label{sec: forward problem}

		In this section, we study the forward problem of the NPME. We begin with an auxiliary lemma, in order to prove the well-posedness of \eqref{eq: main}.

		\subsection{Auxiliary lemma}
		\label{sec: Auxiliary lemma}
		
		\begin{definition}
			\label{def: nonlinear function}
			Let $m>0$ be a fixed number, then we define the function $\Phi^m\colon\R\to\R$ by $\Phi^m(t)=|t|^{m-1}t$ for all $t\in\R$.
		\end{definition}
		
		\begin{lemma}
			\label{auxiliary lemma}
			Let $m\geq 1$ be an arbitrary number, then there exists a sequence $\Phi^m_{\vareps}\colon\R\to\R$, $\vareps >0$ sufficiently small, satisfying the following conditions
			\begin{enumerate}[(i)]
				\item\label{cond 1 auxiliary function} $\Phi^m_{\vareps} \in C^{1}(\R)$,
				\item\label{cond 2 auxiliary function} $\Phi_{\vareps}(t)=-\Phi_{\vareps}(-t)$,
				\item\label{cond 3 auxiliary function} $c_{\vareps}\leq (\Phi^m_{\vareps})'\leq C_{\vareps}$ for some constants $0<c_{\vareps}\leq C_{\vareps}<\infty$,
				\item\label{cond 4 auxiliary function} $\Phi_{\vareps}^m\to \Phi^m$ as $\vareps \to 0$ uniformly on any compact set.
			\end{enumerate}
		\end{lemma}
		
		\begin{proof}
			To simplify the notation, let us replace $m$ by $p-1$ and write $\Phi$ instead of $\Phi^{p-1}$. Note that in the simple case $m=1$ or equivalently $p=2$, one can take $\Phi_{\vareps}=\Phi$, which satisfies all conditions \ref{cond 1 auxiliary function}-\ref{cond 4 auxiliary function}. Hence, we can assume without loss of generality that $p>2$. Now choose $M\in 2\N$ satisfying $M>p\geq 2$. Next, we introduce the strictly positive coefficients
			\begin{equation}
				\label{eq: coefficients}
				a_{\vareps}=\frac{p-2}{M-2}\vareps^{p-M},\quad b_{\vareps}=\frac{M-p}{M-2}\vareps^{p-2}
			\end{equation}
			and define the family $\Phi_{\vareps}\colon\R\to\R$ by
			\[
			\Phi_{\vareps}(t)=\begin{cases}
				\Phi(t),\, &\text{for}\quad \vareps \leq |t|\leq 1/\vareps,\\
				a_{\vareps}t^{M-1}+b_{\vareps}t,\, &\text{for}\quad -\vareps \leq |t|\leq \vareps,\\
				(p-1)\vareps^{2-p}t+(2-p)\vareps^{1-p},\, &\text{for}\quad  t\geq 1/\vareps,\\
				(p-1)\vareps^{2-p}t-(2-p)\vareps^{1-p},\, &\text{for}\quad  t\leq -1/\vareps
			\end{cases}
			\]
			for $t\in\R$. We directly see that these functions satisfy \ref{cond 2 auxiliary function}.

			Next, we show the condition \ref{cond 1 auxiliary function}. Clearly, $\Phi_{\vareps}$ is a continuous function. Note that
			\begin{equation}
				\label{eq: derivative of nonlinearity}
				\Phi'(t)=(p-1)|t|^{p-2}\quad\text{for}\quad t\neq 0
			\end{equation}
			and hence $\Phi_{\vareps}'$ exists and is continuous for $|t|\geq \vareps$. Since $\left.\Phi_{\vareps}\right|_{[-\vareps,\vareps]}$ is clearly $C^1$ it remains to check that its derivative coincides with $\Phi_{\vareps}'$ at the endpoints. Let us define $\Psi\colon [-\vareps,\vareps]\to\R$ by
			\begin{equation}
				\label{eq: help function}
				\Psi(t)=\frac{d}{dt}\LC a_{\vareps}t^{M-1}+b_{\vareps}t \RC  =(M-1)a_{\vareps}t^{M-2}+b_{\vareps}, \text{ for } |t|\leq\vareps.
			\end{equation}
			Since $\Psi$ is symmetric it is enough to show that $\Psi(\vareps)=\Phi'(\vareps)$. We have
			\begin{equation}
				\begin{split}
					\Psi(\vareps)&=(M-1)a_{\vareps}\vareps^{M-2}+b_{\vareps}\\
					&=(M-1)\frac{p-2}{M-2}\vareps^{p-M}\vareps^{M-2}+\frac{M-p}{M-2}\vareps^{p-2}\\
					&=\frac{(Mp-2M-p+2)+(M-p)}{M-2}\vareps^{p-2}\\
					&=\frac{Mp-M-2p+2}{M-2}\vareps^{p-2}\\
					&=\frac{(M-2)p-(M-2)}{M-2}\vareps^{p-2}\\
					&=(p-1)\vareps^{p-2}\\
					&=\Phi'(\vareps).
				\end{split}
			\end{equation}
			This establishes the condition
			\ref{cond 1 auxiliary function}. The condition \ref{cond 3 auxiliary function} is easily seen from \eqref{eq: derivative of nonlinearity}, \eqref{eq: help function} and the fact that $\Phi_{\vareps}$ is for $|t|\geq 1/\vareps$ an affine function with strictly positive slope.

			Finally, let us prove the assertion \ref{cond 4 auxiliary function}. By the symmetry condition \ref{cond 2 auxiliary function}, which is also satisfied by $\Phi$, it is enough to prove that for any $K>0$ we have
			\[
			\left\|\Phi_{\vareps}-\Phi\right\|_{L^{\infty}([0,K])}\to 0\quad\text{as}\quad \vareps\to 0.
			\]
			For this purpose, let us fix $K>0$. Without loss of generality we can assume that $\vareps>0$ is sufficiently small that $K\leq 1/\vareps$. For any $0\leq t\leq \vareps$, we have
			\[
			\begin{split}
				\left|\Phi_{\vareps}(t)-\Phi(t)\right|&=\left|\frac{p-2}{M-2}\vareps^{p-M}t^{M-1}+\frac{M-p}{M-2}\vareps^{p-2}t-|t|^{p-2}t\right|\\
				&=\left|\left(\frac{p-2}{M-2}\vareps^{p-M}t^{M-2}-\frac{p-2}{M-2}\vareps^{p-2}\right)t+(\vareps^{p-2}-|t|^{p-2})t\right|\\
				&\leq\left(\frac{p-2}{M-2}(\vareps^{p-2}-\vareps^{p-M}t^{M-2})+(\vareps^{p-2}-t^{p-2})\right)\vareps\\
				&\overset{(\ast)}{\leq} \left(\frac{p-2}{M-2}\vareps^{p-2}+\vareps^{p-2}\right)\vareps\\
				&=\frac{M+p-4}{M-2}\vareps^{p-1}.        \end{split}
			\]
			In $(\ast)$ we used that both terms are decreasing in $t$ on $[0,\vareps]$. Hence, taking $\left.\Phi_{\vareps}\right|_{[\vareps,1/\vareps]}=\Phi$ into account, we have
			\[
			\left\|\Phi_{\vareps}-\Phi\right\|_{L^{\infty}([0,K])}\leq \frac{M+p-4}{M-2}\vareps^{p-1}, 
			\]
			which shows the claimed convergence.
		\end{proof}
		
		\subsection{Existence of solutions to the forward problem}
		\label{subsec: well-posedness}
		
		Let us begin with the definition of weak solutions.
		
		\begin{definition}[Weak solutions]
			\label{def: weak solutions}
			Let $\Omega\subset\R^n$ be a bounded Lipschitz domain, $T>0$, $0<s<1$, $m\geq 1$ and $L_K\in \mathcal{L}_0$. For given $u_0\in L^{\infty}(\Omega)$, $\varphi\in C_c([0,T]\times \Omega_e)$ with $\Phi^m(\varphi)\in L^2(0,T;H^s(\R^n))$ and $f\in L^2(0,T;H^{-s}(\Omega))$, we say that $u\colon\R^n_T\to\R$ is a \emph{(weak) solution} of
			\begin{equation}
				\label{eq: FPME without absorption}
				\begin{cases}
					\partial_t u+ L_K(\Phi^m(u))=f&\text{ in }\Omega_T,\\
					u=\varphi&\text{ in }(\Omega_e)_T,\\
					u(0)=u_0&\text{ in }\Omega,
				\end{cases}
			\end{equation}
			provided that  
			\begin{enumerate}[(i)]
				\item $\Phi^m(u)\in L^2(0,T;H^s(\R^n))$,
				\item $\Phi^m(u-\varphi)\in L^2(0,T;\widetilde{H}^s(\Omega))$,
				\item $u$ satisfies \eqref{eq: FPME without absorption} in the sense of distributions, that is there holds
				\begin{equation}
					-\int_{\Omega_T} u\partial_t\psi\,dxdt + \int_0^T B_K(\Phi^m(u),\psi)\,dt=\int_0^T\langle f,\psi\rangle\,dxdt+\int_{\Omega}u_0\psi(0)\,dx
				\end{equation}
				for all $\psi\in C_c^{\infty}([0,T)\times \Omega)$. 
			\end{enumerate}
			If a (weak) solutions $u$ of \eqref{eq: FPME without absorption} additionally satisfies $u\in L^{\infty}(\Omega_T)$, then it is called \emph{bounded solution}.
		\end{definition}
		\begin{remark}
			Let $u$ be a (weak) solution of \eqref{eq: FPME without absorption}. Observe that the preceding definition directly implies $\partial_t u\in L^2(0,T;H^{-s}(\Omega))$ and thus $u$ belongs to the same space. If we have an additional nonlinear absorption term this does not need to hold.
		\end{remark}
		
		\begin{theorem}[Basic existence result]
			\label{thm: basic existence}
			Let $\Omega\subset\R^n$ be a bounded Lipschitz domain, $T>0$, $0<s<1$, $m> 1$ and $L_K\in \mathcal{L}_0$. For any $0\leq u_0\in L^{\infty}(\Omega)\cap \widetilde{H}^s(\Omega)$ and $0\leq \varphi\in C_c([0,T]\times \Omega_e)$ with $\Phi^m(\varphi)\in L^2(0,T;H^s(\R^n))$, there exists a non-negative, bounded solution of
			\begin{equation}
				\label{eq: homogeneous FPME}
				\begin{cases}
					\partial_t u+ L_K(\Phi^m(u))=0 &\text{ in }\Omega_T,\\
					u=\varphi &\text{ in }(\Omega_e)_T,\\
					u(0)=u_0 & \text{ in }\Omega.
				\end{cases}
			\end{equation}
			Moreover, there holds
			\begin{equation}
				\label{eq: bound on solution}
				\|u\|_{L^{\infty}(\Omega_T)}\leq \sup_{x\in \Omega}u_0+\sup_{(x,t)\in [0,T]\times\Omega_e}\varphi.
			\end{equation}
		\end{theorem}
		
		\begin{proof}
			Throughout the proof we will write $\Phi, \Phi_{\vareps}, L,B$ instead of $\Phi^m,\Phi^m_{\vareps},L_K,B_K$, where $\Phi_{\vareps}^m$ is the family of functions constructed in Lemma~\ref{auxiliary lemma}. 
			
			We prove existence of a solution $u$ by the Galerkin method. For this purpose, let us first observe that since $\varphi$ is compactly supported in $[0,T]\times\Omega_e$, $u$ is a solution of \eqref{eq: homogeneous FPME} if and only if $v=u-\varphi$ solves
			\begin{equation}
				\label{eq: homogeneous FPME without exterior cond}
				\begin{cases}
					\partial_t v+ L(\Phi(v))=f &\text{ in }\Omega_T,\\
					v=0&\text{ in }(\Omega_e)_T,\\
					v(0)=u_0 &\text{ in }\Omega,
				\end{cases}
			\end{equation}
			with $f=\p_t \varphi -L(\Phi(\varphi))=-L(\Phi(\varphi))$ in $\Omega_T$ by using $\varphi\equiv 0$ in $\Omega_T$. Note that for any $(x,t)\in \Omega_T$, we have
			\begin{equation}
				\label{eq: equivalent source}
				\begin{split}
					f(x,t)&=-\int_{\R^n}K(x,y)\frac{\Phi(\varphi)(x,t)-\Phi(\varphi)(y,t)}{|x-y|^{n+2s}}dy=\int_{\Omega_e}K(x,y)\frac{\Phi(\varphi)(y,t)}{|x-y|^{n+2s}}\,dy,
				\end{split}
			\end{equation}
			where we used that $\Phi(0)=0$ and $\varphi|_{\Omega_T}=0$. The compact support of $\varphi$ in $[0,T]\times\Omega_e$, $\Phi(0)=0$, the uniform ellipticity of $K$, $\Phi(t)\geq 0$ for $t\geq 0$ and $\varphi\geq 0$ clearly implies that $f$ has the following properties:
			\begin{enumerate}[(a)]
				\item\label{cond 1 on f} $0\leq f(x,t)\leq C\int_{\Omega_e}\Phi(\varphi)(y,t)\,dy<\infty$ for $(x,t)\in\Omega_T$,
				\item\label{cond 2 on f} $f\in C([0,T];L^2(\Omega))$.
			\end{enumerate}
			
			Next recall that $\widetilde{H}^s(\Omega)$ is a separable Hilbert space and hence the finite dimensional subspaces 
			\[
			E _m =\mathrm{span}\left\{w_1,\ldots,w_m\right\},\quad m\in\N,
			\]
			form a Galerkin approximation for $\widetilde{H}^s(\Omega)$. Here $\LC w_j\RC_{j\in\N}\subset \widetilde{H}^s(\Omega)$ is a priori any orthogonal basis of $\widetilde{H}^s(\Omega)$. Since $\Omega$ is bounded, we will take in the following $\LC  w_j\RC _{j\in\N}$ to be the eigenfunctions to the positive discrete eigenvalues $\LC \lambda_j\RC _{j\in\N}$ of the positive, symmetric, nonlocal operator $L$. This is possible as the embedding $\widetilde{H}^s(\Omega)\hookrightarrow L^2(\Omega)$ is by the Rellich--Kondrachov theorem compact (see~\cite[Section~6.5, Theorem~1]{EvansPDE}, \cite[Chapter~XVIII, Section~2.1-2.2]{DautrayLionsVol5} and \cite[Theorem~3.2]{KRZ-2023}). Observe by density of $\widetilde{H}^s(\Omega)$ in $L^2(\Omega)$ the family $\LC E_m\RC_{m\in\N}$ is also a Galerkin approximation for $L^2(\Omega)$. Moreover, we will assume in the following that $\LC w_j\RC_{j\in\N}$ are normalized in $L^2(\Omega)$.

			Let us note that if 
			\begin{align}\label{psi_N}
				\psi_N=\sum_{j=1}^N \left\langle \psi,w_j \right\rangle_{L^2}w_j,
			\end{align}
			for some $\psi\in\widetilde{H}^s(\Omega)$, then we have 
			\begin{equation}
				\label{eq: compatibility of approx}
				\left[\psi_N\right]_{H^{s}(\R^n)}\leq C [\psi]_{H^s(\R^n)},
			\end{equation}
			for some constant $C>0$ independent of $\psi$ and $N \in \N$.
			Here and in the following $\langle\cdot,\cdot\rangle_{L^2}$ always denotes the inner product in $L^2(\Omega)$. In fact, by assumption we have
			\begin{align}\label{bilinear eigenexpand}
				B(w_j,\psi)=\lambda_j\left\langle w_j,\psi\right\rangle_{L^2}
			\end{align}
			for all $j\in\N$. Thus, \eqref{psi_N} and \eqref{bilinear eigenexpand} yield that 
			\[
			\psi_N= \sum_{j=1}^{N}\left\langle \psi,w_j\right\rangle_{L^2}w_j=\sum_{j=1}^N B(\psi,w_j/\sqrt{\lambda_j})(w_j/\sqrt{\lambda_j})
			\]
			and $w_j/\sqrt{\lambda_j}$ for $j\in\N$, is an orthonormal basis of $\widetilde{H}^s(\Omega)$ with equivalent inner product induced by $B$. The last part follows from the symmetry and uniform ellipticity of $K$ and the fractional Poincar\'e inequality. But then by Parseval's identity we have
			\[
			B(\psi,\psi)=\sum_{j\in\N}\left| B(\psi,w_j/\sqrt{\lambda_j})\right|^2\geq \sum_{j=1}^N \left| B(\psi,w_j/\sqrt{\lambda_j})\right|^2=B(\psi_N,\psi_N).
			\]
			Now, the uniform ellipticity of $K$ gives \eqref{eq: compatibility of approx}. The rest of the proof is divided into three steps. \\

			\noindent \textit{Step 1: Construction of approximate solutions.} \\

			Let $\vareps>0$, $v_0\in L^2(\Omega)$, $F\in L^2(0,T;H^{-s}(\Omega))$ and consider the initial-exterior value problem
			\begin{equation}
				\label{eq: approximate problem}
				\begin{cases}
					\partial_t v+ L(\Phi_{\vareps}(v))=F &\text{ in }\Omega_T,\\
					v=0&\text{ in }(\Omega_e)_T,\\
					v(0)=v_0&\text{ in }\Omega.
				\end{cases}
			\end{equation}
			Next, choose a sequence $F_N\in C([0,T];H^{-s}(\Omega))$, $N\in\N$,  converging to $F$ in $L^2(0,T;H^{-s}(\Omega))$. This can be achieved by setting $F(t)=0$ for $t\notin [0,T]$ and then mollify in time.  As usual in the Galerkin approximation, in a first step we will be looking for functions $v_{\vareps,N}\in C^1([0,T],E_N)$, that is they can be written as
			\[
			v_{\vareps,N}(x,t)=\sum_{j=1}^N c^j_{N,\vareps}(t)w_j(x)
			\]
			for $N\in\N$, and solve in $E_N$ the problem
			\begin{equation}
				\label{eq: second approximate problem}
				\begin{cases}
					\partial_t v+ L(\Phi_{\vareps}(v))=F_N &\text{ in }\Omega_T,\\
					v=0&\text{ in }(\Omega_e)_T,\\
					v(0)=v_{0,N}&\text{ in }\Omega,
				\end{cases}
			\end{equation}
			where $v_{0,N}=\sum_{j=1}^N \left\langle v_0,w_j\right\rangle_{L^2}w_j$. If $v_{\vareps,N}\in C^1([0,T];E_N)$ solves \eqref{eq: second approximate problem}, then we have 
			\begin{equation}
				\label{eq: weak formulation approximate solutions}
				\left\langle \partial_t v_{\vareps,N},\psi \right\rangle_{L^2}+B(\Phi_{\vareps}(v_{\vareps,N}),\psi)=\left\langle F_N, \psi \right\rangle
			\end{equation}
			for a.e. $t\in (0,T)$ and $\psi\in E_N$. In particular, choosing $\psi=w_j$ for $1\leq j\leq N$ as a test function, we get
			\begin{equation}
				\label{eq: approx sol 2}
				\left\langle \partial_t v_{\vareps,N},w_j \right\rangle_{L^2}+B(\Phi_{\vareps}(v_{\vareps,N}),w_j)=\left\langle F_N, w_j \right\rangle
			\end{equation}
			for all $1\leq j\leq N$ and $0<t<T$. Expanding everything, we see that $$\hat{c}_{\vareps,N}:=\LC c^1_{\vareps,N},\ldots,c^N_{\vareps,N}\RC \in C^1([0,T];\R^N)$$ is a solution of the ordinary differential equation (ODE)
			\begin{equation}
				\label{eq: IVP for coeff}
				\begin{cases}
					\partial_t \hat{c}_{\vareps,N}+ b(\hat{c}_{\vareps,N})=G_N(\hat{c}_{\vareps,N}),\quad\text{for}\quad 0<t<T\\
					\hat{c}_{\vareps,N}(0)=\hat{c}_{0,N},
				\end{cases}
			\end{equation}
			where 
			\begin{equation}
				\label{eq: initial condition}
				\hat{c}_{0,N}=\left( \left\langle v_0,w_1\right\rangle_{L^2},\ldots,\left\langle v_0,w_N\right\rangle_{L^2}\right) ,\quad G_N(\hat{c}_{\vareps,N})=\left\langle F_N, \sum_{j=1}^N c_{\vareps,N}^jw_j\right\rangle
			\end{equation}
			and for $1\leq j\leq N$ the components $b_j\colon\R^N\to\R$ of $b=(b_1,\ldots,b_N)$ are defined by 
			\begin{equation}
				\label{eq: def coeff b}
				\begin{split}
					b_j(c)=B\left(\Phi_{\vareps}\left(\sum_{k=1}^Nc_k w_k\right),w_j\right)
				\end{split}
			\end{equation}
			for $c=(c_1,\ldots,c_N)\in\R^N$.
			
			\begin{claim}
				\label{claim: continuity of coefficients}
				The functions $G_N\colon\R^N\to\R$ and $b\colon\R^N\to\R^N$ are continuous.
			\end{claim}
			\begin{proof}[Proof of Claim \ref{claim: continuity of coefficients}]
				The continuity of $G_N$ is immediate since $F_N\in C([0,T],H^{-s}(\Omega))$. To see the continuity of $b$, consider a sequence $\LC \hat{c}_k\RC _{k\in \N}=\LC \hat{c}_k^1,\ldots,\hat{c}_k^N\RC_{k\in\N}\subset\R^N$, which converges to $\hat{c}=\LC \hat{c}^1,\ldots,\hat{c}^N\RC \in\R^N$ as $k\to \infty$. Note that by uniform ellipticity of $K$ and H\"older's inequality we have
				\begin{align}\label{ineq in claim conti coeff}
					\begin{split}
						\left|b(\hat{c}_k)-b(\hat{c})\right|&\leq \max_{1\leq j\leq N}\left|B\left(\Phi_{\vareps}\left(\sum_{\ell=1}^N \hat{c}^{\ell}_k w_{\ell}\right)-\Phi_{\vareps}\left(\sum_{\ell=1}^N \hat{c}^{\ell} w_{\ell}\right),w_j\right)\right|\\
						&\leq C     \left[\Phi_{\vareps}\left(\sum_{\ell=1}^N \hat{c}^{\ell}_k w_{\ell}\right)-\Phi_{\vareps}\left(\sum_{\ell=1}^N \hat{c}^{\ell} w_{\ell}\right)\right]_{H^s(\R^n)}
					\end{split}
				\end{align}
				Using the Lipschitz continuity of $\Phi_{\vareps}$ and Lebesgue's dominated convergence theorem, one easily sees that the last expression in \eqref{ineq in claim conti coeff} goes to zero as $k\to\infty$ and hence showing the continuity of $b$. This proves the claim.
			\end{proof}
			
			Hence, by Peano's existence theorem for ODEs there exist a solution $\hat{c}_{N,\vareps}\in C^1([0,\delta];\R^N)$ of \eqref{eq: IVP for coeff} for possibly a small $\delta>0$. Later, we show that the solution actually extends to $[0,T]$. Let us denote the corresponding approximate solution of \eqref{eq: second approximate problem} by $v_{\vareps,N}\in C^1([0,\delta],E_N)$. Multiplying \eqref{eq: approx sol 2} by $c_{N,\vareps}^j$ and summing from $j=1$ to $N$ gives
			\begin{equation}
				\label{eq: energy identity approx sol}
				\left\langle \partial_t v_{\vareps,N},v_{\vareps,N}\right\rangle_{L^2}+B(\Phi_{\vareps}(v_{\vareps,N}),v_{\vareps,N})=\left\langle F_N,v_{\vareps,N}\right\rangle.
			\end{equation}
			We have
			\begin{equation}
				\label{eq: energy identity approx sol 2}
				\begin{split}
					&\left\langle \partial_t v_{\vareps,N},v_{\vareps,N}\right\rangle_{L^2}=\frac{d}{dt}\frac{\|v_{\vareps,N}\|_{L^2}^2}{2},\\
					&\left|\langle F_N,v_{\vareps,N}\rangle\right| \leq \|F_N\|_{H^{-s}(\Omega)}\|v_{\vareps,N}\|_{H^s(\R^n)}.
				\end{split}
			\end{equation}
			Moreover, using the fundamental theorem of calculus, the uniform ellipticity of $\Phi'_{\vareps}$ and $K$ the second term in \eqref{eq: energy identity approx sol} can be lower bounded as (suppressing the $t$ dependence)
			\begin{equation}
				\label{eq: energy identity approx sol 3}
				\begin{split}
					&\quad B(\Phi_{\vareps}(v_{\vareps,N}),v_{\vareps,N})\\
					&=\int_{\R^{2n}}K(x,y)\frac{(\Phi_{\vareps}(v_{\vareps,N})(x)-\Phi_{\vareps}(v_{\vareps,N})(y))(v_{\vareps,N}(x)-v_{\vareps,N}(y))}{|x-y|^{n+2s}}\,dxdy\\
					&=\int_{\R^{2n}}K(x,y)\frac{ \int_0^1 \frac{d}{d\tau} \left[ \Phi_{\vareps}(v_{\vareps,N}(y)+\tau(v_{\vareps,N}(x)-v_{\vareps,N}(y)))\right] d\tau }{|x-y|^{n+2s}} \\
					&\qquad \qquad \cdot \LC v_{\vareps,N}(x)-v_{\vareps,N}(y)\RC  \,dxdy\\
					&=\int_{\R^{2n}}K(x,y)\frac{\int_0^1 \Phi'_{\vareps}(v_{\vareps,N}(y)+\tau(v_{\vareps,N}(x)-v_{\vareps,N}(y)))\,d\tau}{|x-y|^{n+2s}} \\
					&\qquad \qquad \cdot 
					\LC  v_{\vareps,N}(x)-v_{\vareps,N}(y)\RC ^2\,dxdy\\
					&\geq C\left[v_{\vareps,N}\right]^2_{H^s(\R^n)},
				\end{split}
			\end{equation}
			where the constant $C>0$ is independent of $N$ and $t$ but depends generally on $\vareps$. Thus, by the fractional Poincar\'e inequality we have
			\begin{equation}
				\label{eq: energy identity approx sol 4}
				B(\Phi_{\vareps}(v_{\vareps,N}),v_{\vareps,N})\geq C\left\|v_{\vareps,N}\right\|_{H^s(\R^n)}^2.
			\end{equation}
			Thus, an integration of \eqref{eq: energy identity approx sol} over $[0,T_0]$ with $0<T_0\leq \delta$ and using \eqref{eq: energy identity approx sol 2} and \eqref{eq: energy identity approx sol 4}, we obtain
			\begin{equation}
				\label{eq: energy identity approx sol 5}
				\begin{split}
					&\quad \left\|v_{\vareps,N}(T_0)\right\|_{L^2(\Omega)}^2+\left\|v_{\vareps,N}\right\|_{L^2(0,T_0;H^s(\R^n))}^2\\
					& \leq C\LC \left\|F_N \right\|_{L^2(0,T_0;H^{-s}(\Omega))}\left\|v_{\vareps,N}\right\|_{L^2(0,T_0;H^s(\R^n))}+\left\|v_{0,N} \right\|_{L^2(\Omega)}^2\RC, 
				\end{split}
			\end{equation}
			for some constant $C>0$ independent of $N$ and $0<T_0\leq \delta$. By applying Young's inequality we can absorb the factor $\left\|v_{\vareps,N}\right\|_{L^2(0,T_0;H^s(\R^n))}$, appearing on the right of \eqref{eq: energy identity approx sol 5}, on the left hand side to obtain
			\begin{equation}
				\label{eq: energy identity approx sol 5.1}
				\begin{split}
					&\quad \left\|v_{\vareps,N}(T_0)\right\|_{L^2(\Omega)}^2+ \left\|v_{\vareps,N}\right\|_{L^2(0,T_0;H^s(\R^n))}^2 \\
     &\leq  C\LC \left\|F_N \right\|^2_{L^2(0,T_0;H^{-s}(\Omega))}+ \left\|v_{0,N} \right\|_{L^2(\Omega)}^2\RC,
				\end{split}
			\end{equation}
   for some constant $C>0$ independent of $N$ and $0<T_0\leq \delta$.

			Notice that by the convergence $F_N\to F$ in $L^2(0,T;H^{-s}(\Omega))$, then for $N\in \N$, each term $\|F_N\|_{L^2(0,T_0;H^{-s}(\Omega))}$ is uniformly bounded in $N\in \N$ and $T_0$ and additionally by Parseval's identity, one has 
			\begin{equation}
				\label{eq: Parseval}
				\left\|v_{0,N} \right\|_{L^2(\Omega)}\leq \left\|v_0\right\|_{L^2(\Omega)}.
			\end{equation}
			Thus, after taking the supremum in $T_0\in [0,\delta]$ we get
			\begin{equation}
				\label{eq: energy identity approx sol 6}
				\left\|v_{\vareps,N} \right\|_{L^{\infty}(0,\delta;L^2(\Omega))}+\left\|v_{\vareps,N} \right\|_{L^2(0,\delta;H^s(\R^n))}\leq C
			\end{equation}
			uniformly in $N\in\N$, for some constant $C>0$ independent of $N$. But this means $v_{\vareps,N}$ remains in $L^2(\Omega)$ as $t\to \delta$ and hence we can repeat our local existence result finitely many times to conclude that $v_{\vareps,N}\in C^1([0,T];E_N)$ solves \eqref{eq: second approximate problem} on $[0,T]$. 
			
			Next observe that \eqref{eq: energy identity approx sol 5.1} and \eqref{eq: Parseval} gives us by standard arguments a useful energy estimate for $v_{\vareps,N}$, namely
			\begin{equation}
				\label{eq: energy identity approx sol 7}
				\begin{split}
					&\quad \left\|v_{\vareps,N} \right\|_{L^{\infty}(0,T;L^2(\Omega))}+\left\|v_{\vareps,N}\right\|_{L^2(0,T;H^s(\R^n))}\\
					&\leq C\LC \left\|F_N\right\|_{L^2(0,T\,;H^{-s}(\Omega))}+\left\|v_0 \right\|_{L^2(\Omega)}\RC \\
					&\leq C\LC 1+\left\|F \right\|_{L^2(0,T\,;H^{-s}(\Omega))}+\left\|v_0\right\|_{L^2(\Omega)}\RC 
				\end{split}
			\end{equation}
			uniformly in $N\in\N$. In addition, we want to control $\partial_t v_{\vareps,N}$ in $L^2(0,T;H^{-s}(\Omega))$. Note that by \eqref{eq: second approximate problem}, H\"older's inequality, the uniform ellipticity of $\Phi'_{\vareps}$ and $K$, \eqref{eq: energy identity approx sol 7}, the convergence $F_N\to F$ in $L^2(0,T;H^{-s}(\Omega))$ and \eqref{eq: compatibility of approx}, for any $\psi\in L^2(0,T;\widetilde{H}^s(\Omega))$ we have
			\begin{equation}
				\label{eq: time regularity for each N}
				\begin{split}
					&\quad \left|\int_0^T \left\langle \partial_t v_{\vareps,N}, \psi \right\rangle_{L^2}\,dt\right| \\
					&=\left|\int_0^T \left\langle \partial_t v_{\vareps,N}, \psi_N \right\rangle_{L^2}\,dt\right|\\
					&=\left|-\int_0^T B(\Phi_{\vareps}(v_{\vareps,N}),\psi_N)+ \left\langle F_N,v_N \right\rangle\,dt \right|\\
					&\leq C\left[\Phi_{\vareps}(v_{\vareps,N})\right]_{L^2(0,T;H^s(\R^n))}\left[\psi_N\right]_{L^2(0,T;H^s(\R^n))} \\
					&\quad +\|F_N\|_{L^2(0,T;H^{-s}(\Omega))}\|v_N\|_{L^2(0,T;H^s(\R^n))}\\
					&\leq C \LC \left[v_{\vareps,N}\right]_{L^2(0,T;H^s(\R^n))}+\left\|F_N \right\|_{L^2(0,T;H^{-s}(\Omega))}\RC \left\|\psi_N \right\|_{L^2(0,T;H^s(\R^n))}\\
					&\leq C \left\|\psi_N \right\|_{L^2(0,T;H^s(\R^n))} \\
                    &\leq C \left\|\psi \right\|_{L^2(0,T;H^s(\R^n))},
				\end{split}
			\end{equation}
			for some constant $C$ independent of $N$, where we set $\psi_N=\sum_{j=1}^N\left\langle \psi,w_j\right\rangle_{L^2}w_j$. Finally, using $\langle U,V\rangle_{L^2}=\langle U,V\rangle$ whenever $U,V\in\widetilde{H}^s(\Omega)$, we see that 
			\begin{equation}
				\label{eq: regularity time derivative}
				\partial_t v_{\vareps,N}\in L^2(0,T;H^{-s}(\Omega))\quad\text{with}\quad \left\|\partial_t v_{\vareps,N}\right\|_{L^2(0,T;H^{-s}(\Omega))}\leq C,
			\end{equation}
			for some $C>0$ independent of $N$.\\

			\noindent\textit{Step 2: Passing to the limit $N\to\infty$}. \\
			
			By \eqref{eq: energy identity approx sol 7} and  \eqref{eq: regularity time derivative}, we see that $ \LC v_{\vareps,N}\RC_{N\in\N}\subset X$ is uniformly bounded, where 
			\[
			X \vcentcolon =L^2(0,T;\widetilde{H}^s(\Omega))\cap H^1(0,T;H^{-s}(\Omega)).
			\]
			Since $X$ is a reflexive Banach space, there exists $v_{\vareps}\in X$ such that
			\begin{equation}
				\label{eq: weak convergence in N}
				v_{\vareps,N}\weak v_{\vareps}\ \text{in}\ X\ \text{as}\ N\to\infty
			\end{equation}
			(up to extracting a subsequence). Next recall that by the Rellich--Kondrachov theorem $\widetilde{H}^s(\Omega)\hookrightarrow L^2(\Omega)$ is compact and thus Theorem~\ref{Aubin-Lions lemma} implies that $\LC v_{\vareps,N}\RC_{N\in\N}$ is precompact in $L^2(\Omega_T)$ and hence up to extracting a subsequence we have
			\begin{equation}
				\label{eq: strong convergence in L2 in N}
				v_{\vareps,N}\to v_{\vareps}\ \text{in}\ L^2(\Omega_T)\ \text{and a.e. in}\ \Omega_T\ \text{as}\ N\to\infty.
			\end{equation}
			Next, we claim that:
			\begin{claim}
				\label{claim convergence as N to infty of bilinear forms}
				There holds
				\begin{equation}
					\label{eq: convergence of bilinear form as N to infty}
					\int_0^T B(\Phi_{\vareps}(v_{\vareps,N}),\psi)\,dt\to \int_0^T B(\Phi_{\vareps}(v_{\vareps}),\psi)\,dt \text{ as }N\to \infty,
				\end{equation}
				for all $\psi\in L^2(0,T;\widetilde{H}^s(\Omega))$.
			\end{claim}
			
			\begin{proof}[Proof of Claim \ref{claim convergence as N to infty of bilinear forms}]
				Using a similar method as in \eqref{eq: energy identity approx sol 3}, the fundamental theorem of calculus yields that 
				\[
				\begin{split}
					&\quad \left|\int_0^T B(\Phi_{\vareps}(v_{\vareps,N}),\psi)\,dt- \int_0^T B(\Phi_{\vareps}(v_{\vareps}),\psi)\,dt \right|\\
					&=\left|\int_{\R^{2n}_T}K(x,y)\left(\int_0^1\Phi'_{\vareps}(v_{\vareps,N}(y)+\tau (v_{\vareps,N}(x)-v_{\vareps,N}(y)))\,d\tau\right) \right. \\
					& \qquad \qquad \qquad \cdot \left. \frac{(v_{\vareps,N}(x)-v_{\vareps,N}(y))(\psi(x)-\psi(y))}{|x-y|^{n+2s}}\,dxdydt\right.\\
					&\qquad \quad -\int_{\R^{2n}_T}K(x,y)\left(\int_0^1\Phi'_{\vareps}(v_{\vareps}(y)+\tau (v_{\vareps}(x)-v_{\vareps}(y)))\,d\tau\right)  \\
					&\qquad \qquad \qquad \cdot \left. \frac{(v_{\vareps}(x)-v_{\vareps}(y))(\psi(x)-\psi(y))}{|x-y|^{n+2s}}\,dxdydt\right|\\
					&\leq \left|\int_{\R^{2n}_T}K(x,y)\Psi_{\vareps,N}(x,y)\frac{(v_{\vareps,N}(x)-v_{\vareps,N}(y))(\psi(x)-\psi(y))}{|x-y|^{n+2s}}\,dxdydt\right|\\
					&\qquad +\left|\int_{\R^{2n}_T}K(x,y)\left(\int_0^1\Phi'_{\vareps}(v_{\vareps}(y)+\tau (v_{\vareps}(x)-v_{\vareps}(y)))\,d\tau\right) \right. \\
					&\quad \qquad \qquad  \left. \cdot \frac{((v_{\vareps,N}-v_{\vareps})(x)-(v_{\vareps,N}-v_{\vareps})(y))(\psi(x)-\psi(y))}{|x-y|^{n+2s}}\,dxdydt\right|\\
					&:=I_1^N+I_2^N,
				\end{split}
				\]
				where we have set
				\[
				\begin{split}
					&\Psi_{\vareps,N}(x,y)\\
					&:= \int_0^1(\Phi'_{\vareps}(v_{\vareps,N}(y)+\tau (v_{\vareps,N}(x)-v_{\vareps,N}(y)))-\Phi'_{\vareps}(v_{\vareps}(y)+\tau (v_{\vareps}(x)-v_{\vareps}(y))))\,d\tau.  
				\end{split}
				\]
				It is clear that the operator in $I_2^N$ defines an element of $L^2(0,T;H^{-s}(\Omega))$ for fixed $\psi$ and $v_{\vareps}$. Thus, the weak convergence $v_{\vareps,N}\weak v_{\vareps}$ in $L^2(0,T;\widetilde{H}^s(\Omega))$ (see  \eqref{eq: weak convergence in N}), implies that $I_2^N\to 0$ as $N\to\infty$. Hence, we need only to show that $I_1^N\to 0$ as $N\to\infty$.

                As $\Phi_{\vareps}\in C^1(\R)$,  the convergence \eqref{eq: strong convergence in L2 in N} implies 
				\[
				\Phi'_{\vareps}(v_{\vareps,N}(y)+\tau (v_{\vareps,N}(x)-v_{\vareps,N}(y)))-\Phi'_{\vareps}(v_{\vareps}(y)+\tau (v_{\vareps}(x)-v_{\vareps}(y)))\to 0
				\]
				for a.e. $x,y\in\R^n$ and $t\in (0,T)$ as $N\to\infty$. Hence, the uniform ellipticity of $\Phi'_{\vareps}$ and Lebesgue's dominated convergence theorem implies 
				\begin{equation}
					\label{eq: psi vanishs in limit}
					\Psi_N(x,y)\to 0\ \text{as}\ N\to\infty,
				\end{equation}
				for a.e. $x,y\in\R^n$. Now, let us define
				\[
				W_{\vareps,N}(x,y):=K(x,y)\Psi_{\vareps,N}(x,y)\frac{v_{\vareps,N}(x)-v_{\vareps,N}(y)}{|x-y|^{n/2+s}}\in L^2(\R^{2n}_T).
				\]
				By \eqref{eq: energy identity approx sol 7} and the uniform ellipticity of $K$, this sequence is uniformly bounded in $L^2(\R^{2n}_T)$. Thus up to subsequences, we know that $W_{\vareps,N}\weak W_{\vareps}$ in $L^2(\R^{2n}_T)$ for some $W_{\vareps}\in L^2(\R^{2n}_T)$. Additionally, we know that $W_{\vareps,N}\to 0$ a.e. in $\R^{2n}_T$ by \eqref{eq: psi vanishs in limit}. As a consequence of Egoroff's theorem (see~\cite[Theorem~1.16]{evans-mt}), we deduce that $W_{\vareps}=0$ a.e. in $\R^{2n}_T$. Now, testing the weak convergence $W_{\vareps,N}\weak 0$ in $L^2(\R^{2n}_T)$ with $\frac{\psi(x)-\psi(y)}{|x-y|^{n/2+s}}\in L^2(\R^{2n}_T)$ establishes $I_1^N\to 0$ as $N\to\infty$ and we can conclude the proof of the claim.
			\end{proof} 
			
			Now, we return back to our proof of Theorem \ref{thm: basic existence}. Let $N\geq M$, use $\psi\in C^1([0,T];E_M)$ satisfying $\psi(T)=0$ in \eqref{eq: weak formulation approximate solutions} as a test function and integrate by parts in time the resulting equation to obtain
			\begin{equation}
				\label{eq: relevant equation to pass to the limit in N}
				\begin{split}
					-\int_{\Omega_T}v_{\vareps,N}\partial_t \psi\,dxdt+\int_0^T B(\Phi_{\vareps}(v_{\vareps,N}),\psi)\,dt
					=\int_0^T \langle F_N,\psi\rangle\,dt+\int_{\Omega}v_{0,N}\psi(0)\,dx.
				\end{split}   
			\end{equation}
			By \eqref{eq: weak convergence in N}, \eqref{eq: strong convergence in L2 in N}, Claim~\ref{claim convergence as N to infty of bilinear forms}, $F_N\to F$ in $L^2(0,T;H^{-s}(\Omega))$ and $v_{0,N}\to v_0$ in $L^2(\Omega)$, we can pass to the limit in the above equation \eqref{eq: relevant equation to pass to the limit in N} and get
			\[
			\begin{split}
				-\int_{\Omega_T}v_{\vareps}\partial_t \psi\,dxdt+\int_0^T B(\Phi_{\vareps}(v_{\vareps}),\psi)\,dt
				=\int_0^T \langle F,\psi\rangle\,dt+\int_{\Omega}v_{0}\psi(0)\,dx
			\end{split}
			\]
			for all $\psi\in C^1([0,T];E_M)$ with $\psi(T)=0$ and $M\geq 0$. As $\widetilde{H}^s(\Omega)=\overline{\bigcup_{M\in\N}E_M}$, this actually holds for all $\psi\in C^1([0,T];\widetilde{H}^s(\Omega))$ with $\psi(T)=0$. Hence, in particular it holds for all $\psi \in C_c^{\infty}([0,T)\times \Omega)$. Thus, for any $\vareps>0$, $v_0\in L^2(\Omega)$, $F\in L^2(0,T;H^{-s}(\Omega))$ we have found a solution $v_{\vareps}\in L^2(0,T;\widetilde{H}^s(\Omega))\cap H^1(0,T;H^{-s}(\Omega))$ of 
			\begin{equation}
				\label{eq: solution for every epsilon}
				\begin{cases}
					\partial_t v+ L(\Phi_{\vareps}(v))=F&\text{ in }\Omega_T,\\
					v=0&\text{ in }(\Omega_e)_T,\\
					v(0)=v_0&\text{ in }\Omega.
				\end{cases}
			\end{equation}

            \medskip
   
			\noindent\textit{Step 3: Passing to the limit $\vareps\to 0$}.\\

			Now, let $0\leq u_0\in L^{\infty}(\Omega)\cap \widetilde{H}^s(\Omega)$ and $\varphi\in C_c^{\infty}((\Omega_e)_T)$ or $\varphi\in C_c([0,T]\times\Omega_e)$ with $\Phi(\varphi)\in L^2(0,T;H^s(\R^n))$ as in the assumptions of Theorem~\ref{thm: basic existence}. The same we said at the beginning of the proof for the solvability of \eqref{eq: homogeneous FPME} is true when we replace $\Phi$ by $\Phi_{\vareps}$. More precisely, $u_{\vareps}$ solves 
			\begin{equation}
				\label{eq: original eq for approximate sols}
				\begin{cases}
					\partial_t u+ L(\Phi_{\vareps}(u))=0& \text{ in }\Omega_T,\\
					u=\varphi&\text{ in }(\Omega_e)_T,\\
					u(0)=u_0&\text{ in }\Omega
				\end{cases}
			\end{equation}
			if and only if $v_{\vareps}=u_{\vareps}-\varphi$ solves
			\begin{equation}
				\label{eq: homogeneous problem for approximate problem}
				\begin{cases}
					\partial_t v+ L(\Phi_{\vareps}(v))=f_{\vareps} &\text{ in }\Omega_T,\\
					v=0&\text{ in }(\Omega_e)_T,\\
					v(0)=u_0& \text{ in }\Omega,
				\end{cases}
			\end{equation}
			where $f_{\vareps}=-L(\Phi_{\vareps}(\varphi))$. The functions $f_{\vareps}$ satisfy
			\begin{enumerate}[(A)]
				\item\label{cond 1 on f eps} $0\leq f_{\vareps}(x,t)\leq C\int_{\Omega_e}\Phi_{\vareps}(\varphi)(y,t)\,dy$ for $(x,t)\in\Omega_T$,
				\item\label{cond 2 on f eps} $f_{\vareps}\in C([0,T];L^2(\Omega))$.
			\end{enumerate}
			By \textit{Step 1} and \textit{Step 2}, we know that a solution $v_{\vareps}$ of \eqref{eq: homogeneous problem for approximate problem} exists such that $v_{\vareps}\in L^2(0,T;\widetilde{H}^s(\Omega))\cap H^1(0,T;H^{-s}(\Omega))$. Note that the additional approximation of the inhomogenity is here not needed by \ref{cond 2 on f eps}. Next, we need to establish the following claim:
			\begin{claim}
				\label{claim further regularity of solution v epsilon}
				For any $\vareps>0$, there holds
				\begin{enumerate}[(i)]
					\item\label{time regularity} $\partial_t v_{\vareps}\in L^2(\Omega_T)$,
					\item\label{nonlinear regularity} $\Phi_{\vareps}(v_{\vareps})\in L^2(0,T;\widetilde{H}^s(\Omega))$,
					\item\label{time regularity nonlinearity} $\partial_t \LC \Phi_{\vareps}(v_{\vareps})\RC =\Phi'_{\vareps}(v_{\vareps})\partial_t v_{\vareps}$ a.e. in $\Omega_T$.
				\end{enumerate}
			\end{claim}
			\begin{proof}[Proof of Claim \ref{claim further regularity of solution v epsilon}]
				Since $\Phi_{\vareps}\in C^1(\R)$ and $v_{\vareps,N}\in C^1([0,T];E_N)$, we may compute
				\begin{equation}
					\label{eq: time derivative nonlinearity with N}
					\partial_t \LC \Phi_{\vareps}(v_{\vareps,N})\RC =\Phi'_{\vareps}(v_{\vareps,N})\partial_t v_{\vareps,N}.
				\end{equation}
				Hence, $\partial_t \Phi_{\vareps}(v_{\vareps,N})\in E_N$ for a.e. $t\in (0,T)$ and therefore using this function in \eqref{eq: second approximate problem}, we obtain
				\[
				\begin{split}
					\int_{\Omega}\Phi'_{\vareps}(v_{\vareps,N})\left|\partial_t v_{\vareps,N}\right|^2\,dx+\frac{1}{2}\partial_t B(\Phi_{\vareps}(v_{\vareps,N}),\Phi_{\vareps}(v_{\vareps,N}))=\left\langle f_{\vareps},\partial_t \Phi_{\vareps}(v_{\vareps,N})\right\rangle.
				\end{split}
				\]
				An integration over $[0,t]\subset [0,T]$, using the uniform ellipticity of $\Phi'_{\vareps}$ and $K$ and the fractional Poincar\'e inequality, we get 
				\[
				\begin{split}
					&\quad \int_0^t\int_{\Omega}|\partial_t v_{\vareps,N}|^2\,dx ds+\|\Phi_{\vareps}(v_{\vareps,N}(t))\|^2_{H^s(\R^n)}\\
					&\leq C\LC B(\Phi_{\vareps}(v_{0,N}),\Phi_{\vareps}(v_{0,N}))+\|f_{\vareps}\|_{L^2(\Omega_T)}\|\partial_t \Phi_{\vareps}(v_{\vareps,N})\|_{L^2(\Omega_T)}\RC .
				\end{split}
				\]
				By standard arguments this gives
				\[
				\begin{split}
					&\quad \|\partial_t v_{\vareps,N}\|_{L^2(\Omega_T)}+\|\Phi_{\vareps}(v_{\vareps,N})\|_{L^{\infty}(0,T;H^s(\R^n))}\\
					&\leq C\LC B(\Phi_{\vareps}(u_{0,N}),\Phi_{\vareps}(u_{0,N}))+\|f_{\vareps}\|_{L^2(\Omega_T)}\RC .
				\end{split}
				\]
				Finally, using the uniform ellipticity of $K$ and $\Phi_{\vareps}$, $u_0\in \widetilde{H}^s(\Omega)$ and \eqref{eq: compatibility of approx}, we get
				\begin{equation}
					\label{eq: time derivative estimate indep of N}
					\begin{split}
						&\quad \left\|\partial_t v_{\vareps,N}\right\|_{L^2(\Omega_T)}+\left\|\Phi_{\vareps}(v_{\vareps,N})\right\|_{L^{\infty}(0,T;H^s(\R^n))}\\
						&\leq  C\LC [\Phi_{\vareps}(u_{0,N})]_{H^s(\R^n)}+\|f_{\vareps}\|_{L^2(\Omega_T)}\RC \\
						&\leq C\LC [u_{0,N}]_{H^s(\R^n)}+\|f_{\vareps}\|_{L^2(\Omega_T)}\RC\\
						&\leq C\LC [u_{0}]_{H^s(\R^n)}+\|f_{\vareps}\|_{L^2(\Omega_T)}\RC .
					\end{split}
				\end{equation}
				Hence, $\partial_t \Phi_{\vareps}(v_{\vareps,N})$ is uniformly bounded in $L^2(\Omega_T)$ in $N$. Therefore, we conclude that
				\begin{equation}
					\label{eq: time regularity for each epsilon}
					\partial_t v_{\vareps,N}\weak \partial_t v_{\vareps}\ \text{in} \ L^2(\Omega_T) \text{ as }N\to \infty
				\end{equation}
                as we wish (see \eqref{eq: strong convergence in L2 in N}).  Hence, we have established \ref{time regularity}.

                Additionally, we see that $\Phi_{\vareps}(v_{\vareps,N})$ is uniformly bounded in $L^{\infty}(0,T;H^s(\R^n))$ and hence $\Phi_{\vareps}(v_{\vareps,N})\weak w_{\vareps}$ in $L^2(0,T;\widetilde{H}^s(\Omega))$ as $N\to\infty$ for some $w_{\vareps}\in L^2(0,T;\widetilde{H}^s(\Omega))$. Again by the Aubin--Lions lemma (Theorem~\ref{Aubin-Lions lemma}), we have $\Phi_{\vareps}(v_{\vareps,N})\to w_{\vareps}$ in $L^2(\Omega_T)$ as $N\to\infty$. Furthermore, the Lipschitz continuity of $\Phi_{\vareps}$ and \eqref{eq: strong convergence in L2 in N} imply $w_{\vareps}=\Phi_{\vareps}(v_{\vareps})$. Hence, we can conclude
				\begin{equation}
					\label{eq: space regularity v epsilon}
					\Phi_{\vareps}(v_{\vareps,N})\weak \Phi_{\vareps}(v_{\vareps})\, \text{in} \ L^2(0,T;\widetilde{H}^s(\Omega)) \text{ as }N\to\infty.
				\end{equation}
		This shows \ref{nonlinear regularity}. By \eqref{eq: strong convergence in L2 in N} and $\Phi'_{\vareps}\in C^0$, we know
				\[
				\Phi'_{\vareps}(v_{\vareps,N})\to \Phi'_{\vareps}(v_{\vareps})\ \text{a.e. in}\ \Omega_T \text{ as }N\to \infty,
				\]
			but the uniform boundedness of $\Phi'_{\vareps}$ and Lebesgue's dominated convergence theorem give rise to 
				\begin{equation}
					\label{eq: strong convergence of derivatives}
					\Phi'_{\vareps}(v_{\vareps,N})\to \Phi'_{\vareps}(v_{\vareps})\ \text{in}\ L^2(\Omega_T) \text{ as }N\to \infty.
				\end{equation}
				Therefore, we can deduce
				\[
				\partial_t\Phi_{\vareps}(v_{\vareps,N})=\Phi'_{\vareps,N}(v_{\vareps,N})\partial_t v_{\vareps,N}\to \Phi'_{\vareps}(v_{\vareps})\partial_t v_{\vareps}\ \text{in}\ L^2(\Omega_T) \text{ as }N\to\infty,
				\]
				 since it is a product of a weakly and strongly converging sequence in $L^2(\Omega_T)$ (see \eqref{eq: time regularity for each epsilon} and \eqref{eq: strong convergence of derivatives}). On the other hand, by \eqref{eq: space regularity v epsilon} we know 
                 \[
                 \partial_t \Phi_{\vareps}(v_{\vareps,N})\to \partial_t \Phi_{\vareps}(v_{\vareps}) \text{ in }\distr(\Omega_T)\text{ as }N\to\infty,
                 \]
                  and hence
				\begin{equation}
					\label{eq: convergence of time derivative of nonlinearity}
					\partial_t \Phi_{\vareps}(v_{\vareps})=\Phi'_{\vareps}(v_{\vareps})\partial_t v_{\vareps}\ \text{a.e. in} \ \Omega_T.
				\end{equation}
				This finally shows \ref{time regularity nonlinearity} and we can conclude the proof of Claim \ref{claim further regularity of solution v epsilon}.
			\end{proof}

			In order to complete the proof of Theorem \ref{thm: basic existence}, we next prove a maximum principle for problem \eqref{eq: homogeneous problem for approximate problem}.
			\begin{claim}
				\label{claim: maximum principle}
				There exists $M>0$ such that
				\begin{equation}
					\label{eq: maximum princple}
					0\leq v_{\vareps}(x,t)\leq M\ \text{a.e. in}\ \Omega_T.
				\end{equation}
			\end{claim}
			\begin{proof}[Proof of Claim \ref{claim: maximum principle}]
				First note that by construction the function $u_{\vareps}\vcentcolon = v_{\vareps}+\varphi\in L^2(0,T;H^s(\R^n))$ solves
				
				\begin{equation}
					\label{eq: original eq for approximate sols max princ}
					\begin{cases}
						\partial_t u+ L(\Phi_{\vareps}(u))=0&\text{ in }\Omega_T,\\
						u=\varphi&\text{ in }(\Omega_e)_T,\\
						u(0)=u_0&\text{ in }\Omega.
					\end{cases}
				\end{equation}
				Here, we used that $\supp\LC v_{\vareps}\RC \subset\Omega_T$, $\supp \LC \varphi\RC \subset [0,T]\times\Omega_e$ are disjoint, and hence
				\[
				\Phi_{\vareps}(u_{\vareps})=\Phi_{\vareps}(v_{\vareps})+\Phi_{\vareps}(\varphi).
				\]
				Moreover, observe by Claim~\ref{claim further regularity of solution v epsilon}, there holds
				\begin{equation}
					\label{eq: PDE for max principle}
					\int_{\Omega_T} \LC \partial_t u_{\vareps}\RC \psi \,dxdt+\int_0^T B(\Phi_{\vareps}(u_{\vareps}),\psi)\,dt=0
				\end{equation}
				for all $\psi \in L^2(0,T;\widetilde{H}^s(\Omega))$ and thus 
				\begin{equation}
					\label{eq: PDE for max principle 2}
					\int_{\Omega} \LC \partial_t u_{\vareps}\RC \psi \,dx+ B(\Phi_{\vareps}(u_{\vareps}),\psi)=0
				\end{equation}
				for a.e. $0<t<T$ and all $\psi\in\widetilde{H}^s(\Omega)$.

                Next, let us define $w_{\vareps}\in L^2(\R^n_T)$ by
				\begin{equation}
					\label{eq: auxiliary test function}
					w_{\vareps}=\LC u_{\vareps}-M\RC_+ \quad \text{with} \quad  M\vcentcolon =\sup_{x\in\Omega}u_0+\sup_{(x,t)\in [0,T]\times\Omega_e}\varphi.
				\end{equation}
				Here and later, we write a subscript $+$ for the positive part and $-$ for the negative part of a given function. We assert that $w_{\vareps}\in H^1(0,T;L^2(\Omega))\cap L^2(0,T;\widetilde{H}^s(\Omega))$. In fact, by \eqref{eq: auxiliary test function} we known $\partial_t u_{\vareps}\in L^2(\Omega_T)$ and hence $\partial_t w_{\vareps}\in L^2(\Omega_T)$, then it is known that 
				\begin{equation}
					\label{eq: time derivative max principle}
					\partial_t w_{\vareps}=\LC \partial_t u_{\vareps}\RC\chi_{\{ u_{\vareps}\geq M\} }\in L^2(\Omega_T),
				\end{equation}
				where $\chi_A$ denotes the characteristic function of a given set $A$. On the other hand, we may estimate (suppressing the $t$ dependence)
				\[
				\left[w_{\vareps}\right]_{H^s(\R^n)}\leq \left[u_{\vareps}\right]_{H^s(\R^n)}.
				\]
				where we used that $t\mapsto t_+=\max(t,0)$ is Lipschitz with Lipschitz constant $1$. Taking into account
				\[    \left\|w_{\vareps}\right\|_{L^2(\R^n)}=\left\|w_{\vareps}\right\|_{L^2(\Omega)}\leq \left\|u_{\vareps}\right\|_{L^2(\Omega)}+M|\Omega|^{1/2},
				\]
				where $|\Omega|$ denotes the Lebesgue measure of $\Omega$, and this implies $w_{\vareps}\in L^2(0,T;H^s(\R^n))$. Moreover, by assumption there holds $w_{\vareps}=0$ in $(\Omega_e)_T$ and thus the Lipschitz continuity gives $w_{\vareps}\in L^2(0,T;\widetilde{H}^s(\Omega))$. This completes the proof of the fact $w_{\vareps}\in H^1(0,T;L^2(\Omega))\cap L^2(0,T;\widetilde{H}^s(\Omega))$.
				
				Using $\psi=w_{\vareps}$ as a test function in \eqref{eq: PDE for max principle 2} and \eqref{eq: time derivative max principle}, we have
				\begin{equation}
					\label{eq: first chain of estimates}
					\begin{split}
						0&=\int_{\Omega} (\partial_t u_{\vareps})w_{\vareps} \,dx+ B(\Phi_{\vareps}(u_{\vareps}),w_{\vareps})\\
						&=\int_{\Omega} (\partial_t w_{\vareps})w_{\vareps} \,dx+ B(\Phi_{\vareps}(u_{\vareps}),w_{\vareps})\\
						&=\partial_t\int_{\Omega} \frac{|w_{\vareps}|^2}{2} \,dx+ B(\Phi_{\vareps}(u_{\vareps}),w_{\vareps}).
					\end{split}
				\end{equation}
				Next, we write the bilinear form as
				\begin{equation}
					\label{eq: for sign estimate bilinear}
					\begin{split}
						B(\Phi_{\vareps}(u_{\vareps}),w_{\vareps})
						=&\int_{\R^{2n}}K(x,y) \LC \int_0^1 \Phi'_{\vareps}(u_{\vareps}(y)+\tau(u_{\vareps}(x)-u_{\vareps}(y))\,d\tau\RC  \\
						&\qquad \qquad \cdot \frac{(u_{\vareps}(x)-u_{\vareps}(y))(w_{\vareps}(x)-w_{\vareps}(y))}{|x-y|^{n+2s}}\,dxdy
					\end{split}
				\end{equation}
				and observe
				\[
				\begin{split}
					&\quad \LC u_{\vareps}(x)-u_{\vareps}(y)\RC \LC w_{\vareps}(x)-w_{\vareps}(y)\RC\\
					&=\left|w_{\vareps}(x)-w_{\vareps}(y)\right|^2-\LC(u_{\vareps}-M)_{-}(x)-(u_{\vareps}-M)_{-}(y)\RC \LC w_{\vareps}(x)-w_{\vareps}(y)\RC\\
					&=\left|w_{\vareps}(x)-w_{\vareps}(y)\right|^2+\LC u_{\vareps}-M\RC _{-}(x)w_{\vareps}(y)+\LC u_{\vareps}-M\RC _{-}(y)w_{\vareps}(x)\\
					&\geq 0.
				\end{split}
				\]
				This guarantees 
				\begin{equation}
					\label{eq: nonneg bilinear form}
					\int_0^t B(\Phi_{\vareps}(u_{\vareps}),w_{\vareps})\,d\tau\geq 0
				\end{equation}
				for all $0\leq t\leq T$. Integration of \eqref{eq: first chain of estimates} over $[0,t]\subset [0,T]$ gives
				\[
				\int_{\Omega} \frac{\left|w_{\vareps}(t)\right|^2}{2} \,dx+ \int_0^t B(\Phi_{\vareps}(u_{\vareps}(\tau)),w_{\vareps}(\tau))=\int_{\Omega} \frac{\left|w_{\vareps}(0)\right|^2}{2} \,dx.
				\]
				By the fact that $w_{\vareps}(0)=0$ and \eqref{eq: nonneg bilinear form}, we deduce $w_{\vareps}=0$ a.e. in $\Omega_T$ and hence $u_{\vareps}\leq M$ a.e. in $\Omega_T$. This establishes the upper bound.
				
				On the other hand, we want to show that $u_{\vareps}\geq 0$ a.e. in $\Omega_T$. The proof is exactly the same but this time we take 
				\[
				w_{\vareps}=\LC u_{\vareps}\RC_{-}\in L^2(0,T;\widetilde{H}^s(\Omega))
				\]
				as a test function in \eqref{eq: PDE for max principle 2}. This regularity follows precisely the same lines as before. Taking into account
				\[
				\partial_t w_{\vareps}= \LC -\partial_t u_{\vareps}\RC \chi_{\{  u_{\vareps}\leq  0\}},
				\]
				we deduce
				\begin{equation}
					\label{eq: for lower bound}
					-\partial_t\int_{\Omega}\frac{\left|w_{\vareps}\right|^2}{2}\,dx+B(\Phi_{\vareps}(u_{\vareps}),w_{\vareps})=0
				\end{equation}
				for a.e. $0<t< T$. Next, note that we have
				\begin{equation}
					\label{eq: negativity}
					\begin{split}
						&\quad (u_{\vareps}(x)-u_{\vareps}(y))(w_{\vareps}(x)-w_{\vareps}(y))\\
						&=-(w_{\vareps}(x)-w_{\vareps}(y))^2-(u_{\vareps})_+(x)w_{\vareps}(y)-(u_{\vareps})_+(y)w_{\vareps}(x) \\
						&\leq 0.
					\end{split}
				\end{equation}
				Hence, combining \eqref{eq: negativity} with   \eqref{eq: for sign estimate bilinear} we get
				\[
				\int_0^tB(\Phi_{\vareps}(u_{\vareps}),w_{\vareps}) \,d\tau\leq 0
				\]
				for $0<t<T$. Hence, integrating \eqref{eq: for lower bound} over $[0,t]\subset [0,T]$, we obtain
				\begin{equation}
                \label{eq: eq for nonneg}
				\int_{\Omega}\frac{\left|w_{\vareps}(t)\right|^2}{2}\,dx-\int_0^t B(\Phi_{\vareps}(u_{\vareps}),w_{\vareps}) \,d\tau=\int_{\Omega}\frac{\left|w_{\vareps}(0)\right|^2}{2}\,dx.
				\end{equation}
				By the fact that $w_{\vareps}(0)=0$ as $u_0\geq 0$, the right hand side vanishes. From \eqref{eq: negativity}, we deduce that terms on the left hand side of \eqref{eq: eq for nonneg} are nonnegative and thus it follows that $u_{\vareps}\geq 0$  a.e. in $\Omega_T$. Since $\varphi=0$ in $\Omega_T$, this shows the desired inequality \eqref{eq: maximum princple} for $v_{\vareps}$, which establishes Claim \ref{claim: maximum principle}.
			\end{proof}
			As an immediate consequence of Claim~\ref{claim: maximum principle}, we see that there exists $\mathcal{M}>0$ such that
			\begin{equation}
				\label{eq: uniform estimate nonlinearity}
				0\leq \mathcal{V}_{\vareps}\vcentcolon = \Phi_{\vareps}(v_{\vareps})\leq \mathcal{M}\ \text{a.e. in}\ \Omega_T
			\end{equation}
			for all sufficiently small $\vareps>0$. The lower bound follows from the fact that $v_{\vareps}\geq 0$ and $\Phi_{\vareps}(t)\geq 0$ for $t\geq 0$. On the other hand, $0\leq v_{\vareps}\leq M$ and \ref{cond 4 auxiliary function} of Lemma~\ref{auxiliary lemma}, establishes
			\[
			\begin{split}
				\mathcal{V}_{\vareps}&=\Phi_{\vareps}(v_{\vareps})\\
      &\leq  |\Phi_{\vareps}(v_{\vareps})-\Phi(v_{\vareps})|+\Phi(v_{\vareps})\\
			&	\leq  \|\Phi_{\vareps}-\Phi\|_{L^{\infty}([0,M])}+\Phi(v_{\vareps}) \\
    &\leq 2\Phi(M),
			\end{split}
			\]
			where we used that $\Phi(M)<\infty$, $M>0$. Hence, $v_{\vareps},\mathcal{V}_{\vareps}\in L^{\infty}(\Omega_T)$ are uniformly bounded in $\vareps$ in this space and thus there exist $\overline{v},\overline{\mathcal{V}}\in L^{\infty}(\Omega_T)$ such that
			\begin{equation}
				\label{eq: weak star convergence in Linfty}
				v_{\vareps}\weakstar \overline{v}\quad  \mathcal{V}_{\vareps}\weakstar \overline{\mathcal{V}}\quad  \text{in}\quad  L^{\infty}(\Omega_T), \text{ as }\vareps \to 0,
			\end{equation}
		and by lower semi-continuity of $\|\cdot\|_{L^{\infty}(\Omega_T)}$ under weak-$\ast$ convergence in $L^{\infty}(\Omega_T)$ we have
			\begin{equation}
				\label{eq: uniform bound}
				\left\|\overline{v} \right\|_{L^{\infty}(\Omega_T)}\leq M,\quad \text{and} \quad    \left\|\overline{\mathcal{V}}\right\|_{L^{\infty}(\Omega_T)}\leq \mathcal{M}.
			\end{equation}
			a.e. in $\Omega_T$. By boundedness of $\Omega_T$ and \eqref{eq: weak star convergence in Linfty}, we have
			\begin{equation}
				\label{eq: weak convergence in any Lp}
				v_{\vareps}\weak \overline{v},\ \mathcal{V}_{\vareps}\weak \overline{\mathcal{V}}\ \text{in}\ L^{p}(\Omega_T)
			\end{equation}
			for any $1\leq p<\infty$.

          Next, let us denote by $\Psi_{\vareps}\colon\R\to\R$ to be the antiderivative of $\Phi_{\vareps}$ satisfying $\Psi_{\vareps}(0)=0$.
			Since $\mathcal{V}_{\vareps}\in L^2(0,T;\widetilde{H}^s(\Omega))$ and $\partial_t v_{\varepsilon}\in L^2(\Omega_T)$ (see Claim~\ref{claim further regularity of solution v epsilon}), we can test the equation for $v_{\vareps}$ with $\mathcal{V}_{\vareps}$ to obtain
			\begin{equation}
				\label{eq: test with Phi v}
				\int_{\Omega_t}\LC \partial_t v_{\vareps}\RC \mathcal{V}_{\vareps}\,dxd\tau+\int_0^t B(\Phi_{\vareps}(v_{\vareps}),\mathcal{V}_{\vareps})\,d\tau=\int_{\Omega_t}f_{\vareps}\mathcal{V}_{\vareps}\,dxd\tau
			\end{equation}
			for $0<t\leq T$. Clearly, there holds
			\[
			\partial_t \LC \Psi_{\vareps}(v_{\vareps})\RC =\Psi'_{\vareps}(v_{\vareps})\partial_t v_{\vareps}=\mathcal{V}_{\vareps}\partial_t v_{\vareps}.
			\]
			Hence, using the uniform ellipticity of $K$, the fractional Poincar\'e inequality and H\"older's inequality, we obtain
			\[
			\begin{split}
				& \quad \int_{\Omega} \Psi_{\vareps}(v_{\vareps})(t)\,dx +\int_0^t\|\Phi_{\vareps}(v_{\vareps})\|_{H^s(\R^n)}^2\,d\tau\\
			&	\leq  C\left(\|f_{\vareps}\|_{L^2(\Omega_t)}\|\Phi_{\vareps}\|_{L^2(\Omega_t)}+\int_{\Omega}\Psi_{\vareps}(u_0)\,dx\right)
			\end{split}
			\]
			By the usual argument this gives
			\begin{equation}
				\label{eq: useful estimate for passing epsilon to 0}
				\begin{split}
					\sup_{0\leq t\leq T}\int_{\Omega} \Psi_{\vareps}(v_{\vareps})(t)\,dx+\|\Phi_{\vareps}(v_{\vareps})\|_{L^2(0,T;H^s(\R^n))}^2
					\leq  C\left(\|f_{\vareps}\|_{L^2(\Omega_T)}^2+\int_{\Omega}\Psi_{\vareps}(u_0)\,dx\right).
				\end{split}
			\end{equation}
			Assuming $\vareps$ is sufficient small (i.e. $\left\|u_0\right\|_{L^{\infty}(\Omega)}<1/\vareps$), then
			\begin{equation}
				\label{eq: uniform bound last integral}
				\int_{\Omega}\Psi_{\vareps}(u_0)\,dx\leq C
			\end{equation}
			for all $\vareps>0$ sufficiently small.

                On the other hand, let $\mathcal{K}\subset [0,T]\times \Omega_e$ be a compact set such that $\supp(\varphi)\subset \mathcal{K}$ and denote its projection onto $\R^n$ by $\mathcal{K}'\subset\Omega_e$, then using \ref{cond 1 on f}, $\Phi_{\vareps}(0)=0$, the monotonicity of $\Phi_{\vareps}$ and $\Phi_{\vareps}|_{[\vareps,1/\vareps]}=\Phi$ (see Lemma~\ref{auxiliary lemma}), we deduce
			\begin{equation}
				\label{eq: uniform bound f epsilon}
				\begin{split}
					\left\|f_{\vareps}\right\|_{L^2(\Omega_T)}^2&\leq C \int_{\Omega_T}\left(\int_{\mathcal{K}'}\frac{\Phi_{\vareps}(\varphi(y,t))}{|x-y|^{n+2s}}\,dy\right)^2\,dxdt\\
					&\leq \frac{C}{\dist(\partial\Omega,\mathcal{K}')^{2n+4s}}|\Omega|\int_0^T\left(\int_{\mathcal{K}'}\Phi(\max(1,\|\varphi\|_{L^{\infty}((\Omega_e)_T)}))\,dy\right)^2\\
                  &\leq C,
				\end{split}
			\end{equation}
			for all $\vareps >0$ sufficiently small, where $C>0$ is a constant independent of $\vareps>0$. Therefore, \eqref{eq: useful estimate for passing epsilon to 0}, \eqref{eq: uniform bound last integral} and \eqref{eq: uniform bound f epsilon}, show that $\mathcal{V}_{\vareps}$ is uniformly bounded in $L^2(0,T;\widetilde{H}^s(\Omega))$. Hence, we have
			\begin{equation}
				\label{eq: convergence nonlinearity}
				\overline{\mathcal{V}}\in L^2(0,T;\widetilde{H}^s(\Omega))\quad  \text{and}\quad  \mathcal{V}_{\vareps}\weak \overline{\mathcal{V}}\ \text{in}\ L^2(0,T;\widetilde{H}^s(\Omega)) \text{ as }\vareps\to 0.
			\end{equation} 
            Furthermore, we know that
			$L(\Phi_{\vareps}(\varphi))\to L(\Phi(\varphi))$ a.e. in $\Omega_T$ as $\vareps\to 0$. The convergence $\Phi_{\vareps}(\varphi)\to \Phi(\varphi)$ a.e. in $\R^n_T$ follows from the fact that $\varphi\in C_c([0,T]\times \Omega_e)$ and $\Phi_{\vareps}\to \Phi$ uniformly on compact sets. This pointwise convergence, the uniform ellipticity of $K$ and $\varphi\in C_c([0,T]\times \Omega)$, then imply the asserted convergence. But then \eqref{eq: uniform bound f epsilon} and Lebesgue's dominated convergence theorem justifies
			\begin{equation}
				\label{eq: convergence of right hand side}
				f_{\vareps}\to f\vcentcolon = - L(\Phi(\varphi))\ \text{in}\ L^2(\Omega_T) \text{ as }\vareps\to 0.
			\end{equation}
		Now, recall that $v_{\vareps}$ satisfies
			\begin{equation}
				\label{eq: almost finished PDE}
				-\int_{\Omega_T}v_{\vareps}\partial_t\psi\,dxdt+\int_0^T B(\Phi_{\vareps}(v_{\vareps}),\psi)\,dt=\int_{\Omega_T}f_{\vareps}\psi\,dxdt+\int_{\Omega}u_0\psi(0)\,dx,
			\end{equation}
			for all $\psi\in C_c^{\infty}([0,T)\times \Omega)$. Using \eqref{eq: weak convergence in any Lp}, \eqref{eq: convergence nonlinearity} and \eqref{eq: convergence of right hand side}, we can pass to the limit $\vareps\to 0$ and obtain
			\begin{equation}
				\label{eq: limit equation}
				-\int_{\Omega_T}\overline{v}\partial_t\psi\,dxdt+\int_0^T B(\overline{\mathcal{V}},\psi)\,dt=\int_{\Omega_T}f\psi\,dxdt+\int_{\Omega}u_0\psi(0)\,dx,
			\end{equation}
			for all $\psi\in C_c^{\infty}([0,T)\times \Omega)$. Next, we show:
			\begin{claim}
				\label{claim: form of nonlinearity}
				We have $\overline{\mathcal{V}}=\Phi(\overline{v})$.
			\end{claim}
			\begin{proof}[Proof of Claim \ref{claim: form of nonlinearity}]
				We first want to show that $\partial_t v_{\vareps}$ is uniformly bounded in $L^2(0,T;H^{-s}(\Omega))$. For this observe that we have
				\[
				\partial_t v_{\vareps}=-L(\mathcal{V}_{\vareps})+f_{\vareps}
				\]
				as distributions on $\Omega_T$. Now, as $\mathcal{V}_{\vareps}$ is uniformly bounded in $L^2(0,T;\widetilde{H}^s(\Omega))$ and $f_{\vareps}$ in $L^2(\Omega_T)$, we see that $\partial_t v_{\vareps}$ is uniformly bounded in $L^2(0,T;H^{-s}(\Omega))$. Additionally by Claim~\ref{claim: maximum principle}, $v_{\vareps}$ is uniformly bounded in $L^2(\Omega_T)$. Since the compact embedding $\widetilde{H}^s(\Omega)\hookrightarrow L^2(\Omega)$ and Schauder's theorem tells us that $L^2(\Omega)\hookrightarrow H^{-s}(\Omega)$ is compact, we see that all conditions of Theorem~\ref{Aubin-Lions-Simon lemma} with
				\[
				X=L^2(\Omega), \ B=Y=H^{-s}(\Omega), \text{ for } p=r=2, \text{ and } s=1
				\]
				are satisfied. Hence, we can conclude that up to extracting a subsequence we have 
				\begin{equation}
					\label{eq: strong convergence of v epsilon}
					v_{\vareps}\to \overline{v}\ \text{in}\ L^2(0,T;H^{-s}(\Omega)) \text{ as }\vareps \to 0.
				\end{equation}
				But then taking into account \eqref{eq: convergence nonlinearity}, we get
				\begin{equation}
					\label{eq: convergence of product}
			  \int_{\Omega_T}v_{\vareps}\mathcal{V}_{\vareps}\,dxdt\to \int_{\Omega_T} \overline{v}\overline{\mathcal{V}}\,dxdt \text{ as }\vareps\to 0,
				\end{equation}
			   since its the product of a strongly converging sequence in $L^2(0,T;H^{-s}(\Omega))$ and a weakly converging sequence in $L^2(0,T;H^{-s}(\Omega))$. 
				
				Next choose $w\in C_c^{\infty}(\Omega_T)$ and observe that the monotonicity of $\Phi_{\vareps}$ implies
				\begin{equation}
					\label{eq: monotonicity estimate}
					\begin{split}
						0&\leq \int_{\Omega_T}\LC \mathcal{V}_{\vareps}-\Phi_{\vareps}(w)\RC \LC v_{\vareps}-w \RC dxdt\\
						&=\int_{\Omega_T}(\mathcal{V}_{\vareps}v_{\vareps}-\mathcal{V}_{\vareps}w-\Phi_{\vareps}(w)v_{\vareps}+\Phi_{\vareps}(w)w)\,dxdt.
					\end{split}
				\end{equation}
				As $v_{\vareps}\weak \overline{v}$ in $L^2(\Omega_T)$, $\mathcal{V}_{\vareps}\weak \overline{\mathcal{V}}$ in $L^2(0,T;\widetilde{H}^s(\Omega))$, $\Phi_{\vareps}\to\Phi$ uniformly on compact sets and \eqref{eq: convergence of product}, we obtain in the limit $\vareps\to 0$:
				\begin{equation}
					\label{eq: montonocity estimate}
					\begin{split}
						0&\leq \int_{\Omega_T}\LC \overline{\mathcal{V}}\overline{v}-\overline{\mathcal{V}}w-\Phi(w)\overline{v}+\Phi(w)w\RC dxdt\\
						&= \int_{\Omega_T}\LC \overline{\mathcal{V}}-\Phi(w)\RC (\overline{v}-w)\,dxdt,
					\end{split}
				\end{equation}
				for all $w\in C_c^{\infty}(\Omega_T)$. Now, since $v_{\vareps},\overline{v}\in L^{\infty}(\Omega_T)$, we can change $\Phi$ (and accordingly $\Phi_{\vareps}$) outside a compact set without affecting our established convergence results. In fact, if we make $\Phi$ outside a sufficiently large compact set linear, then we see by an approximation argument that \eqref{eq: montonocity estimate} is true for all $L^2(\Omega_T)$. 
				
				But then we have for all $\eta \in L^2(\Omega_T)$ and $0\leq \lambda\leq 1$ the estimate
				\begin{equation}
					\label{eq: useful id}
					\left\langle \overline{\mathcal{V}}-\widetilde{\Phi}((1-\lambda)\overline{v}+\lambda \eta),\overline{v}-\eta\right\rangle_{L^2}\geq 0.
				\end{equation}
				Next, we observe that $L^2(\Omega_T)\ni v\mapsto \langle \widetilde{\Phi}(v),u\rangle_{L^2}$ is continuous for every fixed $u\in L^2(\Omega_T)$. In fact, if $v_k\to v$ in $L^2(\Omega_T)$, then up to extracting a subsequence we have $\widetilde{\Phi}(v_k)\to \widetilde{\Phi}(v)$ a.e. in $\Omega_T$. But then the linearity assumption outside a compact set guarantees that 
				\[
				\left|\widetilde{\Phi}(v_k)u\right|\leq \max \LC C,c|v_k|\RC |u|\in L^1(\Omega_T).
				\]
				By \cite[Theorem~4.9]{Brezis}, $|v_k|\leq |h|$ for some $h\in L^2(\Omega_T)$ and thus $|\widetilde{\Phi}(v_k)u|\leq g$ for some $g\in L^1(\Omega_T)$. But then the dominated convergence theorem guarantees 
				\[
				\left\langle\widetilde{\Phi}(v_k),u\right\rangle_{L^2}\to \left\langle \widetilde{\Phi}(v),u\right\rangle_{L^2} \text{ as }k\to\infty. 
				\]
			 Hence, we can pass to the limit $\lambda\to 0$ in \eqref{eq: useful id} and get
				\[
				\left\langle \overline{\mathcal{V}}-\widetilde{\Phi}(\overline{v}),\overline{v}-\eta\right\rangle_{L^2(\Omega_T)}\geq 0.
				\]
				This can only hold if $\overline{\mathcal{V}}=\widetilde{\Phi}(\overline{v})$ and we can conclude the proof of Claim \ref{claim: form of nonlinearity}.
			\end{proof}
			Next, we may observe by the compactnes of the embedding $\widetilde{H}^s(\Omega)\hookrightarrow L^2(\Omega)$ that $\mathcal{V}_{\vareps}\weak \overline{\mathcal{V}}$ in $L^2(0,T;\widetilde{H}^s(\Omega))$ implies $\mathcal{V}_{\vareps}\to \mathcal{V}$ in $L^2(\Omega_T)$ and a.e. in $\Omega_T$ as $\vareps\to 0$, but then $\mathcal{V}_{\vareps}\geq 0$ ensures $\overline{\mathcal{V}}\geq 0$. Since by Claim~\ref{claim: form of nonlinearity} we have $\overline{\mathcal{V}}=\Phi(\overline{v})$, this guarantees $\Phi(\overline{v})\geq 0$ and thus $\overline{v}\geq 0$ as $\Phi(t)< 0$ for $t<0$. Therefore, we have found a solution of \eqref{eq: homogeneous FPME without exterior cond}. Using the observation at the beginning of the proof, we see that $u_{\vareps}=v_{\vareps}+\varphi$ is a solution of the original problem \eqref{eq: homogeneous FPME} and the estimate \eqref{eq: bound on solution} follows from \eqref{eq: uniform bound}. Let us note that $\Phi(u)=\Phi(u-\varphi)+\Phi(\varphi)$ belongs to $L^2(0,T;H^s(\R^n))$, as $\Phi(\varphi)\in L^2(0,T;H^s(\R^n))$. This proves the assertion of Theorem \ref{thm: basic existence}. 
		\end{proof}

            We next define the used notion of weak solutions to NPMEs with a linear absorption term.
		
		\begin{definition}[Weak solutions with absorption term]
			\label{def: weak solutions 2}
			Let $\Omega\subset\R^n$ be a bounded Lipschitz domain, $T>0$, $0<s<\alpha\leq 1$, $m\geq 1$ and $L_K\in \mathcal{L}_0$. Assume additionally that we have given $\rho,q\in C^{1,\alpha}_+(\R^n)$ with $\rho$ uniformly elliptic. For given $u_0\in L^{\infty}(\Omega)$, $\varphi\in C_c([0,T]\times \Omega_e)$ with $\Phi^m(\varphi)\in L^2(0,T;H^s(\R^n))$ and $f\in L^2(0,T;H^{-s}(\Omega))$, we say that $u\colon\R^n_T\to\R$ is a \emph{(weak) solution} of
			\begin{equation}
				\label{eq: FPME}
				\begin{cases}
					\rho\partial_t u+ L_K(\Phi^m(u))+qu=f &\text{ in }\Omega_T,\\
					u=\varphi&\text{ in }(\Omega_e)_T,\\
					u(0)=u_0&\text { in }\Omega,
				\end{cases}
			\end{equation}
			provided that 
			\begin{enumerate}[(i)]
				\item $\Phi^m(u)\in L^2(0,T;H^s(\R^n))$,
				\item $\Phi^m(u-\varphi)\in L^2(0,T;\widetilde{H}^s(\Omega))$,
				\item $u$ satisfies \eqref{eq: FPME} in the sense of distributions, that is there holds
				\begin{equation}
					\label{eq: weak formulation}
					\begin{split}
						&-\int_{\Omega_T}\rho u\partial_t\psi\,dxdt + \int_0^T B_K(\Phi^m(u),\psi)\,dt+\int_{\Omega_T}qu\psi\,dxdt\\
						=&\int_0^T\langle f,\psi\rangle\,dxdt+\int_{\Omega}\rho u_0\psi(0)\,dx
					\end{split}
				\end{equation}
				for all $\psi\in C_c^{\infty}([0,T)\times \Omega)$. 
			\end{enumerate}
			If it additionally satisfies $u\in L^{\infty}(\Omega_T)$, then it is called \emph{bounded solution}.
		\end{definition}
		
		\begin{remark}
			\label{remark: hoelder coeff estimate}
			Note that by \cite[Lemma~3.1]{CRTZ-2022} we have
			\begin{equation}
				\label{eq: continuity estimate hoelder coeff}
				\|Qw\|_{H^s(\R^n)}\leq C\|Q\|_{C^{0,\alpha}}\|w\|_{H^s(\R^n)}
			\end{equation}
			for $Q\in C^{1,\alpha}(\R^n)$ and $u\in H^s(\R^n)$, where $0<s<\alpha\leq 1$ and $C>0$ only depends on $s,\alpha$ and $n$. In particular, if $w\in \widetilde{H}^s(\Omega)$ and $\partial\Omega\in C^{0,1}$, then $Qw\in \widetilde{H}^s(\Omega)$. This in turn guarantees $Qv\in L^2(0,T;H^{-s}(\Omega))$, whenever $v\in L^2(0,T;H^{-s}(\Omega))$.
		\end{remark}
		
		\begin{theorem}[Existence result with linear absorption term]
			\label{Existence result with linear absorption term}
			Let $\Omega\subset\R^n$ be a bounded Lipschitz domain, $T>0$, $0<s<\alpha\leq 1$, $m> 1$ and $L_K\in \mathcal{L}_0$. Assume additionally that we have given $\rho,q\in C^{1,\alpha}_+(\R^n)$ with $\rho$ uniformly elliptic. For any $0\leq u_0\in L^{\infty}(\Omega)\cap \widetilde{H}^s(\Omega)$ and $0\leq \varphi\in C_c([0,T]\times \Omega_e)$ with $\Phi^m(\varphi)\in L^2(0,T;H^s(\R^n))$ there exists a non-negative, bounded solution of
			\begin{equation}
				\label{eq: homogeneous FPME general}
				\begin{cases}
					\rho\partial_t u+ L_K(\Phi^m(u))+qu=0&\text{ in }\Omega_T,\\
					u=\varphi&\text{ in }(\Omega_e)_T,\\
					u(0)=u_0&\text{ in }\Omega.
				\end{cases}
			\end{equation}
			Moreover, there holds
			\begin{equation}
				\label{eq: bound on solution existence 2}
				\|u\|_{L^{\infty}(\Omega_T)}\leq \sup_{x\in \Omega}u_0+\sup_{(x,t)\in [0,T]\times \Omega_e}\varphi.
			\end{equation}
		\end{theorem}
		
		\begin{proof}
			We use the same conventions as in the proof of Theorem~\ref{thm: basic existence}. Instead of repeating the whole proof, we will only highlight the main differences. \\
			
			\noindent\textit{Modifications in Step 1.} As in Theorem~\ref{thm: basic existence} we look for solutions of the form
			\[
			v_{\vareps,N}(x,t)=\sum_{j=1}^N c^j_{N,\vareps}(t)w_j(x)
			\]
			for $N\in\N$. Next note that in the current situation the ODE to solve in the Galerkin approximation are:
			\begin{equation}
				\label{eq: approx sol general}
				\left\langle \rho \partial_t v_{\vareps,N},w_j \right\rangle_{L^2}+B(\Phi_{\vareps}(v_{\vareps,N}),w_j)+\left\langle q v_{\vareps,N} , w_j\right\rangle_{L^2}=\left\langle F_N, w_j \right\rangle
			\end{equation}
			for all $1\leq j\leq N$ and $0<t<T$. Expanding everything, we see that $\hat{c}_{\vareps,N}=\LC c^1_{\vareps,N},\ldots,c^N_{\vareps,N}\RC \in C^1([0,T];\R^N)$ solves
			\begin{equation}
				\label{eq: IVP for coeff general}
				\begin{cases}
					A\partial_t \hat{c}_{\vareps,N}+ b(\hat{c}_{\vareps,N})+Q\hat{c}_{\vareps,N}=G_N(\hat{c}_{\vareps,N}),\quad\text{for}\quad 0<t<T\\
					\hat{c}_{\vareps,N}(0)=\hat{c}_{0,N},
				\end{cases}
			\end{equation}
			where 
			\begin{equation}
				\begin{cases}
					A=(A_{i,j})_{1\leq i,j\leq N}\quad  \text{with} \quad  A_{i,j}=\int_{\Omega}\rho w_i w_j\,dx,\\
					Q=(Q_{i,j})_{1\leq i,j\leq N}\quad  \text{with} \quad  Q_{i,j}=\left\langle q w_i,w_j\right\rangle_{L^2},\\
					\hat{c}_{0,N}=\LC \left\langle v_0,w_1 \right\rangle_{L^2},\ldots,\left\langle v_0,w_N\right\rangle_{L^2}\RC,\\ 
					G_N(\hat{c}_{\vareps,N})=\left\langle F_N, \sum_{j=1}^N c_{\vareps,N}^jw_j\right\rangle,
				\end{cases}
			\end{equation}
			and for $1\leq j\leq N$ the components $b_j\colon\R^N\to\R$ of $b=(b_1,\ldots,b_N)$ are defined by 
			\begin{equation}
				\label{eq: def coeff b general}
				\begin{split}
					b_j(c)=B\left(\Phi_{\vareps}\left(\sum_{k=1}^Nc_k w_k\right),w_j\right)
				\end{split}
			\end{equation}
			for $c=\LC c_1,\ldots,c_N\RC\in\R^N$.

            We observe that the matrix $A$ is invertible. First note that for all $\xi\in \R^N$, we have
			\[
			\begin{split}
				\xi\cdot A\xi=& \xi_i A_{ij}\xi_j=\int_{\Omega}\rho (\xi_i w_i)(\xi_j w_j)\,dx \\
				=&\left\langle \sqrt{\rho}\xi_iw_i,\sqrt{\rho}\xi_jw_j\right\rangle_{L^2}=\|\sqrt{\rho}\xi_i w_i\|_{L^2}^2\geq 0,
			\end{split}
			\]
			where we used the Einstein summation convention. Now, if $\xi\cdot A\xi=0$ for some $\xi\in\R^N$, then the above calculation shows
			\[
			\sum_{i=1}^N\xi_i\sqrt{\rho}w_i=0\ \text{a.e. in}\ \Omega.
			\]
			Since $\rho>0$, this implies 
			\[
			\sum_{i=1}^N\xi_iw_i=0\ \text{a.e. in}\ \Omega
			\]
			and hence the linear independence of $w_i$, $1\leq i\leq N$, shows $\xi=0$.
			Therefore, $A$ is positive definite. As $A$ is additionally symmetric, we see that $A$ is invertible. Therefore, finding a solution of \eqref{eq: IVP for coeff general} is equivalent to find a solution of
			\begin{equation}
				\label{eq: IVP for coeff general 2}
				\begin{cases}
					\partial_t \hat{c}_{\vareps,N}+ \widetilde{b}(\hat{c}_{\vareps,N})+\widetilde{Q}\hat{c}_{\vareps,N}=\widetilde{G}_N(\hat{c}_{\vareps,N}),\quad\text{for}\quad 0<t<T\\
					\hat{c}_{\vareps,N}(0)=\hat{c}_{0,N},
				\end{cases}
			\end{equation}
			where $\widetilde{b}=A^{-1}b$, $\widetilde{Q}=A^{-1}Q$ and $\widetilde{G}_N=A^{-1}G_N$. Clearly, for this new problem we can again find by Peano's existence theorem for ODEs a solution on possibly a small interval $[0,\delta]$, $\delta>0$. Due to the fact that $\rho$ is uniformly elliptic and $q\geq 0$, we again arrive at equation \eqref{eq: energy identity approx sol 7}. Hence, also in this setting we can extend $v_{\vareps,N}$ to a solution on $[0,T]$. Doing the computation in \eqref{eq: time regularity for each N} with $\psi/\rho$ instead of $\psi$ and using Remark~\ref{remark: hoelder coeff estimate}, we again see that $\partial_t v_{\vareps,N}\in L^2(0,T;H^{-s}(\Omega))$ with 
			\begin{equation}
				\label{eq: uniform bound in N time derivative general}
				\left\|\partial_t v_{\vareps,N}\right\|_{L^2(0,T;H^{-s}(\Omega))} \leq C
			\end{equation}
			for some $C>0$ independent of $N$.\\
			
			\noindent\textit{Modifications in Step 2.} All computations in this step hold without any further insights. \\
			
			\noindent\textit{Modifications in Step 3.} First note that the uniform ellipticity of $\rho$, $q\in L^{\infty}(\Omega)$ and \eqref{eq: energy identity approx sol 7} guarantee that \eqref{eq: time derivative estimate indep of N} still holds (after usings Young's inequality). Hence, we can again deduce property \ref{time regularity} of Claim~\ref{claim further regularity of solution v epsilon}. The proof of the other two properties of Claim~\ref{claim further regularity of solution v epsilon} remain the same. Also the proof of Claim~\ref{claim: maximum principle} is a simple modification after noting
			\[
			\begin{split}
				&qu_{\vareps}\LC u_{\vareps}-M\RC_+=q\LC u_{\vareps}-M\RC _+^2+qM\LC u_{\vareps}-M\RC_{+} \geq 0,\\
				&qu_{\vareps}(u_{\vareps})_{-}=-q(u_{\vareps})_{-}^2\leq 0.
			\end{split}
			\]
			Next, recall that to uniformly bound $\mathcal{V}_{\vareps}=\Phi_{\vareps}(v_{\vareps})$ independently of $\vareps$, we computed \eqref{eq: test with Phi v}. We see that the additional absorption term does not cause any problems since we have 
			\[
			qv_{\vareps}\mathcal{V}_{\vareps}\geq 0,
			\]
			by the monotonicity of $\Phi_{\vareps}$ and $\Phi_{\vareps}(0)=0$. Using the $v_{\vareps}\geq 0$, $\Psi_{\vareps}(t)\geq 0$ for $t\geq 0$ and the uniform ellipticity of $\rho$, we have again an estimate of the form \eqref{eq: useful estimate for passing epsilon to 0}. Then arguing the same way as above, we again see that we can pass to the limit $\vareps\to 0$ to obtain
			\begin{equation}
				\label{eq: almost finished PDE general}
				\begin{split}
					&-\int_{\Omega_T}\rho \overline{v}\partial_t\psi\,dxdt+\int_0^T B(\overline{\mathcal{V}},\psi)\,dt+\int_{\Omega_T}q\overline{v}\psi\,dxdt\\
					=&\int_{\Omega_T}f\psi\,dxdt+\int_{\Omega}\rho u_0\psi(0)\,dx,
				\end{split}
			\end{equation}
			for all $\psi\in C_c^{\infty}([0,T)\times \Omega)$. 
   
            Hence, it remains only to check that $\mathcal{V}=\Phi(\overline{v})$. Similarly as in the beginning of the proof of Claim~\ref{claim: form of nonlinearity}, we see that $\partial_t v_{\vareps}$ is uniformly bounded in $L^2(0,T;H^{-s}(\Omega))$. The only difference one has to notice is that taking into account the uniform bound of $v_{\vareps}$ in $L^2(\Omega_T)$ one gets from the PDE that $\rho \partial_t v_{\vareps}$ is uniformly bounded in $L^2(0,T;H^{-s}(\Omega))$, but then using Remark~\ref{eq: continuity estimate hoelder coeff} one obtains the uniform boundedness of $\partial_t v_{\vareps}$ in $L^2(0,T;H^{-s}(\Omega))$. The rest of the proof of Claim~\ref{claim: form of nonlinearity} is unaffected. Now, we can finish the proof of Theorem~\ref{Existence result with linear absorption term} completely analogous to the proof of Theorem~\ref{thm: basic existence}. 
		\end{proof}
		
		From the proof, we directly obtain the following corollary.
		
		\begin{corollary}
			\label{cor: approximation}
			Assume the conditions of Theorem~\ref{Existence result with linear absorption term} are satisfied and let $u$ be the constructed solution of \eqref{eq: homogeneous FPME general}. Then there exist a sequence $\LC \vareps_k\RC _{k\in\N}\subset (0,\infty)$ with $\vareps_k\to 0$ as $k\to\infty$, and a sequence of solutions $\LC u_k\RC_{k\in \N}$ of
			\begin{equation}
				\label{eq: approximate homogeneous FPME}
				\begin{cases}
					\rho\partial_t u+ L_K(\Phi^m_{\vareps_k}(u))+qu=0 &\text{ in }\Omega_T,\\
					u=\varphi&\text{ in }(\Omega_e)_T,\\
					u(0)=u_0&\text{ in }\Omega.
				\end{cases}
			\end{equation}
			Moreover, there holds 
			\begin{enumerate}[(i)]
				\item $\Phi_{\vareps_k}(u_k-\varphi)\in H^1(0,T; L^2(\Omega))\cap L^2(0,T;\widetilde{H}^s(\Omega))$,
				\item $u_{k}\in L^{\infty}(\Omega_T)$,
				\item $u_k\in L^2(0,T;H^s(\R^n))$ with $\partial_t u_k\in L^2(\Omega_T)$,
				\item $u_k\weak u$ in $L^2(\Omega_T)$ as $k\to \infty$.
			\end{enumerate}
		\end{corollary}

		\subsection{Uniqueness of solutions to the forward problem}

         In the end of this section, we show the uniqueness of solutions to NPMEs with absorption term.
         
		\begin{theorem}[Basic uniqueness result]
			\label{thm: basic uniqueness}
			Let $\Omega\subset\R^n$ be a bounded Lipschitz domain, $T>0$, $0<s<1$, $m> 1$ and $L_K\in \mathcal{L}_0$. If for $j=1,2$, we have given initial data $u_{0,j}\in L^{\infty}(\Omega)\cap \widetilde{H}^s(\Omega)$ and exterior conditions $\varphi_j\in C_c([0,T]\times \Omega_e)$ with $\Phi^m(\varphi_j)\in L^2(0,T;H^s(\R^n))$ and $u_j$ is a solution of
			\begin{equation}
				\label{eq: uniqueness FPME}
				\begin{cases}
					\partial_t u+ L_K(\Phi^m(u))=0 &\text{ in }\Omega_T,\\
					u=\varphi_j  & \text{ in }(\Omega_e)_T,\\
					u(0)=u_{0,j} &\text{ in }\Omega,
				\end{cases}
			\end{equation}
			then there holds
			\begin{equation}
				\label{eq: continuity estimate}
				\begin{split}
					&\quad \left\|u_1-u_2\right\|_{L^{\infty}(0,T;H^{-s}(\Omega))}+\left\|\mathcal{U}_1-\mathcal{U}_2\right\|_{L^{\infty}(0,T;H^s(\R^n))}\\
					&\leq C\LC \left\|u_{0,1}-u_{0,2}\right\|_{H^{-s}(\Omega)}+ \left\|\Phi^m(\varphi_1)-\Phi^m(\varphi_2)\right\|_{L^1(0,T;H^s(\R^n))}\RC ,
				\end{split}
			\end{equation}
			where $\mathcal{U}_j(x,t)=\int_0^t u_j(x,\tau)\,d\tau$. In particular, the constructed solutions in Theorem~\ref{thm: basic existence} are unique. 
		\end{theorem}
		
		\begin{remark}
			Note that the same proof works if we have additional sources $F_j\in L^2(0,T;H^{-s}(\Omega))$ in \eqref{eq: uniqueness FPME}, but then we would have an additional contribution $\|F_1-F_2\|_{L^1(0,T;H^{-s}(\Omega))}$ in the continuity estimate \eqref{eq: continuity estimate}.
		\end{remark}
		
		\begin{proof}[Proof of Theorem \ref{thm: basic uniqueness}]
			As in the previous proofs, we write $\Phi, \Phi_{\vareps}, L,B$ instead of $\Phi^m,\Phi^m_{\vareps},L_K,B_K$. By assumption the functions $v_j=u_j-\varphi_j$ solve 
			\begin{equation}
				\label{eq: uniqueness homogeneous FPME}
				\begin{cases}
					\partial_t v+ L(\Phi(v))=f_j &\text{ in }\Omega_T,\\
					u=0&\text{ in }(\Omega_e)_T,\\
					u(0)=u_{0,j}&\text{ in }\Omega,
				\end{cases}
			\end{equation}
			where $f_j\vcentcolon = -L(\Phi(\varphi_j))$, for $j=1,2$. Hence, by subtracting the weak formulations that there holds
			\begin{equation}
				\label{eq: weak form useful uniqueness}
				\begin{split}
					&-\int_{\Omega_T}(v_1-v_2)\partial_t\psi\,dxdt+\int_0^TB(\Phi(v_1)-\Phi(v_2),\psi)\,dt\\
					=&\int_{\Omega_T}(f_1-f_2)\psi\,dxdt+\int_{\Omega}(u_{0,1}-u_{0,2})\psi(0)\,dx
				\end{split}
			\end{equation}
			for all $\psi\in C_c^{\infty}([0,T)\times \Omega)$. Next, we define for fixed $0\leq t_0\leq T$ the function
			\begin{equation}
				\label{eq: help function 2}
				\Psi(x,t)\vcentcolon = \begin{cases}
					\int_t^{t_0}(\Phi(v_1)-\Phi(v_2))(x,\tau)\,d\tau&\quad\text{if}\ 0\leq t\leq t_0,\\
					0&\quad\text{otherwise}.
				\end{cases}
			\end{equation}
			Using $\Phi(v_j)\in L^2(0,T;\widetilde{H}^s(\Omega))$ and the fundamental theorem of calculus, one can check that $\Psi\in H^1(0,T;\widetilde{H}^s(\Omega))$ with $\Psi(T)=0$. Recalling Proposition~\ref{prop: density}, we see that $\Psi$ can be used as a test function in \eqref{eq: weak form useful uniqueness}. This gives
			\begin{equation}
				\label{eq: weak form useful uniqueness 2}
				\begin{split}
					&\int_{\Omega_{t_0}}(v_1-v_2)(\Phi(v_1)-\Phi(v_2))\,dxdt \\
					&+\int_0^{t_0}B\LC \Phi(v_1)-\Phi(v_2),\int_t^{t_0}(\Phi(v_1)-\Phi(v_2))(x,\tau)\,d\tau\RC dt\\
					=&\int_{\Omega_T}(f_1-f_2)\int_t^{t_0}(\Phi(v_1)-\Phi(v_2))(x,\tau)\,d\tau\,dxdt \\
					&+\int_{\Omega}\LC u_{0,1}-u_{0,2}\RC \int_0^{t_0}(\Phi(v_1)-\Phi(v_2))(x,\tau)\,d\tau\,dx.
				\end{split}
			\end{equation}
			Next, using Proposition~\ref{prop: density}, take a sequence $\psi_k\in C_c^{\infty}([0,T)\times\Omega)$ such that $\psi_k\to \Psi$ in $H^1(0,T;\widetilde{H}^s(\Omega))$. Since $\Psi(t_0)=0$, we can assume that $\psi_k(t)=0$ for $t_0\leq t\leq T$, then by dominated convergence, we have
			\begin{equation}
				\label{eq: bilinear form uniqueness}
				\begin{split}
					& \quad \int_0^{t_0}B \LC \Phi(v_1)-\Phi(v_2),\int_t^{t_0}(\Phi(v_1)-\Phi(v_2))(x,\tau)\,d\tau\RC dt \\
					&=-\int_0^{t_0}B(\partial_t \Psi,\Psi)\,dt =-\lim_{k\to\infty}\int_0^{t_0}B(\partial_t \psi_k,\psi_k)\,dt \\
					&=-\frac{1}{2}\lim_{k\to\infty}\int_0^{t_0}\partial_tB(\psi_k,\psi_k)\,dt= -\frac{1}{2}\lim_{k\to\infty}\left.B(\psi_k,\psi_k)\right|_{t=0}^{t=t_0}\\
					&=\frac{1}{2}\lim_{k\to\infty}B(\psi_k(0),\psi_k(0))=\frac{1}{2}B(\Psi(0),\Psi(0))\\
					&=\frac{1}{2}B\LC \int_0^{t_0}(\Phi(v_1)-\Phi(v_2))\,d\tau,\int_0^{t_0}(\Phi(v_1)-\Phi(v_2))\,d\tau\RC ,
				\end{split}
			\end{equation}
			where we used the Sobolev embeddding $H^1(0,T;\widetilde{H}^s(\Omega))\hookrightarrow C([0,T];\widetilde{H}^s(\Omega))$. Additionally, by the monotonicity of $\Phi$ we have
			\begin{equation}
				\label{eq: monotonicity for uniqueness}
				\int_{\Omega_{t_0}}(\Phi(v_1)-\Phi(v_2))(v_1-v_2)\,dxdt\geq 0.
			\end{equation}
			From now on we use the notation
			\[
			\mathcal{V}_j(x,t):=\int_0^t \Phi(v_j(x,\tau))\,d\tau,
			\]
			for $j=1,2$.
			Hence, using \eqref{eq: bilinear form uniqueness}, \eqref{eq: monotonicity for uniqueness}, uniform ellipticity of $K$ and the fractional Poincar\'e inequality, we deduce from \eqref{eq: weak form useful uniqueness 2} the estimate
			\begin{equation}
				\label{eq: estimate uniqueness nonlin}
				\begin{split}
					&\quad \|(\mathcal{V}_1-\mathcal{V}_2)(t_0)\|_{H^s(\R^n)}^2 \\
					&=\left\|\int_0^{t_0}(\Phi(v_1)-\Phi(v_2))\,d\tau\right\|_{H^s(\R^n)}^2\\
					&\leq CB \LC \int_0^{t_0}(\Phi(v_1)-\Phi(v_2))\,d\tau,\int_0^{t_0}(\Phi(v_1)-\Phi(v_2))\,d\tau\RC \\
					&\leq C\left\{\int_{\Omega_{t_0}}(v_1-v_2)(\Phi(v_1)-\Phi(v_2))\,dxdt \right. \\
                    &\quad  \qquad \left.+\int_0^{t_0}B\LC \Phi(v_1)-\Phi(v_2),\int_t^{t_0}(\Phi(v_1)-\Phi(v_2))(x,\tau)\,d\tau\RC dt\right\}\\
					&=C\left\{\int_{\Omega_T}(f_1-f_2)\LC \int_t^{t_0}(\Phi(v_1)-\Phi(v_2))(x,\tau)\,d\tau\RC dxdt\right. \\
					&\quad \qquad \left. +\int_{\Omega}(u_{0,1}-u_{0,2})\LC \int_0^{t_0}(\Phi(v_1)-\Phi(v_2))(x,\tau)\,d\tau \RC dx\right\}\\
					&\leq  C\left\{\|f_1-f_2\|_{L^1(0,T;H^{-s}(\Omega))}\left\|\int_t^{t_0}(\Phi(v_1)-\Phi(v_2))\,d\tau\right\|_{L^{\infty}(0,t_0;H^s(\R^n))}\right.\\
					&\quad \qquad \left. + \|u_{0,1}-u_{0,2}\|_{H^{-s}(\Omega))}\left\|\int_0^{t_0}(\Phi(v_1)-\Phi(v_2))\,d\tau\right\|_{H^s(\R^n)}\right\}.
				\end{split}
			\end{equation}
			Observe that we have
			\begin{equation}
				\label{eq: time segment integral}
				\begin{split}
					&\quad \left\|\int_t^{t_0}(\Phi(v_1)-\Phi(v_2))\,d\tau\right\|_{L^{\infty}(0,t_0;H^s(\R^n))}\\
					&=\|(\mathcal{V}_1-\mathcal{V}_2)(t_0)-(\mathcal{V}_1-\mathcal{V}_2)(t)\|_{L^{\infty}(0,t_0;H^s(\R^n))}\\
					&\leq \|(\mathcal{V}_1-\mathcal{V}_2)(t_0)\|_{H^s(\R^n)}+\|\mathcal{V}_1-\mathcal{V}_2\|_{L^{\infty}(0,t_0;H^s(\R^n))}\\
					&\leq 2\|\mathcal{V}_1-\mathcal{V}_2\|_{L^{\infty}(0,t_0;H^s(\R^n))}.
				\end{split}
			\end{equation}
			Thus, \eqref{eq: estimate uniqueness nonlin} simplifies to
			\[
			\begin{split}
				&\quad \left\|(\mathcal{V}_1-\mathcal{V}_2)(t_0) \right\|_{H^s(\R^n)}^2\\
				&\leq C \LC \|f_1-f_2\|_{L^1(0,t_0;H^{-s}(\Omega))}+ \|u_{0,1}-u_{0,2}\|_{H^{-s}(\Omega))}\RC \left\|\mathcal{V}_1-\mathcal{V}_2 \right\|_{L^{\infty}(0,t_0;H^s(\R^n))}\\
				&\leq C \LC \|f_1-f_2\|_{L^1(0,T;H^{-s}(\Omega))}+ \|u_{0,1}-u_{0,2}\|_{H^{-s}(\Omega))}\RC \left\|\mathcal{V}_1-\mathcal{V}_2\right\|_{L^{\infty}(0,T;H^s(\R^n))}.
			\end{split}
			\]
			Taking the supremum in $t_0\in (0,T)$ and absorbing the last factor on the left hand side, we obtain
			\begin{equation}
				\label{eq: estimate nonlinear term}
    \begin{split}
       & \quad  \|\mathcal{V}_1-\mathcal{V}_2\|_{L^{\infty}(0,T;H^s(\R^n))}\\
       &\leq  C\LC \|f_1-f_2\|_{L^1(0,T;H^{-s}(\Omega))}+ \|u_{0,1}-u_{0,2}\|_{H^{-s}(\Omega))}\RC .
    \end{split}
			\end{equation}
			Let us point out that the PDE \eqref{eq: uniqueness homogeneous FPME} guarantees $\partial_t v_j\in L^2(0,T;H^{-s}(\Omega))$ and $v_j\in L^2(0,T;H^{-s}(\Omega))$. The first assertion is immediate.

            To see the second one, let $\eta\in C^{\infty}_c(\Omega_T)$ and insert $\psi(x,t)=-\int_t^T\eta(x,\tau)\,d\tau$ into the weak formulation of \eqref{eq: uniqueness homogeneous FPME} to obtain
			\[
			\begin{split}
				\left|\int_{\Omega_T}v_j\eta\,dxdt\right|&=\left|-\int_{\Omega_T}v_j\partial_t \psi\,dxdt\right|\\
				&=\left|-\int_0^T B(\Phi(v_j),\psi)\,dt+\int_0^T f_j\psi\,dxdt+\int_{\Omega}u_{0,j}\psi(0)\,dx\right|\\
				&\leq C\left( [\Phi(v_j)]_{L^2(0,T;H^s(\R^n))}[\psi]_{L^2(0,T;H^s(\R^n))}\right. \\
				& \quad \qquad  +\|f_j\|_{L^2(0,T;H^{-s}(\Omega))}\|\psi\|_{L^2(0,T;H^s(\R^n))}\\
                &\left.\quad \qquad +\|u_{0,j}\|_{L^2(\Omega)}\|\psi(0)\|_{L^2(\Omega)}\right).
			\end{split}
			\]
			By Jensen's inequality, we have 
            \[
                \|\psi\|_{L^2(\R^n_T)}\leq  C\|\eta\|_{L^2(\Omega_T)}, \quad \|\psi(0)\|_{L^2(\Omega)}\leq C\|\eta\|_{L^2(\Omega_T)}
            \]
            and
			\[
			\begin{split}
				[\psi]^2_{H^s(\R^n)}&\leq \int_{\R^{2n}}\frac{\left|\int_t^T(\eta(x,\tau)-\eta(y,\tau))\,d\tau\right|^2}{|x-y|^{n+2s}}\,dxdy\leq C[\eta]_{L^2(0,T;H^s(\R^n))}^2.
			\end{split}
			\]
			Thus, we get 
			\[
			\left|\int_{\Omega_T}v_j\eta\,dxdt\right|\leq C\|\eta\|_{L^2(0,T;H^s(\R^n))},
			\]
			but this means nothing else than $v_j\in L^2(0,T;H^{-s}(\Omega))$ as $C_c^{\infty}(\Omega_T)$ is dense in $L^2(0,T;\widetilde{H}^s(\Omega))$. By the Sobolev embedding we have $v_j\in C([0,T];H^{-s}(\Omega))$ and hence the fundamental theorem of calculus, \eqref{eq: uniqueness homogeneous FPME}, the uniform ellipticity of $K$ and the estimate \eqref{eq: estimate nonlinear term} give
			\begin{equation}
				\label{eq: control of function}
				\begin{split}
					&\quad \|(v_1-v_2)(t)\|_{H^{-s}(\Omega)} \\
					&\leq \left\|\int_0^t\partial_t (v_1-v_2)(\tau)\,d\tau\right\|_{H^{-s}(\Omega)}+\|(v_1-v_2)(0)\|_{H^{-s}(\Omega)}\\
					&\leq  \sup_{\psi\in\widetilde{H}^s(\Omega):\,\|\psi\|_{H^s(\R^n)}\leq 1}\left|\int_0^t\langle \partial_t(v_1-v_2)(\tau),\psi\rangle\,d\tau\right|+\|u_{0,1}-u_{0,2}\|_{H^{-s}(\Omega)}\\
					&\leq  \sup_{\psi\in\widetilde{H}^s(\Omega):\,\|\psi\|_{H^s(\R^n)}\leq 1}\left|-\int_0^t B(\Phi(v_1)-\Phi(v_2),\psi)\,d\tau+\int_{\Omega_t}(f_1-f_2)\psi\,dxd\tau\right|\\
					&\qquad +\|u_{0,1}-u_{0,2}\|_{H^{-s}(\Omega)}\\
					&\leq  \sup_{\psi\in\widetilde{H}^s(\Omega):\,\|\psi\|_{H^s(\R^n)}\leq 1} \left|B\LC (\mathcal{V}_1-\mathcal{V}_2)(t),\psi\RC \right|\\
					& \quad +\|f_1-f_2\|_{L^1(0,T;H^{-s}(\Omega))}+\|u_{0,1}-u_{0,2}\|_{H^{-s}(\Omega)}\\
					&\leq  C\LC \|f_1-f_2\|_{L^1(0,T;H^{-s}(\Omega))}+ \|u_{0,1}-u_{0,2}\|_{H^{-s}(\Omega))}\RC, 
				\end{split}
			\end{equation}
			for all $0<t\leq T$. Next, let us observe that by uniform ellipticity of $K$ there holds
			\begin{equation}
				\label{eq: estimate source}
				\begin{split}
					\|f_1-f_2\|_{L^1(0,T;H^{-s}(\Omega))}&\leq \|L(\Phi(\varphi_1)-\Phi(\varphi_2))\|_{L^1(0,T;H^{-s}(\Omega))}\\
					&=\int_0^T\sup_{\psi\in \widetilde{H}^s(\Omega):\ \|\psi\|_{H^s(\R^n)}\leq 1}|B(\Phi(\varphi_1)-\Phi(\varphi_2),\psi)|\,dt\\
					&\leq C\|\Phi(\varphi_1)-\Phi(\varphi_2)\|_{L^1(0,T;H^s(\R^n))}.
				\end{split}
			\end{equation}
			Combining \eqref{eq: estimate source} with \eqref{eq: control of function} and \eqref{eq: estimate nonlinear term}, we have proved
			\begin{equation}
				\label{eq: for homogeneous equation required estimate}
				\begin{split}
					&\quad \|v_1-v_2\|_{L^{\infty}(0,T;H^{-s}(\Omega))}+\|\mathcal{V}_1-\mathcal{V}_2\|_{L^{\infty}(0,T;H^s(\R^n))}\\
					&\leq C\LC \|\Phi(\varphi_1)-\Phi(\varphi_2)\|_{L^1(0,T;H^s(\R^n))}+\|u_{0,1}-u_{0,2}\|_{H^{-s}(\Omega))}\RC.
				\end{split}
			\end{equation}
			Next, we go back to the variables $u_j=v_j+\varphi_j$. Let us note that
			\begin{equation}
				\label{eq: going back to usual variables}
				\begin{split}
					\mathcal{U}_1-\mathcal{U}_2&=\int_0^t(\Phi(u_1)-\Phi(u_2))\,d\tau\\
					&=\int_0^t(\Phi(v_1+\varphi_1)-\Phi(v_2+\varphi_2))\,d\tau\\
					&=\mathcal{V}_1-\mathcal{V}_2+\int_0^t(\Phi(\varphi_1)-\Phi(\varphi_2))\,d\tau,
				\end{split}
			\end{equation}
			where we used that $v_j$ and $\varphi_j$ have disjoint supports. Hence, inserting \eqref{eq: going back to usual variables} into \eqref{eq: for homogeneous equation required estimate} and $\left. \varphi_j\right|_{\Omega_T}=0$, we get 
			\begin{equation}
				\label{eq: for nonhomogeneous equation required estimate}
				\begin{split}
					&\quad \|u_1-u_2\|_{L^{\infty}(0,T;H^{-s}(\Omega))}+\|\mathcal{U}_1-\mathcal{U}_2\|_{L^{\infty}(0,T;H^s(\R^n))}\\
					&\leq  C\bigg(\|\Phi(\varphi_1)-\Phi(\varphi_2)\|_{L^1(0,T;H^s(\R^n))}+\|u_{0,1}-u_{0,2}\|_{H^{-s}(\Omega))}\\
					& \quad \qquad +\|\varphi_1-\varphi_2\|_{L^{\infty}(0,T;H^{-s}(\Omega))}+\left\|\int_0^t(\Phi(\varphi_1)-\Phi(\varphi_2))\,d\tau\right\|_{L^{\infty}(0,T;H^s(\R^n))}\bigg)\\
					&\leq  C\LC \|\Phi(\varphi_1)-\Phi(\varphi_2)\|_{L^1(0,T;H^s(\R^n))}+\|u_{0,1}-u_{0,2}\|_{H^{-s}(\Omega))}\RC .
				\end{split}
			\end{equation}
			Hence, we can conclude the proof.
		\end{proof}
		
		\begin{theorem}[Uniqueness with linear absorption term]
			\label{Uniquenerss with lin absorption term}
			Let $\Omega\subset\R^n$ be a bounded Lipschitz domain, $T>0$, $0<s<\alpha\leq 1$, $m> 1$ and $L_K\in \mathcal{L}_0$. Assume additionally that we have given $\rho,q\in C^{1,\alpha}_+(\R^n)$ with $\rho$ uniformly elliptic. Suppose that $u_{0}\in L^{\infty}(\Omega)\cap \widetilde{H}^s(\Omega)$, $\varphi\in C_c([0,T]\times \Omega_e)$ with $\Phi^m(\varphi)\in L^2(0,T;H^s(\R^n))$ and $u_1,u_2$ solve
			\begin{equation}
				\label{eq: uniqueness FPME0}
				\begin{cases}
					\rho\partial_t u+ L_K(\Phi^m(u))+qu=0 &\text{ in }\Omega_T,\\
					u=\varphi  & \text{ in }(\Omega_e)_T,\\
					u(0)=u_{0} &\text{ in }\Omega,
				\end{cases}
			\end{equation}
			then there holds $u_1=u_2$. 
		\end{theorem}
		
		\begin{proof}
			Like in the proofs above we write $\Phi,B,L$ for $\Phi^m,B_K,L_K$. As in the proof of Theorem~\ref{thm: basic uniqueness}, we go over to the variables $v_j=u_j-\varphi$. First observe that $v_1-v_2$ satisfies 
			\begin{equation}
				\label{eq: uniqueness FPME 2}
				\begin{cases}
					\rho\partial_t (v_1-v_2)+ L_K(\Phi(v_1)-\Phi(v_2))+q(v_1-v_2)=0 &\text{ in }\Omega_T,\\
					v_1-v_2=0  & \text{ in }(\Omega_e)_T,\\
					(v_1-v_2)(0)=0 &\text{ in }\Omega
				\end{cases}
			\end{equation}
			and hence there holds
			\begin{equation}
				\label{eq: weak form uniqueness lin abs}
				-\int_{\Omega_T}(v_1-v_2)\partial_t \psi\,dxdt+\int_0^T B(\Phi(v_1)-\Phi(v_2),\psi)\,dt+\int_{\Omega_T}q(v_1-v_2)\,dxdt=0
			\end{equation}
			for all $\psi\in C_c^{\infty}([0,T)\times \Omega)$. We can again use the function $\Psi$, as defined in \eqref{eq: help function 2}, as a test function in \eqref{eq: weak form uniqueness lin abs}. Using \eqref{eq: bilinear form uniqueness}, we obtain
			\begin{equation}
				\label{eq: usefuld eq uniqueness linear abs}
				\begin{split}
					&\int_{\Omega_{t_0}}\rho(v_1-v_2)(\Phi(v_1)-\Phi(v_2))\,dxdt+\frac{1}{2}B((\mathcal{V}_1-\mathcal{V}_2)(t_0),(\mathcal{V}_1-\mathcal{V}_2)(t_0))\\
					&+\int_{\Omega_{t_0}}q(v_1-v_2)\int_t^{t_0}\Phi(v_1)-\Phi(v_2)\,d\tau\,dxdt=0,
				\end{split}
			\end{equation}
			where we used the notation $\mathcal{V}_j(t)=\int_0^t \Phi(v_j)\,d\tau$. Next, note that the monotonicity of $\Phi$ and $\rho\geq 0$ implies
			\begin{equation}
				\label{eq: nonneg uniqueness linear abs}
				\int_{\Omega_{t_0}}\rho(\Phi(v_1)-\Phi(v_2))(v_1-v_2)\,dxdt\geq 0.
			\end{equation}
			Now by using \eqref{eq: nonneg uniqueness linear abs}, the fractional Poincar\'e inequality, the uniform ellipticity of $K$, $q\in C^{0,\alpha}(\R^n)$, \eqref{eq: time segment integral} and Remark~\ref{remark: hoelder coeff estimate}, we deduce that
			\begin{equation}
				\begin{split}
					&\quad \|(\mathcal{V}_1-\mathcal{V}_2)(t_0)\|_{H^s(\R^n)}^2 \\
					&=\left\|\int_0^{t_0}(\Phi(v_1)-\Phi(v_2))\,d\tau\right\|_{H^s(\R^n)}^2\\
					&\leq  CB\left(\int_0^{t_0}(\Phi(v_1)-\Phi(v_2))\,d\tau,\int_0^{t_0}(\Phi(v_1)-\Phi(v_2))\,d\tau\right)\\
					&\leq  C\left(\int_{\Omega_{t_0}}\rho (v_1-v_2)(\Phi(v_1)-\Phi(v_2))\,dxdt \right. \\ &\quad \qquad \left. +\frac{1}{2}\int_0^{t_0}B(\Phi(v_1)-\Phi(v_2),\int_t^{t_0}(\Phi(v_1)-\Phi(v_2))(x,\tau)\,d\tau)\,dt\right)\\
					&=-C\int_{\Omega_{t_0}}q(v_1-v_2)\int_{t}^{t_0}(\Phi(v_1)-\Phi(v_2))\,d\tau dxdt\\
					&\leq C\|q\|_{C^{0,\alpha}(\R^n)}\|v_1-v_2\|_{L^1(0,t_0;H^{-s}(\Omega))}\left\|\int_{t}^{t_0}(\Phi(v_1)-\Phi(v_2))\,d\tau\right\|_{L^{\infty}(0,t_0;H^s(\R^n))}\\
					&\leq C\|q\|_{C^{0,\alpha}(\R^n)}\|v_1-v_2\|_{L^1(0,t_0;H^{-s}(\Omega))}\|\mathcal{V}_1-\mathcal{V}_2\|_{L^{\infty}(0,t_0;H^s(\R^n))}.
				\end{split}
			\end{equation}
			Next, let $0<\widetilde{T}\leq T$ be a given constant, which we will fix later. The previous estimate then shows
			\begin{equation}
				\label{eq: first estimate uniqueness lin abs}
				\|\mathcal{V}_1-\mathcal{V}_2\|_{L^{\infty}(0,\widetilde{T},H^s(\R^n))}\leq C\widetilde{T}\|q\|_{C^{0,\alpha}(\R^n)}\|v_1-v_2\|_{L^{\infty}(0,\widetilde{T};H^{-s}(\Omega))}.
			\end{equation}
			On the other hand, following the computation in \eqref{eq: control of function} and use $\rho\in C^{0,\alpha}(\R^n)$, Remark~\ref{remark: hoelder coeff estimate}, \cite[Lemma~4.1]{CRTZ-2022} and the uniform ellipticity of $K$, we find
			\[
			\begin{split}
				&\quad \|(v_1-v_2)(t_0)\|_{H^{-s}(\Omega)} \\
				&\leq \left\|\int_0^{t_0}\partial_t (v_1-v_2)(\tau)\,d\tau\right\|_{H^{-s}(\Omega)}\\
				&\leq \sup_{\psi\in\widetilde{H}^s(\Omega):\,\|\psi\|_{H^s(\R^n)}\leq 1}\left|\int_0^{t_0}\left\langle \partial_t(v_1-v_2)(\tau),\psi\right\rangle 
 d\tau\right|\\
				&\leq \sup_{\psi\in\widetilde{H}^s(\Omega):\,\|\psi\|_{H^s(\R^n)}\leq 1}\left|\int_0^{t_0}\left\langle \rho\partial_t(v_1-v_2)(\tau),\psi/\rho \right\rangle d\tau\right|\\
				&=\sup_{\psi\in\widetilde{H}^s(\Omega):\,\|\psi\|_{H^s(\R^n)}\leq 1}\left|\int_0^{t_0}B(\Phi(v_1)-\Phi(v_2),\psi/\rho)\,d\tau +\int_{\Omega_{t_0}}q(v_1-v_2)\psi/\rho\,dxd\tau\right|\\
				&\leq C\left(\left\|(\mathcal{V}_1-\mathcal{V}_2)(t_0)\right\|_{H^s(\R^n)}+\|q\|_{C^{0,\alpha}(\R^n)}\|v_1-v_2\|_{L^{1}(0,t_0;H^{-s}(\Omega))}\right),
			\end{split}
			\]
			where the constant $C>0$ only depends on $\|K\|_{L^{\infty}(\R^n)}$, the lower bound of $\rho$ and $\|\rho\|_{C^{0,\alpha}(\R^n)}$. Arguing as above this gives
			\begin{equation}
				\label{eq: second estimate uniqueness lin abs}
				\begin{split}
					\|v_1-v_2\|_{L^{\infty}(0,\widetilde{T};H^{-s}(\Omega))}&\leq C\left\|\mathcal{V}_1-\mathcal{V}_2\right\|_{L^{\infty}(0,\widetilde{T},H^s(\R^n))}\\
					&\quad +C\widetilde{T}\|q\|_{C^{0,\alpha}(\R^n)}\|v_1-v_2\|_{L^{\infty}(0,\widetilde{T};H^{-s}(\Omega))}.
				\end{split}
			\end{equation}
			Inserting \eqref{eq: first estimate uniqueness lin abs} into \eqref{eq: second estimate uniqueness lin abs}, we get
			\begin{equation}
				\label{eq: third estimate uniqueness lin abs}
				\|v_1-v_2\|_{L^{\infty}(0,\widetilde{T};H^{-s}(\Omega))}\leq C\widetilde{T}\|q\|_{C^{0,\alpha}(\R^n)}\|v_1-v_2\|_{L^{\infty}(0,\widetilde{T};H^{-s}(\Omega))}.
			\end{equation}
			Hence, by choosing $\widetilde{T}$ sufficiently small such that $C\widetilde{T}\|q\|_{C^{0,\alpha}(\R^n)}\leq 1/2$, \eqref{eq: third estimate uniqueness lin abs} shows $v_1=v_2$ for a.e. $0\leq t\leq \widetilde{T}$. Since $v_1-v_2\in C([0,T];H^{-s}(\Omega))$ is continuous, we know $(v_1-v_2)(\widetilde{T})=0$. Hence, we can repeat our local uniqueness result and finally find $v_1=v_2$ on $[0,T]$. This completes the proof.
		\end{proof}

		\section{Comparison principle}
		\label{sec: comparison principle}
		
		In this section, we show the comparison principle for NPMEs.

		\begin{theorem}[Basic comparison principle]
			\label{thm: basic comparison}
			Let $\Omega\subset\R^n$ be a bounded Lipschitz domain, $T>0$, $0<s<1$, $m> 1$ and $L_K\in \mathcal{L}_0$. Assume that for $j=1,2$ we have given initial data $u_{0,j}\in L^{\infty}(\Omega)\cap\widetilde{H}^s(\Omega)$, exterior conditions $\varphi_j\in C_c([0,T]\times \Omega_e)$ with $\Phi^m(\varphi_j)\in L^2(0,T;H^s(\R^n))$, sources $F_j\in L^2(\Omega_T)$ and $u_j$ solves
			\begin{equation}
				\label{eq: comparison FPME}
				\begin{cases}
					\partial_t u+ L_K(\Phi^m(u))=F_j&\text{ in }\Omega_T,\\
					u=\varphi_j&\text{ in }(\Omega_e)_T,\\
					u(0)=u_{0,j}&\text{ in }\Omega.
				\end{cases}
			\end{equation}
			Additionally, suppose that there exist sequences $\LC F_j^{\vareps}\RC _{\vareps>0}\subset L^2(\Omega_T)$, $\LC \varphi_j^{\vareps}\RC _{\vareps>0}\subset C_c([0,T]\times \Omega_e)$ with $\Phi^m_{\vareps}(\varphi_j^{\vareps})\in L^2(0,T;H^s(\R^n))$, $\LC u_{0,j}^{\vareps}\RC _{\vareps>0}\subset L^{\infty}(\Omega)\cap\widetilde{H}^s(\Omega)$ and $\LC u_{j,\vareps}\RC _{\vareps>0}\subset H^1(0,T;L^2(\Omega))\cap L^2(0,T;H^s(\R^n))$ satisfying
            \begin{enumerate}[(i)]
                \item\label{prop 1 comparison} $\liminf_{\vareps\to 0}\int_{\Omega_T}\LC F_1^{\vareps}-F_2^{\vareps}\RC_+ \, dxdt\leq \int_{\Omega_T}\LC F_1-F_2\RC _+\,dxdt$,
                \item\label{prop 2 comparison} $\liminf_{\vareps\to 0}\int_{\Omega_{T}}\int_{\Omega_e}\frac{\left(\Phi^m_{\vareps}(\varphi_1^{\vareps})-\Phi^m_{\vareps}(\varphi_2^{\vareps})\right)_+(x,t)}{|x-y|^{n+2s}}\,dxdydt$\\
                $\leq $ $\int_{\Omega_{T}}\int_{\Omega_e}\frac{\left(\Phi^m(\varphi_1)-\Phi^m(\varphi_2)\right)_+(x,t)}{|x-y|^{n+2s}}\,dxdydt$,
                \item\label{prop 3 comparison} $u_{0,j}^{\vareps}\to u_{0,j}$ in $L^1(\Omega)$ as $\vareps\to 0$,
                \item\label{prop 4 comparison} $u_{j,\vareps}\weak u_j\ \text{in}\ L^1(\Omega_T)$ as $\vareps\to 0$
                \item\label{prop 5 comparison} and $u_{j,\vareps}$ are solutions of
			\begin{equation}
				\begin{cases}
					\partial_t u+ L_K(\Phi_{\vareps}^m(u))=F^{\vareps}_j&\text{ in }\Omega_T,\\
					u=\varphi_j^{\vareps}&\text{ in }(\Omega_e)_T,\\
					u(0)=u_{0,j}^{\vareps}&\text{ in }\Omega.
				\end{cases}
			\end{equation}
            \end{enumerate}
			Then there exists a constant $C>0$ independent of these solutions, initial data, boundary data and sources such that 
			\begin{equation}
				\label{eq: comparison}
				\begin{split}
					\int_{\Omega}\left(u_{1}-u_{2} \right)_+(x,t_0)\,dx&\leq  C\left( \int_{\Omega_T}\left(F_1-F_2\right)_+\,dxdt+\int_{\Omega}\left(u_{0,1}-u_{0,2}\right)_+\,dx\right.\\
					&\quad \qquad +\left.\int_{\Omega_{T}}\int_{\Omega_e}\frac{\left(\Phi^m(\varphi_1)-\Phi^m(\varphi_2)\right)_+(x,t)}{|x-y|^{n+2s}}\,dxdydt\right).
				\end{split}
			\end{equation}
		\end{theorem}
		\begin{remark}
			If one can know that $u_{j,\vareps}\to u_j$ strongly in $L^2(\Omega)$ as $\vareps\to 0$ for a.e. $0<t<T$, one can replace the left hand side of \eqref{eq: comparison} by $\int_{\Omega}\LC u_1(x,t)-u_2(x,t)\RC dx$.
		\end{remark}
		
		\begin{proof}[Proof of Theorem \ref{thm: basic comparison}]
			As in the existence and uniqueness proofs above, we write $\Phi, \Phi_{\vareps}, L,B$ instead of $\Phi^m,\Phi^m_{\vareps},L_K,B_K$. Next, let us note that $v_{j,\vareps}\vcentcolon = u_{j,\vareps}-\varphi_j^{\vareps}$ solves
			\begin{equation}
				\begin{cases}
					\partial_t v+ L(\Phi_{\vareps}(v))=f_{j,\vareps}+F_j^{\vareps}&\text{ in }\Omega_T,\\
					v=0&\text{ in }(\Omega_e)_T,\\
					v(0)=u_{0,j}^{\vareps}&\text{ in }\Omega,
				\end{cases}
			\end{equation}
			where we set $f_{j,\vareps}=-L(\Phi_{\vareps}(\varphi_j^{\vareps}))\in L^2(\Omega_T)$. Hence, $v_{1,\vareps}-v_{2,\vareps}$ solves
			\begin{equation}
				\label{eq: difference equation for comparison}
				\begin{cases}
					\partial_t (v_{1,\vareps}-v_{2,\vareps}) + L(\Phi_{\vareps}(v_{1,\vareps})-\Phi_{\vareps}(v_{2,\vareps}))=(f_{1,\vareps}-f_{2,\vareps})+(F_1^{\vareps}-F_2^{\vareps})&\text{ in }\Omega_T,\\
					v_{1,\vareps}-v_{2,\vareps}=0&\text{ in }(\Omega_e)_T,\\
					(v_{1,\vareps}-v_{2,\vareps})(0)=u_{0,1}^{\vareps}-u_{0,2}^{\vareps}&\text{ in }\Omega.
				\end{cases}
			\end{equation}
			By assumption, the definition of $\Phi_{\vareps}$ and $\partial\Omega\in C^{0,1}$, we have 
			\[
			v_{j,\vareps}\in H^1(0,T;L^2(\Omega))\cap L^2(0,T;\widetilde{H}^s(\Omega)).
			\]
			Thus, there holds
			\begin{equation}
				\label{eq: weak formulation comparison difference}
				\begin{split}
					\int_{\Omega_T}\partial_t \LC v_{1,\vareps}-v_{2,\vareps}\RC\psi\,dxdt+\int_0^TB(\Phi_{\vareps}(v_{1,\vareps})-\Phi_{\vareps}(v_{2,\vareps}),\psi)\,dt=\int_{\Omega_T}G_{\vareps}\psi\,dxdt,
				\end{split}
			\end{equation}
			for all $\psi\in L^2(0,T;\widetilde{H}^s(\Omega))$, where we defined $G_{\vareps}\vcentcolon = f_{1,\vareps}-f_{2,\vareps}+F_1^{\vareps}-F_2^{\vareps}$.

             Let $\chi\colon\R\to\R$ be the step function
			\[
			\chi(t)=\begin{cases}
				1&\ \text{if}\ t>0,\\
				0& \ \text{if} \ t\leq 0.
			\end{cases}
			\]
			We want to use $\psi=\chi \LC \Phi_{\vareps}(v_{1,\vareps})-\Phi_{\vareps}(v_{2,\vareps})\RC $ as a test function  in the weak formulation \eqref{eq: weak formulation comparison difference}. Unfortunately, the test function $\psi$ does not have the right regularity properties in \eqref{eq: weak formulation comparison difference}. Hence, in order to remedy this, let us introduce the auxiliary functions $\chi_{\delta}\in C^{\infty}(\R)$ satisfying
			\[
			0\leq \chi'_{\delta}\leq \frac{2}{\delta}\quad \text{and}\quad
			\chi_{\delta}(t)=\begin{cases}
				1,&\ \text{if}\ t\geq \delta\\
				0,& \ \text{if} \ t\leq 0
			\end{cases}
			\]
			for $\delta>0$. One can easily verify that $\chi_{\delta}\LC \Phi_{\vareps}(v_{1,\vareps})-\Phi_{\vareps}(v_{2,\vareps})\RC \in L^2(0,T;\widetilde{H}^s(\Omega))$ by the Lipschitz continuity of $\chi_\delta$, then we get 
			\begin{equation}
				\label{eq: space regularity}
				\begin{split}
					&\left\|\Phi_{\vareps}(v_{1,\vareps})-\Phi_{\vareps}(v_{2,\vareps})\right\|_{L^2(0,T;H^s(\R^n))}\leq C_{\delta} \left\|\Phi_{\vareps}(v_{1,\vareps})-\Phi_{\vareps}(v_{2,\vareps})\right\|_{L^2(0,T;H^s(\R^n))},
				\end{split}
			\end{equation}
			for some constant $C_{\delta}>0$ depending on $\delta>0$. Let us define
			\[
			\psi_{\delta,k}(x,t):=\chi_{\delta}\LC \Phi_{\vareps}(v_{1,\vareps})-\Phi_{\vareps}(v_{2,\vareps})\RC (x,t)\eta_k(t),
			\]
			where $\eta_k\in C_c^{\infty}([0,T))$ satisfies
			\[
			0\leq \eta_k\leq 1, \ \left. \eta_k\right|_{[0,t_0]}=1,\ \left. \eta_{k}\right|_{[t_0+1/k,T)}=0,
			\]
			for $0<t_0<T$ and $k\in\N$ sufficiently large. Then $\psi_{\delta,k}$ can be viewed as a test function in \eqref{eq: weak formulation comparison difference} and we get
			\begin{equation}
				\label{eq: useful estimate comparison}
				\begin{split}
					&\int_{\Omega_T}\partial_t(v_{1,\vareps}-v_{2,\vareps})\chi_{\delta}\LC \Phi_{\vareps}(v_{1,\vareps})-\Phi_{\vareps}(v_{2,\vareps})\RC \eta_k\,dxdt\\
					=&-\int_0^TB \LC \Phi_{\vareps}(v_{1,\vareps})-\Phi_{\vareps}(v_{2,\vareps}),\chi_{\delta}(\Phi_{\vareps}(v_{1,\vareps})-\Phi_{\vareps}(v_{2,\vareps}))\RC \eta_k\,dt\\
					&+\int_{\Omega_T}G_{\vareps}\chi_{\delta}\LC \Phi_{\vareps}(v_{1,\vareps})-\Phi_{\vareps}(v_{2,\vareps})\RC\eta_k\,dxdt.
				\end{split}
			\end{equation}

           In the rest of the proof, let us define
			\[
			\mathcal{V}_{\vareps}:=\Phi_{\vareps}(v_{1,\vareps})-\Phi_{\vareps}(v_{2,\vareps})
			\]
           to simplify the notation. We want to show that the term involving the bilinear form in \eqref{eq: useful estimate comparison} is nonnegative. In fact, by the usual trick as used before, we have 
			\begin{equation}
				\label{eq: nonnegativity of bilinear form}
				\begin{split}
					&\quad B\LC \mathcal{V}_{\vareps},\chi_{\delta}\mathcal{V}_{\vareps}\RC \\
					&=\int_{\R^{2n}}K(x,y)\frac{\LC \mathcal{V}_{\vareps}(x)-\mathcal{V}_{\vareps}(y)\RC \LC\chi_{\delta}(\mathcal{V}_{\vareps}(x))-\chi_{\delta}(\mathcal{V}_{\vareps}(y))\RC }{|x-y|^{n+2s}}\,dxdy\\
					&=\int_{\R^{2n}}K(x,y)\left(\int_{0}^1\chi'_{\delta}(\mathcal{V}_{\vareps}(y)+\tau(\mathcal{V}_{\vareps}(x)-\mathcal{V}_{\vareps}(y)))\,d\tau\right)\frac{(\mathcal{V}_{\vareps}(x)-\mathcal{V}_{\vareps}(y))^2}{|x-y|^{n+2s}}\,dxdy\\
					&\geq  0.
				\end{split}
			\end{equation}
			Plugging \eqref{eq: nonnegativity of bilinear form} into \eqref{eq: useful estimate comparison}, we arrive at the estimate
			\begin{equation}
				\label{eq: useful estimate in comparison principle0}
				\begin{split}
					&\quad \int_{\Omega_T}\partial_t(v_{1,\vareps}-v_{2,\vareps})\chi_{\delta}\LC \Phi_{\vareps}(v_{1,\vareps})-\Phi_{\vareps}(v_{2,\vareps})\RC \eta_k\,dxdt\\
					&\leq \int_{\Omega_T}G_{\vareps}\chi_{\delta}\LC \Phi_{\vareps}(v_{1,\vareps})-\Phi_{\vareps}(v_{2,\vareps})\RC\eta_k\,dxdt.
				\end{split}
			\end{equation}
			Hence, passing to the limit $\delta\to 0$ by Lebesgue's dominated convergence theorem on \eqref{eq: useful estimate in comparison principle0}, this shows
			\begin{equation}
				\label{eq: useful estimate in comparison principle}
				\begin{split}
					&\quad \int_{\Omega_T}\partial_t(v_{1,\vareps}-v_{2,\vareps})\chi(\Phi_{\vareps}(v_{1,\vareps})-\Phi_{\vareps}(v_{2,\vareps}))\eta_k\,dxdt\\
					&\leq \int_{\Omega_T}G_{\vareps}\chi(\Phi_{\vareps}(v_{1,\vareps})-\Phi_{\vareps}(v_{2,\vareps}))\eta_k\,dxdt.
				\end{split}
			\end{equation}
			Note that by monotonicity of $\Phi_{\vareps}$ and $\Phi_{\vareps}^{-1}$ there holds
			\[
			\chi\LC \Phi_{\vareps}(v_{1,\vareps})-\Phi_{\vareps}(v_{2,\vareps})\RC =\chi \LC v_{1,\vareps}-v_{2,\vareps}\RC.
			\]
               Hence, we can deduce from \eqref{eq: useful estimate in comparison principle} the estimate
			\begin{equation}
				\label{eq: estimate for integration}
				\int_{\Omega_T}\partial_t\LC v_{1,\vareps}-v_{2,\vareps}\RC \chi\LC v_{1,\vareps}-v_{2,\vareps}\RC \eta_k\,dxdt\leq \int_{\Omega_T}G_{\vareps}\chi \LC v_{1,\vareps}-v_{2,\vareps}\RC \eta_k\,dxdt.
			\end{equation}
			Using $v_{j,\vareps}\in H^1(0,T;L^2(\Omega))$ for $j=1,2$ and passing to the limit $k\to\infty$, this ensures $\eta_k\to \chi_{[0,t_0]}$ as $k\to \infty$, and 
			\begin{equation}
				\label{eq: almost final estimate}
				\begin{split}
					\int_{\Omega_{t_0}}\partial_t \LC v_{1,\vareps}-v_{2,\vareps}\RC _+\,dxdt &\leq \int_{\Omega_{t_0}}G_{\vareps}\chi(v_{1,\vareps}-v_{2,\vareps})\,dxdt\\
					&= -\int_{\Omega_{t_0}}L(\Phi_{\vareps}(\varphi_1^{\vareps})-\Phi_{\vareps}(\varphi_2^{\vareps}))\chi(v_{1,\vareps}-v_{2,\vareps})\,dxdt\\
					 &\quad +\int_{\Omega_{t_0}}(F_1^{\vareps}-F_2^{\vareps})\chi(v_{1,\vareps}-v_{2,\vareps})\,dxdt.
				\end{split}
			\end{equation}
			Clearly, we have
			\begin{equation}
				\label{eq: 1st estimate comparison}
				\int_{\Omega_{t_0}}(F_1^{\vareps}-F_2^{\vareps})\chi(v_{1,\vareps}-v_{2,\vareps})\,dxdt\leq \int_{\Omega_{t_0}}\LC F_1^{\vareps}-F_2^{\vareps}\RC _+\,dxdt.
			\end{equation}

			On the other hand, we note that by the uniform ellipticity of $K$ and the support conditions on $\varphi_j$ and $v_{j,\vareps}$, then there holds
			\begin{equation}
				\label{eq: 2n estimate comparison}
				\begin{split}
					&\quad -\int_{\Omega_{t_0}}L(\Phi_{\vareps}(\varphi_1^{\vareps})-\Phi_{\vareps}(\varphi_2^{\vareps}))\chi(v_{1,\vareps}-v_{2,\vareps})\,dxdt\\
					&=-\int_{0}^{t_0}\int_{\R^{2n}}K(x,y) \\
                     &\quad \qquad \quad  \cdot\frac{(\mathcal{W}_{\vareps}(x)-\mathcal{W}_{\vareps}(y))(\chi(v_{1,\vareps}-v_{2,\vareps})(x)-\chi(v_{1,\vareps}-v_{2,\vareps})(y))}{|x-y|^{n+2s}}\,dxdydt\\
					&=2\int_{0}^{t_0}\int_{\Omega_e\times\Omega}K(x,y)\frac{\mathcal{W}_{\vareps}(x)\chi(v_{1,\vareps}-v_{2,\vareps})(y)}{|x-y|^{n+2s}}\,dxdydt\\
					&\leq 2\int_{0}^{t_0}\int_{\Omega_e\times\Omega}K(x,y)\frac{(\mathcal{W}_{\vareps})_+(x)\chi(v_{1,\vareps}-v_{2,\vareps})(y)}{|x-y|^{n+2s}}\,dxdydt\\
					&\leq C\int_0^{t_0}\int_{\Omega}\int_{\Omega_e}\frac{(\mathcal{W}_{\vareps})_+(x)}{|x-y|^{n+2s}}\,dxdydt\\
					&=C\int_0^{t_0}\int_{\Omega}\int_{\Omega_e}\frac{(\Phi_{\vareps}(\varphi_1^{\vareps})-\Phi_{\vareps}(\varphi_2^{\vareps}))_+(x)}{|x-y|^{n+2s}}\,dxdydt,
				\end{split}
			\end{equation}
			where we set
			\[
			\mathcal{W}_{\vareps}=\Phi_{\vareps}(\varphi_1^{\vareps})-\Phi_{\vareps}(\varphi_2^{\vareps}).
			\]
			Combining \eqref{eq: estimate for integration}, \eqref{eq: 1st estimate comparison} and \eqref{eq: 2n estimate comparison}, we obtain
			\begin{equation}
				\label{eq: final estimate hom}
				\begin{split}
					\int_{\Omega}\LC v_{1,\vareps}-v_{2,\vareps}\RC _+(x,t_0)\,dx&\leq C\left( \int_{\Omega_T}\LC F_1^{\vareps}-F_2^{\vareps}\RC_+\,dxdt+\int_{\Omega}\LC u_{0,1}^{\vareps}-u_{0,2}^{\vareps}\RC _+\,dx\right.\\
					&\quad \qquad +\left.\int_{\Omega_{T}}\int_{\Omega_e}\frac{\LC \Phi_{\vareps}(\varphi_1^{\vareps})-\Phi_{\vareps}(\varphi_2^{\vareps})\RC_+(x)}{|x-y|^{n+2s}}\,dxdydt\right)
				\end{split}
			\end{equation}
			for a.e. $0<t_0<T$. Going back to $u_{j,\vareps}$ and recalling that $\left. \varphi_j^{\vareps}\right|_{\Omega_T}=0$, shows
			\begin{equation}
				\label{eq: final estimate inhom}
				\begin{split}
					\int_{\Omega} \LC u_{1,\vareps}-u_{2,\vareps}\RC _+(x,t_0)\,dx&\leq  C\left( \int_{\Omega_T}\LC F_1^{\vareps}-F_2^{\vareps}\RC_+\,dxdt+\int_{\Omega}\LC u_{0,1}^{\vareps}-u_{0,2}^{\vareps}\RC_+\,dx\right.\\
					& \quad \qquad +\left.\int_{\Omega_{T}}\int_{\Omega_e}\frac{\LC \Phi_{\vareps}(\varphi_1^{\vareps})-\Phi_{\vareps}(\varphi_2^{\vareps})\RC_+(x)}{|x-y|^{n+2s}}\,dxdydt\right)
				\end{split}
			\end{equation}
			for a.e. $0<t_0<T$.

			Finally, the proof can be accomplished as follows: 
			as $\vareps\to 0$.
			Note that the weak convergence $u_{j,\vareps}\weak u_j$ in $L^1(\Omega_T)$ and the convexity of $t\mapsto t_+$ guarantees
			\[
			\int_{\Omega_T}\LC u_{1}-u_{2}\RC _+\,dxdt\leq \liminf_{\vareps\to 0}\int_{\Omega_T}\LC u_{1,\vareps}-u_{2,\vareps}\RC_+\,dxdt.
			\]
			Thus, the result follows from \eqref{eq: final estimate inhom} after an integration over $0\leq t_0\leq T$, using the convergence assumptions \ref{prop 1 comparison}--\ref{prop 4 comparison} and the Lipschitz continuity of $t\mapsto t_+$. This proves the assertion.
		\end{proof}

  We next prove a similar result for NPME with linear absorption term.

		\begin{theorem}[Comparison principle with absorption term]\label{thm: comparison absorption}
			Let $\Omega\subset\R^n$ be a bounded Lipschitz domain, $T>0$, $0<s<\alpha\leq 1$, $m> 1$ and $L_K\in \mathcal{L}_0$. Assume additionally that we have given $\rho,q\in C^{1,\alpha}_+(\R^n)$ with $\rho$ uniformly elliptic. Suppose that for $j=1,2$ we have given $u_{0,j}\in L^{\infty}(\Omega)\cap \widetilde{H}^s(\Omega)$, $\varphi_j\in C_c([0,T]\times \Omega_e)$ with $\Phi^m(\varphi_j)\in L^2(0,T;H^s(\R^n))$, $F_j\in L^2(\Omega_T)$ and $u_j$ solves 
			\begin{equation}
				\label{eq: comparison FPME linear absorption}
				\begin{cases}
					\rho\partial_t u+ L_K(\Phi^m(u))+qu=F_j & \text{ in }\Omega_T,\\
					u=\varphi_j&\text{ in }(\Omega_e)_T,\\
					u(0)=u_{0,j}&\text { in }\Omega.
				\end{cases}
			\end{equation}
			Additionally, suppose that there exist sequences $\LC F_j^{\vareps}\RC _{\vareps>0}\subset L^2(\Omega_T)$, $\LC \varphi_j^{\vareps}\RC _{\vareps>0}\subset C_c([0,T]\times \Omega_e)$ with $\Phi^m_{\vareps}(\varphi_j^{\vareps})\in L^2(0,T;H^s(\R^n))$, $\LC u_{0,j}^{\vareps}\RC _{\vareps>0}\subset L^{\infty}(\Omega)\cap\widetilde{H}^s(\Omega)$ and $\LC u_{j,\vareps}\RC _{\vareps>0}\subset H^1(0,T;L^2(\Omega))\cap L^2(0,T;H^s(\R^n))$ satisfying
            \begin{enumerate}[(i)]
                \item\label{prop 1 comparison abs} $\liminf_{\vareps\to 0}\int_{\Omega_T}(F_1^{\vareps}-F_2^{\vareps})_+dxdt\leq \int_{\Omega_T}(F_1-F_2)_+\,dxdt$,
                \item\label{prop 2 comparison abs} $\liminf_{\vareps\to 0}\int_{\Omega_{T}}\int_{\Omega_e}\frac{\left(\Phi^m_{\vareps}(\varphi_1^{\vareps})-\Phi^m_{\vareps}(\varphi_2^{\vareps})\right)_+(x,t)}{|x-y|^{n+2s}}\,dxdydt$\\
                $\leq $ $\int_{\Omega_{T}}\int_{\Omega_e}\frac{\left(\Phi^m(\varphi_1)-\Phi^m(\varphi_2)\right)_+(x,t)}{|x-y|^{n+2s}}\,dxdydt$,
                \item\label{prop 3 comparison abs} $u_{0,j}^{\vareps}\to u_{0,j}$ in $L^1(\Omega)$ as $\vareps\to 0$,
                \item\label{prop 4 comparison abs} $u_{j,\vareps}\weak u_j\ \text{in}\ L^1(\Omega_T)$ as $\vareps\to 0$
                \item\label{prop 5 comparison abs} and $u_{j,\vareps}$ are solutions of
			\begin{equation}
				\begin{cases}
					\partial_t u+ L_K(\Phi_{\vareps}^m(u))+qu=F^{\vareps}_j&\text{ in }\Omega_T,\\
					u=\varphi_j^{\vareps}&\text{ in }(\Omega_e)_T,\\
					u(0)=u_{0,j}^{\vareps}&\text{ in }\Omega.
				\end{cases}
			\end{equation}
            \end{enumerate}
			Then there exists a constant $C>0$ independent of these solutions, initial data, boundary data and sources such that 
			\begin{equation}
				\label{eq: comparison linear absorption}
				\begin{split}
					\int_{\Omega}\LC u_{1}-u_{2}\RC _+(x,t_0)\,dx&\leq C\left( \int_{\Omega_T}\LC F_1-F_2\RC _+\,dxdt+\int_{\Omega}\LC u_{0,1}-u_{0,2}\RC _+\,dx\right.\\
					&\qquad \quad +\left.\int_{\Omega_{T}}\int_{\Omega_e}\frac{\LC \Phi(\varphi_1)-\Phi(\varphi_2)\RC _+(x)}{|x-y|^{n+2s}}\,dxdydt\right).
				\end{split}
			\end{equation}
		\end{theorem}

        \begin{remark}
        \label{remark: comparison holds for solutions}
            Suppose we have given $0\leq u_{0}\in L^{\infty}(\Omega)\cap \widetilde{H}^s(\Omega)$, $0\leq \varphi\in C_c([0,T]\times \Omega_e)$ with $\Phi^m(\varphi)\in L^2(0,T;H^s(\R^n))$, then Theorem~\ref{Uniquenerss with lin absorption term} and Corollary~\ref{cor: approximation} show that the unique solution $u$ of 
            \begin{equation}
            \label{eq: eq for remark comparison}
				\begin{cases}
					\rho\partial_t u+ L_K(\Phi^m(u))+qu=0 & \text{ in }\Omega_T,\\
					u=\varphi&\text{ in }(\Omega_e)_T,\\
					u(0)=0&\text { in }\Omega.
				\end{cases}
			\end{equation}
            can be approximated weakly in $L^1(\Omega_T)$ by a sequence of solutions 
            $$\LC u_{\vareps}\RC_{\vareps>0} \subset H^1(0,T;L^2(\Omega))\cap L^2(0,T;H^s(\R^n)) $$ to  
             \begin{equation}
				\begin{cases}
					\rho\partial_t v+ L_K(\Phi_{\vareps}^m(v))+qv=0 & \text{ in }\Omega_T,\\
					v=\varphi&\text{ in }(\Omega_e)_T,\\
					v(0)=0&\text { in }\Omega.
				\end{cases}
			\end{equation}
            Hence, we see that the conditions \ref{prop 3 comparison abs}, \ref{prop 4 comparison abs}, \ref{prop 5 comparison abs} of Theorem~\ref{thm: comparison absorption} are satisfied when we take $\varphi^{\vareps}=\varphi$ and $u_0^{\vareps}=u_0=0$ and $F^{\vareps}=F=0$. 
            Additionally, using that $\Phi^m_{\vareps}\to \Phi^m$ uniformly on compact sets and $\varphi\in C_c([0,T]\times \Omega)$, we see that
            \begin{equation}
            \label{eq: uniform convergence of nonlinear ext}
                \int_{\Omega_{T}}\int_{\Omega_e}\frac{|\Phi^m_{\vareps}(\varphi^{\vareps})-\Phi^m(\varphi)|(x,t)}{|x-y|^{n+2s}}\,dxdydt\to 0
            \end{equation}
            as $\vareps\to 0$. Therefore, we see that in particular Theorem~\ref{thm: comparison absorption} can be used to compare solutions to equations of the form \eqref{eq: eq for remark comparison}.
        \end{remark}
		
		\begin{proof}[Proof of Theorem \ref{thm: comparison absorption}]
			As in the existence and uniqueness proofs above, we write $\Phi, \Phi_{\vareps}, L,B$ instead of $\Phi^m,\Phi^m_{\vareps},L_K,B_K$. As in the proof of Theorem~\ref{thm: basic comparison} we introduce the new functions $v_{j,\vareps}=u_{j,\vareps}-\varphi_j^{\vareps}$, which satisfy
			\begin{equation}
				\begin{cases}
					\rho\partial_t v+ L(\Phi_{\vareps}(v))+qv=f_{j,\vareps}+F_j^{\vareps}& \text{ in }\Omega_T,\\
					v=0&\text{ in }(\Omega_e)_T,\\
					v(0)=u_{0,j}^{\vareps}&\text{ in }\quad\Omega,
				\end{cases}
			\end{equation}
			where we set $f_{j,\vareps}=-L(\Phi_{\vareps}(\varphi_j^{\vareps}))\in L^2(\Omega_T)$. Therefore, $v_{1,\vareps}-v_{2,\vareps}$ solves
			\begin{equation}
				\label{eq: difference equation for comparison lin abs}
				\begin{cases}
					\rho\partial_t (v_{1,\vareps}-v_{2,\vareps}) + L(\Phi_{\vareps}(v_{1,\vareps})-\Phi_{\vareps}(v_{2,\vareps}))+q(v_{1,\vareps}-v_{2,\vareps})=G_{\vareps}&\text{ in }\Omega_T,\\
					v_{1,\vareps}-v_{2,\vareps}=0&\text{ in }(\Omega_e)_T,\\
					(v_{1,\vareps}-v_{2,\vareps})(0)=u_{0,1}-u_{0,2}&\text{ in }\Omega.
				\end{cases}
			\end{equation}
			where $G_{\vareps}\vcentcolon = f_{1,\vareps}-f_{2,\vareps}+F_1^{\vareps}-F_2^{\vareps}$.
			Arguing as in the proof of Theorem~\ref{thm: basic comparison}, one has
			\begin{equation}
				\label{eq: useful estimate comparison lin abs}
				\begin{split}
					&\int_{\Omega_T}\rho\partial_t(v_{1,\vareps}-v_{2,\vareps})\chi_{\delta}(\Phi_{\vareps}(v_{1,\vareps})-\Phi_{\vareps}(v_{2,\vareps}))\eta_k\,dxdt\\
					=&-\int_0^TB(\Phi_{\vareps}(v_{1,\vareps})-\Phi_{\vareps}(v_{2,\vareps}),\chi_{\delta}(\Phi_{\vareps}(v_{1,\vareps})-\Phi_{\vareps}(v_{2,\vareps})))\eta_k\,dt\\
					&-\int_{\Omega_T}q(v_{1,\vareps}-v_{2,\vareps})\chi_{\delta}(\Phi_{\vareps}(v_{1,\vareps})-\Phi_{\vareps}(v_{2,\vareps}))\eta_k\,dxdt\\
					&+\int_{\Omega_T}G_{\vareps}\chi_{\delta}(\Phi_{\vareps}(v_{1,\vareps})-\Phi_{\vareps}(v_{2,\vareps}))\eta_k\,dxdt.
				\end{split}
			\end{equation}
			Using \eqref{eq: nonnegativity of bilinear form} one obtains
			\[
			\begin{split}
				&\int_{\Omega_T}\rho\partial_t(v_{1,\vareps}-v_{2,\vareps})\chi_{\delta}(\Phi_{\vareps}(v_{1,\vareps})-\Phi_{\vareps}(v_{2,\vareps}))\eta_k\,dxdt\\
				\leq &-\int_{\Omega_T}q(v_{1,\vareps}-v_{2,\vareps})\chi_{\delta}(\Phi_{\vareps}(v_{1,\vareps})-\Phi_{\vareps}(v_{2,\vareps}))\eta_k\,dxdt\\
				&+\int_{\Omega_T}G_{\vareps}\chi_{\delta}(\Phi_{\vareps}(v_{1,\vareps})-\Phi_{\vareps}(v_{2,\vareps}))\eta_k\,dxdt.
			\end{split}
			\]
			Passing to the limit $\delta\to 0$ by Lebesgue's dominated convergence theorem, one can observe by using
			\[
			\chi(\Phi_{\vareps}(v_{1,\vareps})-\Phi_{\vareps}(v_{2,\vareps}))=\chi(v_{1,\vareps}-v_{2,\vareps})
			\] 
			and $q\geq 0$ that there holds
			\begin{equation}
				\label{eq: useful estimate in comparison principle lin abs}
				\begin{split}
					&\int_{\Omega_T}\rho\partial_t(v_{1,\vareps}-v_{2,\vareps})\chi(\Phi_{\vareps}(v_{1,\vareps})-\Phi_{\vareps}(v_{2,\vareps}))\eta_k\,dxdt\\
					\leq &\int_{\Omega_T}G_{\vareps}\chi(\Phi_{\vareps}(v_{1,\vareps})-\Phi_{\vareps}(v_{2,\vareps}))\eta_k\,dxdt.
				\end{split}
			\end{equation}
			Now, we proceed as in the proof of Theorem~\ref{thm: basic comparison} and use the uniform ellipticity of $\rho$ to find
			\begin{equation}
				\begin{split}
					\int_{\Omega}\LC v_{1,\vareps}-v_{2,\vareps}\RC _+(x,t_0)\,dx&\leq C\left( \int_{\Omega_T}\LC F_1^{\vareps}-F_2^{\vareps}\RC _+\,dxdt+\int_{\Omega}\LC u_{0,1}^{\vareps}-u_{0,2}^{\vareps}\RC _+\,dx\right.\\
					&\qquad \quad +\left.\int_{\Omega_{T}}\int_{\Omega_e}\frac{\LC \Phi_{\vareps}(\varphi_1^{\vareps})-\Phi_{\vareps}(\varphi_2^{\vareps})\RC_+(x)}{|x-y|^{n+2s}}\,dxdydt\right).
				\end{split}
			\end{equation}
			Again by a limit argument one can conclude the proof.
		\end{proof}
		
		\section{DN maps and measurement equivalent operators}
		\label{sec: DN maps and measurment equivalent operators}

   With the well-posedness of \eqref{eq: main}, we are able to define the DN map rigorously.

		\subsection{DN maps}
		\label{subsec: DN maps}

           From now on, since $m>1$ is a fixed number, we write $\Phi,\Phi_{\vareps}$ in place of $\Phi^m$ and $\Phi^m_{\vareps}$. Moreover, for $0<s<1$ and $W\subset\Omega_e$ we set
		\begin{align}
			\label{eq: test space}
			\begin{split}
                X_s(W)\vcentcolon = \left\{ \varphi \in C_c([0,T]\times W)\,;\,\varphi\geq 0\ \text{and}\ \Phi(\varphi)\in L^2(0,T;H^s(\R^n))\, \right\}
			\end{split}
		\end{align}
		Additionally, when its notationally convenient we write for a function space $X$ a subscript $+$ to refer to the nonnegative functions in that space. For example, 
        $$ 
        \test_+([0,T]\times \Omega_e) :=\left\{ \,  \varphi\in  C_c^{\infty}([0,T]\times \Omega_e) ;\,  \varphi\geq 0 \, \right\}.
        $$		
	Next, we define the DN map for the fractional porous medium equation.
		
		\begin{definition}[The DN map]
			Let $\Omega\subset\R^n$ be a bounded Lipschitz domain, $T>0$, $0<s<\alpha\leq 1$, $m> 1$ and $L_K\in \mathcal{L}_0$. Assume additionally that we have given $\rho,q\in C^{1,\alpha}_+(\R^n)$ with $\rho$ being uniformly elliptic. Now, we define the DN map $\Lambda_{\rho,K,q}\colon X_s\to L^2(0,T;H^{-s}(\Omega_e))$ by
			\begin{equation}
				\label{eq: DN map}
				\left\langle\Lambda_{\rho,K,q} \varphi,\psi \right\rangle =\int_0^T B_K(\Phi(u),\psi)\,dt
			\end{equation}
			for all $\varphi \in X_s$ and $\psi\in L^2(0,T;\widetilde{H}^s(\Omega_e))$, where $u$ is the unique, nonnegative, bounded solution of
			\begin{equation}
				\label{eq: homogeneous FPME general for def of DN map}
				\begin{cases}
					\rho\partial_t u+ L_K(\Phi(u))+qu=0,&\text{ in }\Omega_T,\\
					u=\varphi&\text{ in }(\Omega_e)_T,\\
					u(0)=0&\text{ in }\Omega,
				\end{cases}
			\end{equation}
			(see~Theorem~\ref{Existence result with linear absorption term} and \ref{Uniquenerss with lin absorption term}).
		\end{definition}
		
		\begin{remark}
			We refer the reader to \cite[Appendix~A]{LRZ2022calder} for a discussion on an alternative DN map in nonlocal diffusion models and its relation to the one used in this article.
		\end{remark}
		
		For $\varphi\in X_s$, let $u$ be the unique, bounded, nonnegative solution to  \eqref{eq: homogeneous FPME general for def of DN map}. Since $\Phi$ is invertible on $\R_+$, the function 
		$$
		v(x,t):=\Phi(u)(x,t) \in L_+^{\infty}(\R^n_T)
		$$
		satisfies 
		\begin{align}
			\label{transferred equation}
			\begin{cases}
				\rho \p_t \Phi^{-1}(v) +L_K(v) + q \Phi^{-1}(v)=0 &\text{ in }\Omega_T, \\
				v=\widetilde{\varphi} & \text{ in } (\Omega_e)_T,\\
				v(0)=0 &\text{ in }\Omega,
			\end{cases}
		\end{align}
		where $\widetilde{\varphi}=\Phi(\varphi)\in L^2(0,T;H^s(\R^n))$. That $v$ solves the above PDE and initial condition is a direct consequence of Definition~\ref{def: weak solutions 2}. To see that $v=\widetilde{\varphi}$ observe that $u=\varphi$ if and only if $\Phi(u-\varphi)\in L^2(0,T;\widetilde{H}^s(\Omega))$. Now, since $\partial\Omega\in C^{0,1}$ and $\Phi(t)=0$ only for $t=0$ this is equivalent to $u=\varphi$ and therefore $v=\widetilde{\varphi}$ a.e. or in $L^2(0,T;\widetilde{H}^s(\Omega))$ sense.

		Next, note that using Theorem~\ref{Existence result with linear absorption term} and \ref{Uniquenerss with lin absorption term}, we deduce that the problem \eqref{transferred equation} is also well-posed for $\widetilde{\varphi}\in \test_+([0,T]\times \Omega_e)$ and $\rho,q\in C_+^{1,\alpha}(\R^n)$ with $\rho$ being uniformly elliptic. In fact, if $\widetilde{\varphi}\in \test_+([0,T]\times \Omega_e)$, then $\varphi\vcentcolon= \Phi^{-1}(\widetilde{\varphi})\in C_c([0,T]\times \Omega_e)$ and $\Phi(\varphi)=\widetilde{\varphi}\in C_c^{\infty}([0,T]\times \Omega_e)\subset L^2(0,T;H^s(\R^n))$ and hence $\varphi \in  X_s$. Now, by Theorem~\ref{Existence result with linear absorption term} and \ref{Uniquenerss with lin absorption term} there is a unique, nonnegative, bounded solution $u$ of \eqref{eq: homogeneous FPME general for def of DN map} and hence $v=\Phi(u)$ is the unique, nonnegative, bounded solution of \eqref{transferred equation}. 
		
		Hence, we can define the DN map of \eqref{transferred equation}, denoted by $\Lambda^{\Phi}_{\rho,K,q}\colon \test_+([0,T]\times \Omega_e)\to L^2(0,T;H^{-s}(\Omega_e))$, via the formula
		\begin{equation}
			\label{eq: DN map transferred equation}
			\left\langle \Lambda^{\Phi}_{\rho,K,q} \widetilde{\varphi},\psi \right\rangle=\int_0^TB_K(v,\psi)\,dt
		\end{equation}
		with $\widetilde{\varphi} \in \test_+([0,T]\times \Omega_e)$ and $\psi\in L^2(0,T;\widetilde{H}^{s}(\Omega_e))$, where $v$ is the unique, non-negative, bounded solution of \eqref{transferred equation}. 
           It is immediate from our definitions that
		\begin{equation}
			\label{eq: Relation DN maps}
			\langle\Lambda_{\rho,K,q}\varphi,\psi\rangle = \langle\Lambda^{\Phi}_{\rho,K,q}\Phi(\varphi),\psi\rangle
		\end{equation}
		for all $\varphi\in X_s$ with $\Phi(\varphi)\in \test_+([0,T]\times \Omega_e)$ and $\psi\in L^2(0,T;\widetilde{H}^s(\Omega_e))$. Hence, if $(\rho_j,K_j,q_j)$ for $j=1,2$ are coefficients satisfying the usual assumptions and $W_1,W_2\subset\Omega_e$ two generic measurement sets, then
        \begin{equation}
        \label{eq: assumption on DN maps}
            \left. \Lambda_{\rho_1,K_1,q_1}\varphi\right|_{(W_2)_T}=\left.\Lambda_{\rho_2,K_2,q_2}\varphi\right|_{(W_2)_T},\quad \text{for all }  \varphi\in X_s(W_1)
        \end{equation}
        implies
        \begin{equation}
        \label{eq: consequence on transferred DN maps}
            \left.\Lambda^{\Phi}_{\rho_1,K_1,q_1}\widetilde{\varphi}\right|_{(W_2)_T}=\left.\Lambda^{\Phi}_{\rho_2,K_2,q_2}\widetilde{\varphi}\right|_{(W_2)_T},\quad \text{for all }\widetilde{\varphi}\in \test_+\LC [0,T]\times W_1\RC.
        \end{equation}
		Let us note that these restrictions are meant in the sense that we test against $\psi\in L^2(0,T;\widetilde{H}^s(W_2))$.

		Hence, if we can show that the DN maps \eqref{eq: DN map transferred equation} uniquely determine the coefficients $\rho,K,q$, then this will also give a positive answer to the ~Question~\ref{inverse problem}. 
		
		\subsection{Measurement equivalent nonlocal operators}\label{subsec: meas equivalence}

		Recalling Definition \ref{def: nonlocal op for inverse problem}, the next two propositions provide examples of measurement equivalent nonlocal opeartors. The first example deals with nonlocal operators whose kernel are real analytic and separable.
		
		\begin{proposition}[Real analytic, separable coefficients]
			\label{prop: real analytic coefficients}
			Let $\Omega \subset \mathbb{R}^n$ be a bounded open set and $0<s<1$. Then uniformly elliptic nonlocal operators $L_K$ of order $2s$ such that $K(x,y)$ is of the form 
			\begin{equation}
				\label{eq: real analytic}
				K(x,y)=F(\gamma(x))F(\gamma(y)),
			\end{equation}
			for some real analytic functions $\gamma\colon\R^n\to\R$, $F\colon\R_+\to\R_+$ satisfying
			\begin{enumerate}[(i)] 
				\item $\gamma$ is uniformly elliptic,
				\item $F$ is injective
				\item and for any compact interval $[a,b]\subset \R_+$ there exists $c>0$ such that $F(\xi)\geq c$ for all $\xi\in [a,b]$
			\end{enumerate} 
			are measurement equivalent.
		\end{proposition}
		
		\begin{proof}
			The proof is a consequence of \cite[Proposition~1.4]{KLZ-2022} and the observation that the sequence $\LC \Phi_N\RC _{N\in\N}$ in \cite[Theorem~1.1]{KLZ-2022} can be constructed in such a way that $\Phi_N\geq 0$.
		\end{proof}
		
		Next, we state another example of measurement equivalent operators.
		
		\begin{proposition}[Fractional conductivity operators]
			\label{prop: fractional conductivity operator}
			Let $\Omega \subset \mathbb{R}^n$ be a bounded open set and $0<s<\min(1,n/2)$. Let $\gamma\in L^{\infty}(\R^n)$ be uniformly elliptic such that
			\begin{enumerate}[(i)]
				\item the background deviation $m_{\gamma}\vcentcolon = \gamma^{1/2}-1$ belongs to the Bessel potential space $H^{s,n/s}(\R^n)$\footnote{Recall that the Bessel potential space $H^{s,p}(\R^n)$ with $1\leq p<\infty$ and $s\in\R$ is the space of tempered distributions $u\in\tempered(\R^n)$ such that $\|u\|_{H^{s,p}(\R^n)}=\|\langle D\rangle^s u\|_{L^p(\R^n)}<\infty$, where $\langle D\rangle^s$ is the Fourier multiplier with symbol $\langle \xi\rangle^s=(1+|\xi|^2)^{s/2}$.},
				\item $\gamma|_{\Omega_e}=\Gamma$ for some fixed $\Gamma\in L^{\infty}(\R^n)$ with $m_{\Gamma}\in H^{s,n/s}(\R^n)$,
				\item and $q_{\gamma}\vcentcolon = -(-\Delta)^s m_{\gamma}/\gamma^{1/2}\in C(\overline{\Omega})$.
			\end{enumerate}
			Then the nonlocal operator $L_{K}\in \mathcal{L}_0$ with
			\[
			K(x,y)=C_{n,s}\gamma^{1/2}(x)\gamma^{1/2}(y),
			\]
			where $C_{n,s}>0$ is the normalization constant from the fractional Laplacian, is measurement equivalent. 
		\end{proposition}
		
		\begin{remark}
			Let us recall that the operators $L_{K}$ of Proposition~\ref{prop: fractional conductivity operator} are called \emph{fractional conductivity operator} as they converge to the conductivity operator $-\Div(\gamma\nabla u)$ as $s\uparrow 1$ for sufficiently regular functions $\gamma$ and $u$.
		\end{remark}
		
		\begin{proof}[Proof of Proposition \ref{prop: fractional conductivity operator}]
			We only give here the needed modifications in the proofs of the results in \cite{RZ-low-2022} and refer for further details to this article. First, let us observe that any nonnegative $\widetilde{H}^s(W)$ with $W\subset\Omega_e$ can approximated by nonnegative test functions $\varphi\in C_c^{\infty}(W)$. In fact, let $\varphi\in \widetilde{H}^s(W)$ and choose $\LC \varphi_k\RC_{k\in\N}\subset C_c^{\infty}(W)$ such that $\varphi_k\to \varphi$ in $H^s(\R^n)$ as $k\to\infty$. Then we define
			\[
			\psi_k(x)=
			\sqrt{(\varphi_k(x))^2+\vareps_k^2}-\vareps_k
			\]
			for some sequence $\vareps_k\downarrow 0$ as $k\to\infty$. We have $\psi_k\in C_c^{\infty}(W)$ and $\psi_k\geq 0$. Clearly, up to extracting a subsequence we can assume $\varphi_k\to \varphi$ as $k\to\infty$ a.e. in $\R^n$ and hence $\psi_k\to \varphi$ a.e. in $\R^n$ as $k\to \infty$. Since $\varphi_k\to \varphi$ in $L^2(\R^n)$ as $k\to \infty$, there exists $v\in L^2(\R^n)$ such that $|\varphi_k|\leq v$ uniformly in $k$ (see~\cite[Theorem~4.9]{Brezis}). Therefore, Lebesgue's dominated convergence theorem gives $\psi_k\to \varphi$ in $L^2(\R^n)$ as $k\to \infty$. Additionally one has
			\[
			\frac{\psi_k(x)-\psi_k(y)}{|x-y|^{n/2+s}}\to \frac{\varphi(x)-\varphi(y)}{|x-y|^{n/2+s}}\ \text{a.e. in}\ \R^{2n}, \text{ as }k\to \infty,
			\]
			and using the Lipschitz continuity of the function $\zeta(t):= \sqrt{t^2+\vareps^2}-\vareps$, we can get $\left| \zeta'(t) \right|= \left| \frac{t}{\sqrt{t^2+\vareps^2}}\right| \leq 1$, so that 
			\[
			\left|\frac{\psi_k(x)-\psi_k(y)}{|x-y|^{n/2+s}}\right|\leq \frac{|\varphi_k(x)-\varphi_k(y)|}{|x-y|^{n/2+s}} \leq V(x,y)\in L^2(\R^{2n}).
			\]
			The existence of such a $V\in L^2(\R^{2n})$ follows from $\varphi_k\to \varphi$ in $H^s(\R^n)$ as $k\to\infty$ and \cite[Theorem~4.9]{Brezis}. The Lebesgue's dominated convergence theorem implies 
			\[
			\begin{split}
				\left[\psi_k-\varphi\right]^2_{H^s(\R^n)}&=\int_{\R^2n}\frac{\left|((\psi_k-\varphi)(x)-(\psi_k-\varphi)(y)\right|^2}{|x-y|^{n+2s}}\,dxdy\\
				&=\int_{\R^{2n}}\left|\frac{(\psi_k(x)-\psi_k(y))-(\varphi(x)-\varphi(y))}{|x-y|^{n/2+s}}\right|^2\,dxdy\to 0
			\end{split}
			\]
			as $k\to\infty$. Hence, we can conclude the assertion.
			
			Next, let $\gamma_j\in L^{\infty}(\R^n)$ be two functions with associated coefficients $K_j(x,y)=C_{n,s}\gamma_j^{1/2}(x)\gamma_j^{1/2}(y)$ for $j=1,2$, and operators $L_{K_1},L_{K_2}$ such that the conditions in Proposition~\ref{prop: fractional conductivity operator} for two non-disjoint, nonempty, open sets $W_1,W_2\subset \Omega$ hold. Moreover, assume that we have \eqref{eq: for measurement equivalent} for these data. Now one can follow the same arguments as in \cite[Lemma~4.1]{RZ-low-2022} or \cite[Lemma~4.1]{CRZ2022global} to deduce that there holds
			\[
			\left.\Lambda_{q_1}\varphi\right|_{W_2}= \left.\Lambda_{q_2}\varphi\right|_{W_2}
			\]
			for all nonnegative $\varphi \in C_c^{\infty}(W_1)$. Here for $j=1,2$ the potentials $q_j$ are given by 
			\[
			q_j=-\frac{(-\Delta)^s m_{\gamma_j}}{\gamma_j^{1/2}}
			\]
			and $\Lambda_{q_j}$ are the DN maps related to the Dirichlet problem
			\begin{equation}
				\begin{cases}
					\LC (-\Delta)^s+q_j\RC u=0& \text{ in } \Omega\\
					u=\varphi& \text{ in } \Omega_e,
				\end{cases}
			\end{equation}
			which is well-posed by \cite[Lemma~3.11]{RZ-low-2022}. Using \cite[Theorem~1]{GRSU20}, we deduce that $q_1=q_2$ in $\Omega$. To apply this theorem we make the measurement sets if needed smaller.

            Now, we can follow the proof of \cite[Lemma~8.15]{RZ2022unboundedFracCald} to see that
			\[
			\int_{\Omega_e}(-\Delta)^{s/2}\LC m_{\gamma_1}-m_{\gamma_2}\RC (-\Delta)^{s/2}(\varphi\psi)\,dx=0
			\]
			for all nonnegative functions $\varphi\in C_c^{\infty}(W_1)$, $\psi\in C_c^{\infty}(W_2)$. Using a cut-off function this implies
			\[
			\int_{\Omega_e}(-\Delta)^{s/2}\LC m_{\gamma_1}-m_{\gamma_2}\RC (-\Delta)^{s/2}\varphi\,dx=0
			\]
			for all nonnegative functions $\varphi\in C_c^{\infty}(\omega)$ with $\omega\Subset W_1\cap W_2$. But now since 
			\[
			\varphi \psi=\frac{(\varphi+\psi)^2-\varphi^2-\psi^2}{2},
			\]
			we find
			\[
			\int_{\Omega_e}(-\Delta)^{s/2}\LC m_{\gamma_1}-m_{\gamma_2}\RC (-\Delta)^{s/2}(\varphi\psi)\,dx=0
			\]
			for all $\varphi,\psi\in C_c^{\infty}(\omega)$. By using again a cut-off function this implies 
			\[
			\int_{\Omega_e}(-\Delta)^{s/2}\LC m_{\gamma_1}-m_{\gamma_2}\RC (-\Delta)^{s/2}\varphi\,dx=0
			\]
			for all $\varphi\in C_c^{\infty}(\omega')$, where $\omega'\Subset \omega$. Therefore, we can conclude that 
           $$
           (-\Delta)^s \LC m_{\gamma_1}-m_{\gamma_2}\RC=0 \text{ in } \omega'.
           $$ 
           Taking into account $\left.\gamma_1\right|_{\Omega_e}=\left.\gamma_2 \right|_{\Omega_e}$, we deduce from the unique continuation principle for the fractional Laplacian in Bessel potential spaces (see \cite[Theorem~2.2]{KRZ-2023}) that $\gamma_1=\gamma_2$ in $\R^n$ and therefore $K_1=K_2$ in $\R^{2n}$ as asserted.
		\end{proof}
		
          \begin{remark}
          \label{remark: measurement equiv op}
              We have shown that the fractional conductivity operator is a measurement equivalent nonlocal operator.
              It is also natural to ask a question whether two nonlocal operators of the form $\LC -\nabla \cdot A \nabla \RC ^s$ are under certain conditions measurement equivalent or not. In fact, this is still an open problem. Very recently, the works 
              \cite{GU2021calder,CGRU2023reduction} demonstrated that if the nonlocal DN maps agree for $\LC -\nabla \cdot A_1 \nabla \RC ^s$ and $\LC -\nabla \cdot A_2 \nabla \RC^s$, under appropriate conditions, then the corresponding DN maps agree for their local counterparts. This connects nonlocal and local inverse problems by using approaches of either heat semigroup or Caffarelli-Silvestre extension. In addition, similar question has been addressed for nonlocal parabolic operators, and we refer readers to \cite{LLU2022calder} for more details.
          \end{remark}

		\section{Uniqueness of the inverse problem}\label{sec: inverse}
		
		In order to prove Theorem \ref{thm: main}, we need a simple property about the nonlocal operator $L_K\in \mathcal{L}_0$.
		
		\subsection{Dirichlet problem for $L_K\in \mathcal{L}_0$}

		\begin{lemma}\label{cor: exterior value}
			Let $\Omega \subset \R^n$ be a bounded domain, $0<s<1$, $L_K\in\mathcal{L}_0$, $F\in H^{-s}(\Omega)$ and $f\in H^s(\R^n)$. Then the Dirichlet problem
			\begin{align}
            \label{eq: Dirichlet problem for LK}
				\begin{cases}
					L_K u= F &\text{ in }\Omega, \\
					u = f &\text{ in }\Omega_e.
				\end{cases}
			\end{align}
			has a unique solution $u\in H^s(\R^n)$ satisfying
			\begin{align}\label{elliptic estimate nonlocal}
				\norm{u}_{H^s(\R^n)}\leq C \LC  \norm{F}_{H^{-s}(\Omega)}+ \norm{f}_{H^s(\R^n)} \RC,
			\end{align}
			for some constant $C>0$ independent of $V$, $F$ and $f$.
		\end{lemma}
		
		\begin{proof}
			As noted in the proof of Theorem~\ref{thm: basic existence} the bilinear form $B_K$ defined via \eqref{eq: bilinear form kernel} induces an equivalent inner product on $\widetilde{H}^s(\Omega)$ and hence the Lax--Milgram theorem guarantees the existence of a unique solution $v\in \widetilde{H}^s(\Omega)$ to 
            \begin{align}
             \label{eq: Dirichlet problem for LK 2}
				\begin{cases}
					L_K v= G &\text{ in }\Omega, \\
					v = 0 &\text{ in }\Omega_e
				\end{cases}
			\end{align}
            for any $G\in H^{-s}(\Omega)$ and $v$ satisfies the continuity estimate
            \[
                    \|v\|_{H^s(\R^n)}\leq C\|G\|_{H^{-s}(\Omega)}
            \]
            for some $C>0$. Now, the assertions of Lemma~\ref{cor: exterior value} follow by observing that $u\in H^s(\R^n)$ uniquely solves \eqref{eq: Dirichlet problem for LK} if and only if $v=u-f$ uniquely solves \eqref{eq: Dirichlet problem for LK 2} with $G=F-L_K f\in H^{-s}(\Omega)$. Here we are using the fact that $L_K\colon\widetilde{H}^s(\Omega)\to H^{-s}(\Omega)$ is a bounded linear operator (see Section~\ref{subsec: nonlocal operators}).
		\end{proof}

		\subsection{An integral-time transformation}
        \label{subsec: integral time transform}
		
		For given $0<T_0\leq T$, $\beta>1$ and $h>1$, let us consider exterior data $\widetilde{\varphi}\in\test_+([0,T]\times W_1)$ of the form
		\begin{align}\label{exterior data in the proof}
			\widetilde{\varphi}(x,t)=t^mh\widetilde{\varphi}_0(x) \quad \text{with}\quad \widetilde{\varphi}_0 \in \test_+(W_1)\setminus\{0\}.
		\end{align}
		The parameters $T_0>0$ and $\beta>0$ will be specified later. Now, let $v$ be the unique solution to \eqref{transferred equation} (with $\widetilde{\varphi}$ as in \eqref{exterior data in the proof}) and define $V\in H^s(\R^n)$ by the time-integral
		\begin{align}\label{V=V_h}
			V(x):=\int_0^{T_0}(T_0-t)^\beta v(x,t)\, dt.
		\end{align}
        The regularity of $V$ follows from the fact that $v\in L^2(0,T;H^s(\R^n))$. Note that to define $V$ we only need to observe the nonlinear effects of the porous medium up to a possibly small time $T_0$. Next, we assert that $V$ solves 
		\begin{align}
        \label{equation for time-transf.}
			\begin{cases}
				L_KV=-( \mathcal{M}+\mathcal{N}) &\text{ in }\Omega, \\
				V =h\mathcal{B}(\beta+1,m+1)T_0^{\beta+m+1} \widetilde{\varphi}_0  &\text{ in }\Omega_e,  
			\end{cases}
		\end{align}
		where $\mathcal{M}=\mathcal{M}(x)$ and $\mathcal{N}=\mathcal{N}(x)$ are given by
        \begin{align}
			\begin{split}
				\mathcal{M}&= \beta\rho\int_0^{T_0} (T_0-t)^{\beta-1}\Phi^{-1}(v)\, dt, \\
				\mathcal{N}&= q\int_0^{T_0} (T_0-t)^\beta \Phi^{-1}(v)\, dt
			\end{split}
		\end{align}
        and $\mathcal{B}$ is the Euler beta function defined as 
        \[
            \mathcal{B}(a,b)=\int_0^1(1-t)^{a-1}t^{b-1}\,dt,
        \]
        for $a,b>0$. That $V\in H^s(\R^n)$ attains the prescribed exterior values is easily seen from
		\begin{equation}
        \label{eq euler beta function}
		\int_0^{T_0} (T_0-t)^\beta t^m \, dt =\mathcal{B}(\beta+1,m+1)T_0^{\beta+m+1},
		\end{equation}
        and the fact that the exterior values in \eqref{transferred equation} can be considered in an a.e. sense due to the Lipschitz regularity of $\partial\Omega$. The identity \eqref{eq euler beta function} follows by a change of variables and the definition of $\mathcal{B}$. In fact, we may calculate 
        \begin{equation}
        \label{eq: calc bdry val}
            \begin{split}
                \int_0^{T_0}(T_0-t)^{\beta}t^m\,dt&=T_0^{\beta+m+1}\int_0^1 (1-t)^{(\beta+1)-1}t^{(m+1)-1}\,dt\\
                &=T_0^{\beta+m+1}\mathcal{B}(\beta+1,m+1).
            \end{split}
        \end{equation}
        Next, observe that the condition $\beta>1$ shows that the function
        \[
            \psi_w=(T_0-t)^{\beta} w \ \text{with}\ w\in C_c^{\infty}(\Omega)
        \]
        satisfies $\psi_w \in H^1(0,T_0;\widetilde{H}^s(\Omega))$ and $\psi_w(T_0)=0$.
        Hence, using that $v$ solves \eqref{transferred equation}, $\Phi^{-1}(v)=0$, $\rho\in C^{1,\alpha}_+(\R^n)$ is uniformly elliptic, Remark~\ref{eq: continuity estimate hoelder coeff}, Proposition~\ref{prop: density} and $\Phi^{-1}(v)(0)=0$, we may compute
        \[
        \begin{split}
            &\left\langle L_K \LC \int_0^{T_0}(T_0-t)^{\beta}v\,dt \RC ,w \right\rangle \\
            =&\int_0^{T_0}(T_0-t)^{\beta}B_K(v,w)\,dt\\
            =&\int_{\Omega_{T_0}}\Phi^{-1}(v)\partial_t(\rho \psi_w)\,dxdt -\int_{\Omega_{T_0}}q\Phi^{-1}(v)\psi_w\,dxdt\\
            =&\left\langle -\beta\rho\int_0^{T_0}(T_0-t)^{\beta-1}\Phi^{-1}(v)\,dt-q\int_0^{T_0}(T_0-t)^{\beta}\Phi^{-1}(v)\,dt, w\right\rangle
        \end{split}
        \]
        for all $w\in C_c^{\infty}(\Omega)$. This establishes that $V$ solves \eqref{equation for time-transf.}.
        Before, proceeding let us remark that since the exterior data $v|_{(\Omega_e)_T}$ depends on $h$, the quantities $V$, $\mathcal{M}$ and $\mathcal{N}$ also depend on the parameter $h>1$.

		Next, we derive fundamental estimates for  $\mathcal{M}$ and $\mathcal{N}$. Before doing this observe that if $\beta>1$, $m>1$ and the H\"older conjugate $m'$ of $m$ satisfies $\beta-m'>-1$, then we have
        \begin{equation}
        \label{eq: easy integration}
           \begin{split}
               \int_0^{T_0}(T_0-t)^{\beta-m'}\,dt&=T_0^{\beta-m'+1}\int_0^1t^{\beta-m'}\,dt=\frac{T_0^{\beta-m'+1}}{\beta-m'+1}.
           \end{split}
        \end{equation}
        From now on we assume that 
        \begin{equation}
        \label{eq: assump on beta}
            \beta>1\ \text{and}\ \beta>m'-1=\frac{1}{m-1}.
        \end{equation}
        Hence, by using H\"older's inequality, $v\geq 0$, $\Phi^{-1}(t)=t^{1/m}$ for $t\geq 0$, that $\rho$ is uniformly elliptic and \eqref{eq: easy integration}, we get
		\begin{align}\label{point esti M}
			\begin{split}
				0\leq \mathcal M(x)&=\beta \rho(x)\int_0^{T_0} (T_0-t)^{\frac{\beta-m'}{m'}}\LC(T_0-t)^{\beta}v(x,t) \RC^{\frac{1}{m}} \, dt \\
                &\leq  C T_0^{\frac{\beta-m'+1}{m'}}(V(x))^{\frac{1}{m}} \\
                &=  C T_0^{\frac{\beta}{m'}-\frac{1}{m}}(V(x))^{\frac{1}{m}},
			\end{split}
		\end{align}
		for some constant $C>0$ only depending on $\beta$, $m$ and $\|\rho\|_{L^{\infty}(\R^n)}$. Noting that \eqref{eq: easy integration} remains true for $m'=0$, we obtain by H\"older's inequality and $0\leq q\in L^{\infty}(\R^n)$, the pointwise estimate
		\begin{align}\label{point esti N}
			\begin{split}
				0\leq\mathcal{N}(x)\leq C T_0^{\frac{\beta+1}{m'}} (V(x))^{\frac{1}{m}},
			\end{split}
		\end{align}
		for some constant $C>0$ only depending on $\beta$ and $\|q\|_{L^{\infty}(\R^n)}$. By applying \eqref{point esti M}, \eqref{point esti N} and Jensen's inequality, we can obtain the $L^2$ estimates 
		\begin{align}\label{L2 est for M N}
			\begin{split}
				\norm{\mathcal{M}}_{L^2(\Omega)} & \leq C T_0^{\frac{\beta}{m'}-\frac{1}{m}} \norm{V}_{L^2(\Omega)}^{\frac{1}{m}}, \\
				\norm{\mathcal{N}}_{L^2(\Omega)}& \leq  C  T_0^{\frac{\beta+1}{m'}} \norm{ V}_{L^2(\Omega)}^{\frac{1}{m}}
			\end{split}
		\end{align}
		for some constant $C>0$ only depending on $\beta, m$, on the norms $\|\rho\|_{L^{\infty}(\R^n)}, \|q\|_{L^{\infty}(\R^n)}$ and $|\Omega|$. On the other hand, applying the estimate \eqref{elliptic estimate nonlocal} to the equation \eqref{equation for time-transf.}, one has 
		\begin{align}
        \label{esti of V}
			\begin{split}
				\norm{V}_{H^s(\R^n)} &\leq  C \LC \norm{\mathcal{M}+\mathcal{N}}_{H^{-s}(\Omega)} +\mathcal{B}(\beta+1,m+1) T_0^{\beta+m+1}h\norm{\wt \varphi_0}_{H^s(\R^n)}  \RC \\
			&	\leq C \left\{\LC  T_0^{\frac{\beta}{m'}-\frac{1}{m}}+ T_0^{(\beta+1)/m'}\RC \norm{ V}_{L^2(\Omega)}^{\frac{1}{m}}+T_0^{\beta+m+1}h\norm{\wt \varphi_0}_{H^s(\R^n)}   \right\},
			\end{split}
		\end{align}
		for some constant $C>0$ independent of $V$, $\wt \varphi_0$, $T_0$ and $h$. Since, $h>1$ and $m>1$ this implies
        \begin{equation}
        \label{esti of V 2}
            \begin{split}
                &\quad \max(\norm{V}_{H^s(\R^n)},h)\\
                &\leq  C \left\{\LC  T_0^{\frac{\beta}{m'}-\frac{1}{m}}+ T_0^{(\beta+1)/m'}\RC (\max(\norm{ V}_{L^2(\Omega)},h))^{\frac{1}{m}}+T_0^{\beta+m+1}h\norm{\wt \varphi_0}_{H^s(\R^n)} +h  \right\}\\
                &\leq C \left\{\LC  T_0^{\frac{\beta}{m'}-\frac{1}{m}}+ T_0^{(\beta+1)/m'}\RC \max(\norm{ V}_{L^2(\Omega)},h)+T_0^{\beta+m+1}h\norm{\wt \varphi_0}_{H^s(\R^n)} +h  \right\}\\
                &=C \left\{\LC  T_0^{\frac{\beta}{m'}-\frac{1}{m}}+ T_0^{(\beta+1)/m'}\RC \max(\norm{ V}_{L^2(\Omega)},h)+h\left(T_0^{\beta+m+1}\norm{\wt \varphi_0}_{H^s(\R^n)} +1\right)  \right\}
            \end{split}
        \end{equation}

        Next, observe that in the estimate \eqref{esti of V 2} all exponents of $T_0$ are strictly positive by our choice of $\beta$ (see \eqref{eq: assump on beta}) and hence we can absorb the first term in the last line of \eqref{esti of V 2} on the left hand side after choosing $T_0$ sufficiently small. This shows
		\begin{align}\label{estimate for V}
			\norm{V}_{H^s(\R^n)} \leq Ch\left(T_0^{\beta+m+1}\norm{\wt \varphi_0}_{H^s(\R^n)} +1\right)
		\end{align}
        for some $C>0$.

		\subsection{Asymptotic analysis}
        \label{subsec: asymptotic analysis}

		First, let us consider the ansatz 
		\begin{align}\label{ansatz 1}
			V=V_h=\mathcal{B}(\beta+1,m+1)T_0^{\beta+m+1}h V^{(0)} + R_1,
		\end{align}
		where $V^{(0)}$ and $R_1$ are the solutions to 
		\begin{align}\label{V_0 equation}
			\begin{cases}
				L_K V^{(0)}=0 &\text{ in }\Omega, \\
				V^{(0)}=\wt\varphi_0 & \text{ in }\Omega_e
			\end{cases}
		\end{align}
		and 
		\begin{align}\label{remainder eq}
			\begin{cases}
				L_K R_1= -(\mathcal{M}+\mathcal{N}) &\text{ in }\Omega, \\
				R_1=0 & \text{ in }\Omega_e,
			\end{cases}
		\end{align}
		respectively. By using Lemma~\ref{cor: exterior value}, \eqref{L2 est for M N} and \eqref{estimate for V}, we get 
		\begin{align}\label{R_1 estimate}
			\begin{split}
				\left\| R_1\right\|_{H^{s}(\R^n)}&\leq  C   \norm{\mathcal{M}+\mathcal{N}}_{H^{-s}(\Omega)}\\
    &\leq  C \norm{ V}_{L^2(\Omega)}^{\frac{1}{m}}\\
    &\leq  C h^{\frac{1}{m}}(\norm{\wt \varphi_0}_{H^s(\R^n)}+1)^{1/m},
			\end{split}
		\end{align}
		and 
		\begin{align} \label{L of remainder estimate}
			\left\|L_K R_1\right\|_{H^{-s}(\Omega)} =\|\mathcal{M}+\mathcal{N}\|_{H^{-s}(\Omega)} \leq C h^{\frac{1}{m}}\LC \norm{\wt \varphi_0}_{H^s(\R^n)}+1\RC ^{1/m}.
		\end{align}
		This implies that 
		\begin{align}
			h^{-1}L_K R_1 =\mathcal{O}\LC h^{\frac{1}{m}-1} \RC  \text{ as }h\to \infty,
		\end{align}
		in $H^{-s}(\Omega)$. Thus, we have the asymptotic behavior of 
		\begin{align}\label{asymp for L V_h}
			L_K\LC h^{-1} V_h \RC =	h^{-1}L_K R_1 =\mathcal{O}\LC h^{\frac{1}{m}-1} \RC \text{ as }h\to \infty,
		\end{align}
		in $H^{-s}(\Omega)$. For the sake of convenience, let us introduce the following function
		\begin{align}\label{v_0 spacetime}
			v_0(x,t) := h t^m V^{(0)}(x)\in H^1(0,T_0;\widetilde{H}^s(\Omega))
		\end{align}	
		and observe that 
        \begin{equation}
        \label{eq: properties V0}
            0\leq v_0(x,t)\leq M\ \text{for a.e.}\ (x,t)\in\R^n_{T_0}
        \end{equation}
        for some $M>0$. In fact, by definition $V^{(0)}$ has exterior value $\widetilde{\varphi}_0\geq 0$. Arguing as in the proof of \cite[Proposition~4.1]{survey-xavier} one sees that the maximum principle for the equation \eqref{V_0 equation} still holds and guarantees the nonnegativity
        \begin{equation}
        \label{eq: nonnega of V0}
            V^{(0)}(x)\geq 0\ \text{for a.e.}\ x\in\Omega.
        \end{equation}
        Moreover, that by linearity the maximum principle directly implies the comparison principle for equation \eqref{V_0 equation}. Furthermore, \cite[Lemma~5.1]{survey-xavier} remains true in our setting, since $K$ is uniformly elliptic and the function $w\in C_c^{\infty}(\R^n)$ constructed in that result has its maximum in $\Omega$. Thus, in our case one can still establish \cite[Corollary~5.2]{survey-xavier} and hence conclude that there holds 
        \begin{equation}
        \label{eq: upper bound V0}
            \left\|V^{(0)} \right\|_{L^{\infty}(\Omega)}\leq C \left\|\widetilde{\varphi}_0\right\|_{L^{\infty}(\Omega_e)}.
        \end{equation}
        Therefore, we have shown the estimate \eqref{eq: properties V0}.
        
        Next, let us observe that \eqref{eq: calc bdry val} implies
		\begin{align}
			\int_0^{T_0} (T_0-t)^{\beta} v_0(x,t)\, dt =h\mathcal{B}(\beta+1,m+1) T_0^{\beta+m+1} V^{(0)}(x).
		\end{align}

		Motivated by the asymptotic behaviour of the remainder $R_1$ (see \eqref{R_1 estimate} and \eqref{L of remainder estimate}), in a next step we refine the ansatz for $V$ as
		\begin{align}\label{ansatz 2}
			V=\widetilde{V}_h=h \mathcal{B}(\beta+1,m+1)T_0^{\beta+m+1} V ^{(0)} + h^{\frac{1}{m}}V ^{(1)} +R_2,
		\end{align}
		where $V^{(0)}$ is the solution of \eqref{V_0 equation} and $V^{(1)}$ is the solution to 
		\begin{align}\label{V_1 equation}
			\begin{cases}
				L_K V^{(1)}= - h^{-\frac{1}{m}}\LC \mathcal{M}^{(1)}+\mathcal{N}^{(1)}  \RC&\text{ in }\Omega, \\
				V^{(1)}=0 & \text{ in }\Omega_e,
			\end{cases}
		\end{align}
		where 
		\begin{align}\label{M_1}
			\begin{split}
				0\leq \mathcal{M}^{(1)}(x)&=\beta \rho (x)\int_0^{T_0} (T_0-t)^{\beta-1}v_0^{\frac{1}{m}}(x,t)\, dt \\
				&=h^{\frac{1}{m}}\rho(x) \frac{2T_0^{\beta+1}}{\beta(\beta+1)}\LC V^{(0)}(x)\RC^{\frac{1}{m}},
			\end{split} 
		\end{align}
		and 
		\begin{align}\label{N_1}
			\begin{split}
				0\leq \mathcal{N}^{(1)}(x)&=q(x)\int_0^{T_0} (T_0-t)^{\beta}v_0^{\frac{1}{m}}(x,t)\, dt \\
				&=h^{\frac{1}{m}} q(x) \frac{2T_0^{\beta+2}}{(\beta+1)(\beta+2)}\LC V^{(0)}(x)\RC^{\frac{1}{m}}.
			\end{split}
		\end{align}
        The computations in \eqref{M_1} and \eqref{N_1} are easily justified by using \eqref{eq: properties V0}, \eqref{eq: calc bdry val}, $\mathcal{B}(a,b)=\frac{\Gamma(a)\Gamma(b)}{\Gamma(a+b)}$, $\Gamma(a+1)=a\Gamma(a)$ and $\Gamma(n)=n !$ for $a,b>0$ and $n\in\N$.
		Hence, the remainder term $R_2$ needs to satisfy
		\begin{align}\label{R_2 equation}
			\begin{cases}
				L_KR_2=-\left[\LC \mathcal{M}-\mathcal{M}_1 \RC + \LC \mathcal{N}-\mathcal{N}_1 \RC\right] &\text{ in }\Omega, \\
				R_2=0 &\text{ in }\Omega_e.
			\end{cases}
		\end{align}
        
		Thus, we see that $v_0\in H^1(0,T;\widetilde{H}^s(\Omega))$ solves
		\begin{align}
			\begin{cases}
				\rho \p_t v_0^{\frac{1}{m}}+L_K v_0 =h^{\frac{1}{m}}\rho \LC V^{(0)}\RC^{\frac{1}{m}}\geq 0 & \text{ in }\Omega_{T_0}, \\
				v_0 =ht^m \widetilde{\varphi}_0 &\text{ in } (\Omega_e)_{T_0}, \\
				v_0(0)=0 &\text{ in }\Omega.
			\end{cases} 
		\end{align}
		Next, we want to show that by our comparison principle (Theorem \ref{thm: comparison absorption}) there holds
        \begin{equation}
        \label{eq: cons comp principle}
            v_0(x,t)\geq v(x,t)\ \text{for a.e.}\ (x,t)\in \R^n_{T_0}.
        \end{equation}
        For this purpose let $\LC \Phi_{\vareps}\RC _{\vareps>0}$ be the functions constructed in Lemma~\ref{auxiliary lemma}. First notice that 
        \[
            v_0=ht^mV^{(0)}=\LC h^{1/m}t (V^{(0)})^{1/m}\RC^m\in H^1(0,T;\widetilde{H}^s(\Omega))
        \]
        and thus the function $u_0=h^{1/m}t (V^{(0)})^{1/m}$ satisfies
        \begin{equation}
        \label{eq: eq for u_0}
				\begin{cases}
					\rho \partial_t u_0+ L_K(\Phi(u_0))=F_0&\text{ in }\Omega_T,\\
					u_0=\varphi_0&\text{ in }(\Omega_e)_T,\\
					u_0(0)=0&\text{ in }\Omega
				\end{cases}
			\end{equation}
        with $F_0=\Phi^{-1}(h)\rho \Phi^{-1}( V^{(0)})$ and $\varphi_0=\Phi^{-1}(h)t \Phi^{-1}(\widetilde{\varphi}_0)$.
        Next, we introduce the functions $u_{0,\vareps}=\Phi_{\vareps}^{-1}(h)t\Phi^{-1}_{\vareps}(V^{(0)})$.
        Note that the uniform ellipticity of $\Phi'_{\vareps}$ implies that $\Phi_{\vareps}$ is bi-Lipschitz and hence by Remark~\ref{remark: hoelder coeff estimate} we have $\Phi^{-1}_{\vareps}(V^{(0)})\in H^s(\R^n)$. Thus, $u_{0,\vareps}\in H^1(0,T;H^s(\R^n))$. 
        
        Let us assert that 
        \begin{equation}
        \label{eq: convergence of inverse}
            \Phi_{\vareps}^{-1}(t)\to \Phi^{-1}(t)\ \text{for all}\ t\in\R
        \end{equation}
        as $\vareps\to 0$. For completeness we give a proof of this fact. Let $t\in\R$ and $\sigma>0$. By continuity of $\Phi^{-1}$, there exists $\delta>0$ such that
        \[
            \left|\Phi^{-1}(\tau)-\Phi^{-1}(t)\right|\leq \sigma
        \]
        for all $\tau\in \overline{B_{\delta}(t)}$. This in particular implies
        \[
            \left|\Phi^{-1}(t+\delta)-\Phi^{-1}(t)\right|\leq \sigma\quad  \text{and}\quad  \left|\Phi^{-1}(t-\delta)-\Phi^{-1}(t)\right|\leq \sigma.
        \]
        Using the monotonicity of $\Phi^{-1}$, we obtain
        \begin{equation}
        \label{eq: monotonicity estimate for cont of inv}
            \Phi^{-1}(t+\delta)\leq \Phi^{-1}(t)+\sigma\quad  \text{and}\quad   \Phi^{-1}(t)-\sigma\leq \Phi^{-1}(t-\delta).
        \end{equation}
        On the other hand, as $\Phi_{\vareps}\to \Phi$ as $\vareps \to 0$ on compact sets and $\Phi^{-1}$ is continuous, we conclude that there exists $\vareps_0>0$ such that
        \[
            \left|\Phi_{\vareps}(z)-\Phi(z)\right|<\delta,  \text{ for all }   z\in \Phi^{-1}(\overline{B_{\delta}(t)}) \text{ and }  0<\vareps <\vareps_0,
        \]
        and in particular,
        \[
            \left|\Phi_{\vareps}(\Phi^{-1}(t+\delta))-(t+\delta)\right|<\delta\quad \text{and}\quad \left|\Phi_{\vareps}(\Phi^{-1}(t-\delta))-(t-\delta)\right|<\delta
        \]
        for all $\vareps<\vareps_0$. Now, this implies
        \[
            t\leq \Phi_{\vareps}(\Phi^{-1}(t+\delta))\quad  \text{and}\quad  \Phi_{\vareps}(\Phi^{-1}(t-\delta))\leq t,
        \]
        for all  $0<\vareps<\vareps_0$. Hence, by the monotonicity of $\Phi_{\vareps}^{-1}$ we deduce 
        \[
                \Phi^{-1}(t-\delta)\leq \Phi_{\vareps}^{-1}(t)\leq \Phi^{-1}(t+\delta),
        \]
        for all  $0<\vareps<\vareps_0$. Recalling \eqref{eq: monotonicity estimate for cont of inv}, we find
        \[
            \Phi^{-1}(t)-\sigma\leq \Phi_{\vareps}^{-1}(t)\leq \Phi^{-1}(t)+\sigma
        \]
        and thus
        \[
            \left|\Phi^{-1}_{\vareps}(t)-\Phi^{-1}(t)\right|\leq\sigma
        \]
        for all  $0<\vareps<\vareps_0$. This completes the proof of \eqref{eq: convergence of inverse}.
        
        Therefore, using  \eqref{eq: convergence of inverse} we deduce that
        \[
            \Phi_{\vareps}^{-1}(V^{(0)})\to \Phi^{-1}(V^{(0)})\text{ as }\vareps\to 0, \text{ for a.e. } x\in\Omega.
        \]
       Recalling that by \eqref{eq: nonnega of V0} and \eqref{eq: upper bound V0} there holds $0\leq V^{(0)}(x,t)\leq M$, we see from the monotonicity of $\Phi_{\vareps}^{-1}$, $\Phi_{\vareps}^{-1}(t)\geq 0$ for $t\geq 0$ and $\Phi_{\vareps}^{-1}(t)=\Phi^{-1}(t)$ for $t\in [\Phi(\vareps),\Phi(1/\vareps)]$ that $\Phi_{\vareps}^{-1}(V^{(0)})$ is uniformly bounded in $\vareps$. In fact,
        \[
            0\leq \Phi_{\vareps}^{-1}(V^{(0)}))\leq \Phi_{\vareps}^{-1}(\max(1,M))=\Phi^{-1}(\max(1,M))
        \]
        for $\vareps>0$ sufficiently small. Thus, by Lebesgue's dominated convergence theorem we deduce that
        \begin{equation}
        \label{eq: uniform and L1 conv}
            \Phi_{\vareps}^{-1}(V^{(0)})\to \Phi^{-1}(V^{(0)})\ \text{in}\ L^1(\Omega_{T_0}) \text{ as }\vareps\to 0.
        \end{equation}
        Then clearly the same holds for $u_{0,\vareps}$. Furthermore, note that the functions $u_{0,\vareps}$ satisfy
        \begin{equation}
        \label{eq: approximate problem for u0}
				\begin{cases}
					\rho \partial_t u_{0,\vareps}+ L_K(\Phi_{\vareps}(u_{0,\vareps}))=F_0^{\vareps}&\text{ in }\Omega_T,\\
					u_{0,\vareps}=\varphi_0^{\vareps}&\text{ in }(\Omega_e)_T,\\
					u_{0,\vareps}(0)=0&\text{ in }\Omega.
				\end{cases}
			\end{equation}
        with $F_0^{\vareps}=\rho\Phi_{\vareps}^{-1}(h)t\Phi^{-1}_{\vareps}(V^{(0)})$ and $\varphi_0^{\vareps}=\Phi_{\vareps}^{-1}(h)t\Phi^{-1}_{\vareps}(\widetilde{\varphi}_0)$.
        Meanwhile, let us recall that by construction $u=\Phi^{-1}(v)\geq 0$ solves  
        \begin{align}
			\label{transferred equation v}
			\begin{cases}
				\rho \p_t u +L_K(\Phi(u)) =F_1 &\text{ in }\Omega_T, \\
				u=\varphi_1 & \text{ in } (\Omega_e)_T,\\
				u(0)=0 &\text{ in }\Omega
			\end{cases}
		\end{align}
        with $F_1=-qu$ and $\varphi_1=\Phi^{-1}(h) t \Phi^{-1}(\widetilde{\varphi}_0)$.
        Moreover, by Remark~\ref{remark: comparison holds for solutions} there is a sequence $0\leq u_{\vareps}\in H^1(0,T;L^2(\Omega))\cap L^2(0,T;H^s(\R^n))$, $\vareps>0$, satisfying 
             \begin{equation}
             \label{transferred equation v 2}
				\begin{cases}
					\rho\partial_t u_{\vareps}+ L_K(\Phi_{\vareps}(u_{\vareps}))=F_1^{\vareps} & \text{ in }\Omega_T,\\
					u_{\vareps}=\varphi_1^{\eps}&\text{ in }(\Omega_e)_T,\\
					u_{\vareps}(0)=0&\text { in }\Omega.
				\end{cases}
			\end{equation}
        with $F_1^{\vareps}=-q u_{\vareps}$ and $\varphi_1^{\vareps}=\Phi^{-1}(h) t \Phi^{-1}(\widetilde{\varphi}_0)$.

        As a matter of fact, we can show that the functions $u_0,u_{0,\vareps}, u$ and $u_{\vareps}$ fulfill  all conditions in Theorem~\ref{thm: comparison absorption} without absorption term and zero initial condition\footnote{Hence, Theorem~\ref{thm: basic comparison} is enough for our purposes.}. First, let us observe that $q\geq 0$, $\rho >0$, \eqref{eq: nonnega of V0} and $\Phi^{-1}(t), \Phi_{\vareps}^{-1}(t)\geq $ for $t\geq 0$ imply
        \begin{equation}
        \label{eq: nonnegative of F0}
            F_1-F_0\leq 0\ \text{and} \ F_1^{\vareps}-F_0^{\vareps}\leq 0.
        \end{equation}
        This shows that property \ref{prop 1 comparison abs} of Theorem~\ref{thm: comparison absorption} holds. Next, using \eqref{eq: convergence of inverse} we observe that
        \[
            \begin{split}
                \varphi_0^{\vareps}&=\Phi_{\vareps}^{-1}(h)t\Phi^{-1}_{\vareps}(\widetilde{\varphi}_0)\to \Phi^{-1}(h)t \Phi^{-1}(\widetilde{\varphi}_0)=\varphi_0 \text{ as }\vareps\to 0.
            \end{split}
        \]
         This pointwise convergence, $\Phi_{\vareps}\to \Phi$ uniformly on compact sets as $\vareps\to 0$ and $\Phi$ is continuous, we have
        \[
        \begin{split}
            \Phi_{\vareps}(\varphi_1^{\vareps})-\Phi_{\vareps}(\varphi_0^{\vareps})&=\Phi_{\vareps}(\Phi^{-1}(h) t \Phi^{-1}(\widetilde{\varphi}_0))-\Phi_{\vareps}(\Phi_{\vareps}^{-1}(h)t\Phi^{-1}_{\vareps}(\widetilde{\varphi}_0))\to 0
        \end{split}
        \]
        as $\vareps\to 0$. Thus, we have in particular
        \[
            \LC \Phi_{\vareps}(\varphi_1^{\vareps})-\Phi_{\vareps}(\varphi_0^{\vareps})\RC_+\to 0=\LC \varphi_1-\varphi_0\RC_+\ \text{as}\ \vareps\to 0,
        \]
        where we used that $\varphi_0=\varphi_1$. Recall that the support of $\Phi_{\vareps}(\varphi_1^{\vareps})-\Phi_{\vareps}(\varphi_0^{\vareps})$ is compactly contained in $[0,T]\times W_1$. Hence, by Lebesgue's dominated convergence theorem we can conclude the property \ref{prop 2 comparison abs} of Theorem~\ref{thm: comparison absorption} with equality and right hand side equal to zero. Thus, all assumptions of Theorem~\ref{thm: comparison absorption} are satisfied and we can deduce 
        \[
                \int_{\Omega_{T_0}}\LC u-u_{0}\RC _+(x,t)\,dxdt=0.
        \]
        This give $u\leq u_0$ and thus by the monotonicity of $\Phi^{-1}$ the desired estimate \eqref{eq: cons comp principle}.
        

		Furthermore, by $\rho\geq 0$, \eqref{eq: cons comp principle}, the triangle inequality, H\"older's inequality, \eqref{eq: calc bdry val}, \eqref{eq: easy integration} and \eqref{ansatz 1}, we can estimate
		\begin{align}\label{M_1-M estimate}
			\begin{split}
				0&\leq \mathcal{M}^{(1)}-\mathcal{M} \\
				&=\beta \rho\int_0^{T_0} (T_0-t)^{\beta-1} \LC v_0^{1/m}-v^{1/m}\RC dt  \\
				&\leq  \beta \rho \int_0^{T_0} (T_0-t)^{\beta-1} \LC v_0-v\RC^{1/m}\, dt \\
				&\leq T_0^{\frac{\beta}{m'}-\frac{1}{m}}\frac{\beta \rho}{(\beta-m'+1)^{\frac{1}{m'}}} \LC \mathcal{B}(\beta+1,m+1)T_0^{\beta+m+1} h V^{(0)} -V \RC^{1/m} \\
				&= T_0^{\frac{\beta}{m'}-\frac{1}{m}}\frac{\beta \rho}{(\beta-m'+1)^{\frac{1}{m'}}}\LC -R_1\RC ^{1/m}.
			\end{split}
		\end{align}
		Similarly, using $q\geq 0$, \eqref{eq: calc bdry val} and \eqref{ansatz 1}, we obtain
		\begin{align}\label{N_1-N estimate}
			\begin{split}
				0&\leq \mathcal{N}^{(1)}-\mathcal{N} \\
				&=q\int_0^{T_0} (T_0-t)^{\beta} \LC v_0^{1/m}-v^{1/m}\RC dt \\
                &\leq  q\int_0^{T_0} (T_0-t)^{\beta} \LC v_0-v\RC^{1/m}\, dt \\
                &\leq  q\frac{T_0^{(\beta+ 1)/m'}}{(\beta+1)^{1/m'}}\left(T_0^{\beta+m+1}\mathcal{B}(\beta+1,m+1)hV^{(0)}-V\right)^{1/m}\\
				&\leq  q\frac{T_0^{(\beta+ 1)/m'}}{(\beta+1)^{1/m'}}\left(-R_1\right)^{1/m}.
			\end{split}
		\end{align}
		Now, we can deduce the asymptotic behaviour\footnote{We use the Landau asymptotic notation that $A=\mathcal{O}(B)$ stands for that $B$ is nonnegative, and there is a positive constant $C$ such that $|A|\leq C B$ for large $h$.}
		\begin{align}
			\left\| \LC \mathcal{M}^{(1)}-\mathcal{M} \RC +\LC \mathcal{N}^{(1)}-\mathcal{N} \RC \right\|_{L^2(\Omega)}=\mathcal{O}\LC h^{\frac{1}{m^2}}\RC \text{ as } h\to \infty.
		\end{align}
        Here, we used \eqref{M_1-M estimate}, \eqref{N_1-N estimate}, Jensen's inequality, $\rho,q\in L^{\infty}(\R^n)$ and \eqref{R_1 estimate}.
		Combining this with \eqref{R_2 equation} and Lemma~\ref{cor: exterior value}, we infer
		\begin{align}
        \label{eq: asymp vanishing of R2}
			\left\| R_2 \right\|_{H^s(\R^n)}=\mathcal{O}\LC h^{\frac{1}{m^2}}\RC \text{ as } h\to \infty.
		\end{align}

        Next, let us denote by $\psi_{w,h}\in L^2(0,T;\widetilde{H}^s(W_2))$, the function
        \begin{equation}
        \label{eq: help function uniqueness ip}
            \psi_{w,h}(x,t)=\begin{cases}
                h^{-1}\LC T_0-t\RC ^{\beta}w(x),& \text{ if }0\leq t\leq T_0,\\
                0,& \text{ otherwise},
            \end{cases}
        \end{equation}
        where $w\in C_c^{\infty}(W_2)$. Using \eqref{exterior data in the proof}, \ref{ansatz 2} and \eqref{R_2 equation}, we obtain
        \begin{equation}
        \label{eq: asymptotic expansion}
            \begin{split}
                \left\langle \Lambda^{\Phi}_{\rho,K,q}\widetilde{\varphi},\psi_{w,h}\right\rangle &=\int_0^{T_0}(T_0-t)^{\beta}h^{-1}B_K(v,w)\,dt\\
                &=B_K(V,w)\\
                &=h^{-1}B_K(h \mathcal{B}(\beta+1,m+1)T_0^{\beta+m+1} V ^{(0)} + h^{\frac{1}{m}}V ^{(1)} +R_2,w)\\
                &= \mathcal{B}(\beta+1,m+1)T_0^{\beta+m+1}B_K(V ^{(0)},w) + h^{\frac{1}{m}-1}B_K(V ^{(1)},w).
            \end{split}
        \end{equation}
        In the limit $h\to\infty$, we obtain
        \begin{equation}
        \label{eq: asymp expansion of DN map for K}
            \lim_{h\to\infty}\left\langle \Lambda^{\Phi}_{\rho,K,q}\widetilde{\varphi},\psi_{w,h}\right\rangle=\mathcal{B}(\beta+1,m+1)T_0^{\beta+m+1}B_K(V ^{(0)},w)
        \end{equation}
        for all $w\in C_c^{\infty}(W_2)$. Recall that $V^{(0)}\in H^s(\R^n)$ is the unique solution of \eqref{V_0 equation} with exterior value $\widetilde{\varphi}_0$. Hence, if we denote the DN map of this equation by $\Lambda_K\colon \widetilde{H}^s(\Omega_e)\to H^{-s}(\Omega_e)$, then \eqref{eq: asymp expansion of DN map for K} means nothing else than 
        \begin{equation}
        \label{eq: asymp expansion of DN map for K 2}
            \lim_{h\to\infty}\langle \Lambda^{\Phi}_{\rho,K,q}\widetilde{\varphi},\psi_{w,h}\rangle=\mathcal{B}(\beta+1,m+1)T_0^{\beta+m+1}\langle \Lambda_K\widetilde{\varphi}_0,w\rangle
        \end{equation}
        for all $w\in C_c^{\infty}(W_2)$ and $\widetilde{\varphi}_0\in \test_+(W_1)$.

		\subsection{Proof  of Theorem \ref{thm: main}}

		\begin{proof}[Proof of Theorem \ref{thm: main}]
			The proof consists of two steps. In the first step we determine the kernel $K$ and then the coefficients $\rho$ and $q$. Let us start by recalling that the assumptions on the DN maps $\Lambda_{\rho_j,K_j,q_j}$ for $j=1,2$ guarantee that the transferred DN maps $\Lambda^{\Phi}_{\rho_j,K_j,q_j}$, $j=1,2$, satisfy
            \begin{equation}
              \label{eq: equal transferred DN maps}  \langle\Lambda^{\Phi}_{\rho_1,K_1,q_1}\widetilde{\varphi},\psi\rangle=\langle\Lambda^{\Phi}_{\rho_2,K_2,q_2}\widetilde{\varphi},\psi\rangle,
            \end{equation}
            for all $\widetilde{\varphi}\in \test_+\LC [0,T]\times W_1\RC$ and $\psi \in L^2(0,T;\widetilde{H}^s(W_2))$ (see~\eqref{eq: DN map transferred equation} and \eqref{eq: consequence on transferred DN maps}). Moreover, let $T_0>0$ be sufficiently small such that the results of Section~\ref{subsec: integral time transform} and \ref{subsec: asymptotic analysis} hold, but otherwise be arbitrary.\\
			
			\noindent\textit{Step 1: Unique determination of the kernel}. \\
   
          First, let us recall that from the asymptotic expansion of the DN maps \eqref{eq: asymp expansion of DN map for K 2}, we deduce
            \begin{equation}
                \left\langle \Lambda_{K_1}\widetilde{\varphi}_0,w \right\rangle=\left\langle \Lambda_{K_2}\widetilde{\varphi}_0,w\right\rangle
            \end{equation}
            for all $w\in C_c^{\infty}(W_2)$ and $\widetilde{\varphi}_0\in \test_+(W_1)$.
        Since $L_{K_1}$ and $L_{K_2}$ are measurement equivalent, we can deduce that there holds
			\begin{align}\label{K1=K2}
				K(x,y):=K_1(x,y)=K_2(x,y), \text{ for }x,y\in \R^n
			\end{align}
			as desired (see Definition \ref{def: nonlocal op for inverse problem}). 
			This proves the first step. \\
	
			\noindent\textit{Step 2: Unique determination of coefficients}.\\
   
         We now prove the unique determination result for both $\rho_1 =\rho_2$ and $q_1=q_2$ in $\overline{\Omega}$. Now, let $\psi_{w,h^{\sigma}}$ be the function from  \eqref{eq: help function uniqueness ip} with $h$ replaced by $h^{\sigma}$ and $w\in C_c^{\infty}(W_1)$.  From the results of Section~\ref{subsec: asymptotic analysis} and in particular \eqref{eq: asymptotic expansion}, we known that for any $\sigma>0$ and $j=1,2$ there holds
            \begin{equation}
            \label{eq: asymp for uniqueness rho q}
            \begin{split}
                \langle \Lambda^{\Phi}_{\rho_j,K,q_j}\widetilde{\varphi},\psi_{w,h^{\sigma}}\rangle &=h^{-\sigma}\int_0^{T_0}(T_0-t)^{\beta}B_K(v,w)\,dt\\
                &=B_K(V_j,w)\\
                &=h^{-\sigma}B_K(h \mathcal{B}(\beta+1,m+1)T_0^{\beta+m+1} V ^{(0)} + h^{\frac{1}{m}}V ^{(1)}_j +R_{2,j},w),
            \end{split}
            \end{equation}
            where $V_j\in H^s(\R^n)$ is given by
            \[
                V_j=\int_0^{T_0}(T_0-t)^{\beta}v_j\,dt
            \]
            with $v_j$ denoting the unique solution of \begin{align}
			\label{transferred equation unique coef rho q}
			\begin{cases}
				\rho_j \p_t \Phi^{-1}(v) +L_K(v) + q_j \Phi^{-1}(v)=0 &\text{ in }\Omega_T, \\
				v=\widetilde{\varphi} & \text{ in } (\Omega_e)_T,\\
				v(0)=0 &\text{ in }\Omega
			\end{cases}
		\end{align}
            for $\widetilde{\varphi}$ as in \eqref{exterior data in the proof}.

            Moreover, the asymptotic expansions of $V_j$ are denoted as
            \[
                V=\widetilde{V}_h=h \mathcal{B}(\beta+1,m+1)T_0^{\beta+m+1} V ^{(0)} + h^{\frac{1}{m}}V ^{(1)}_j +R_{2,j},
            \]
            where the $V^{(0)}$ are the same as they are solutions to \eqref{V_0 equation} with the same kernel. Subtraction of the expansions \eqref{eq: asymp for uniqueness rho q} for $j=1$ and $j=2$ gives
            \begin{equation}
            \label{eq: vanishing limit of R}
                \left\langle (\Lambda^{\Phi}_{\rho_1,K,q_1}-\Lambda^{\Phi}_{\rho_2,K,q_2})\widetilde{\varphi},\psi_{w,h^{\sigma}}\right\rangle=h^{-\sigma}B_K\LC h^{\frac{1}{m}}(V ^{(1)}_1-V ^{(1)}_2)+(R_{2,1}-R_{2,2}),w\RC .
            \end{equation}
            Now, we take $\sigma=h^{1/m}$ and pass to the limit $h\to \infty$ to obtain
            \begin{equation}
            \label{eq: asymp expansion 2 for uniqueness rho q}
                \lim_{h\to \infty} \left\langle (\Lambda^{\Phi}_{\rho_1,K,q_1}-\Lambda^{\Phi}_{\rho_2,K,q_2})\widetilde{\varphi},\psi_{w,h^{\sigma}}\right\rangle=B_K(V_1^{(1)}-V^{(1)}_2,w)
            \end{equation}
            for all $w\in C_c^{\infty}(W_2)$. Here, we are using that by \eqref{eq: asymp vanishing of R2} for $j=1,2$ we have $R_{2,j}=\mathcal{O}(h^{1/m^2})$ in $H^s(\R^n)$ and hence $h^{-1/m}R_{2,j} \to 0$ in $H^{s}(\R^n)$ as $h\to\infty$. The last limit vanishes as $1/m^2-1/m<0$. This convergence and the uniform ellipticity of $K$, now implies that the second term in \eqref{eq: vanishing limit of R} goes to zero in the limit $h\to\infty$. From \eqref{eq: asymp expansion 2 for uniqueness rho q} and the definition of $V^{(1)}_j$, $j=1,2$, we deduce that the function $V^{(1)}_1-V^{(1)}_2\in H^s(\R^n)$ satisfies
            \[
                L_K\LC V^{(1)}_1-V^{(1)}_2\RC =V^{(1)}_1-V^{(1)}_2=0\ \text{in}\ W_2
            \]
            (see \eqref{V_1 equation}). As $L_K$ satisfies the UCP on $H^s(\R^n)$ as we assumed, this implies $V^{(1)}_1=V^{(1)}_2$ in $\R^n$.

            Now, subtracting the Dirichlet problems \eqref{V_1 equation} for $V^{(1)}_1$ and $V^{(1)}_2$, gives
            \[
                \int_{\Omega} \left[\LC \mathcal{M}^{(1)}_1-\mathcal{M}^{(1)}_2\RC +\LC \mathcal{N}^{(1)}_1-\mathcal{N}^{(1)}_2\RC \right]\psi\,dx=0
            \]
            for any $\psi\in \widetilde{H}^s(\Omega)$, where for $j=1,2$ the quantities $\mathcal{M}^{(1)}_j,\mathcal{N}^{(1)}_j$ are given by
            \begin{equation}
                \begin{split}
                    \mathcal{M}^{(1)}_j&=h^{\frac{1}{m}}\rho_j \frac{2T_0^{\beta+1}}{\beta(\beta+1)}\LC V^{(0)}\RC^{\frac{1}{m}},\\
                    \mathcal{N}^{(1)}_j&=h^{\frac{1}{m}} q_j \frac{2T_0^{\beta+2}}{(\beta+1)(\beta+2)}\LC V^{(0)}\RC^{\frac{1}{m}}
                \end{split}
            \end{equation}
            (see~\eqref{M_1} and \eqref{N_1}). This implies
            \begin{equation}
            \label{eq: characteristic equation for uniqueness}
                \int_{\Omega}\left[\beta^{-1}\left(\rho_1 -\rho_2 \right) +\frac{T_0}{\beta+2}\left(q_1-q_2 \right)\right]\LC V^{(0)}\RC^{\frac{1}{m}} \psi\,dx=0
            \end{equation}
            for any $\psi\in \widetilde{H}^s(\Omega)$. Passing to the limit $T_0\to 0$ shows 
            \[
                \int_{\Omega}(\rho_1-\rho_2)\LC V^{(0)}\RC^{\frac{1}{m}} \psi\,dx=0
            \]
            for all $\psi\in \widetilde{H}^s(\Omega)$. This implies
            \begin{equation}
            \label{eq: char eq for rho}
                (\rho_1-\rho_2)\LC V^{(0)}\RC^{\frac{1}{m}}=0\ \text{a.e. in}\ \Omega.
            \end{equation}
            Let us assert that this implies $\rho_1=\rho_2$. For contradiction assume that $\rho_1(x_0)\neq \rho_2(x_0)$, then by continuity of $\rho_1,\rho_2$ implies that there would be an $r>0$ such that $\rho_1\neq \rho_2$ on $B_r(x_0)\subset\Omega$. Now, \eqref{eq: char eq for rho} would imply $V^{(0)}=0$ on $B_r(x_0)$. Hence, the UCP on $H^s(\R^n)$ for $L_K$ and \eqref{V_0 equation} implies $V^{(0)}=0$ in $\R^n$, which is impossible as $\widetilde{\varphi}_0\neq 0$. Thus, we conclude that $\rho_1=\rho_2$ in $\Omega$. Now, from the continuity of $\rho_1,\rho_2$, we infer $\rho_1=\rho_2$ in $\overline{\Omega}$. Turning back to equation \eqref{eq: characteristic equation for uniqueness}, we see that 
            \[
                \int_{\Omega}(q_1-q_2 )\LC V^{(0)}\RC^{\frac{1}{m}} \psi\,dx=0
            \]
            for all $\psi\in\widetilde{H}^s(\Omega)$. Now, arguing exactly in the same way as for $\rho_1,\rho_2$, we can conclude that $q_1=q_2$ in $\overline{\Omega}$. This finishes the proof.
            \end{proof}

		\appendix
		
		\section{Some compact embeddings}\label{sec: appendix compact}
		
		For the convenience of the reader, we collect here two known compactness results. Here we use the following notation. If $F\subset \distr((0,T);X)$ for a Banach space $X$ and $T>0$, then we set $\partial_t F=\{\partial_tf\,;\,f\in F\}$. 
		
		\begin{theorem}[{Aubin--Lions lemma, \cite[Corollary~4]{Simon}}]
			\label{Aubin-Lions lemma}
			Let $X\hookrightarrow B\hookrightarrow Y$ be Banach spaces, where the first embedding is compact, and $1\leq p<\infty$, $1<r\leq \infty$. If $F$ is bounded set in $L^p(0,T;X)$ and $\partial_t F$ bounded in $L^1(0,T;Y)$, then $F$ is relatively compact in $L^p(0,T;B)$. If $F$ is bounded in $L^{\infty}(0,T;X)$ and $\partial_t F$ bounded in $L^r(0,T;Y)$, then $F$ is relatively compact in $C([0,T];B)$.
		\end{theorem}
		
		\begin{theorem}[{Aubin--Lions--Simon lemma, \cite[Corollary~5]{Simon}}]
			\label{Aubin-Lions-Simon lemma}
			Let $X\hookrightarrow B\hookrightarrow Y$ be Banach spaces, where the first embedding is compact, and $1\leq p\leq\infty$, $1\leq r\leq \infty$. If $F$ is bounded in $L^p(0,T;X)\cap W^{s,r}(0,T;Y)$ with $s>0$, $r\geq p$ or $s>1/r-1/p$, $r\leq p$. Then $F$ is relatively compact in $L^p(0,T;B)$ if $p<\infty$ and otherwise in $C([0,T];B)$.
		\end{theorem}

		\section{A density results}\label{sec: appendix density}
		In some of our proofs the following density result will be important. 
		\begin{proposition}
			\label{prop: density}
			Let $\Omega\subset\R^n$ be a bounded Lipschitz domain, $T>0$ and $0<s<1$. Then the space of test functions $C_c^{\infty}([0,T)\times \Omega)$ is dense in
			\begin{equation}
				\label{eq: vanishing end trace}
				\mathcal{W}_T=\{u\in H^1(0,T;\widetilde{H}^s(\Omega))\,;\,u(T)=0\}.
			\end{equation}
		\end{proposition}
		
		\begin{proof}
			First of all recall that by standard results we have $H^1(0,T;\widetilde{H}^s(\Omega))\hookrightarrow C([0,T];\widetilde{H}^s(\Omega))$ and hence $\mathcal{W}_T$ is as closed subspace of $H^1(0,T;\widetilde{H}^s(\Omega))$ itself a Hilbert space. First, we show that any element in $\mathcal{W}_T$ can be approximated by elements in
			\begin{equation}
				\label{eq: truncated subspace}
				\mathcal{V}_T=\left\{v\in \mathcal{W}_T\,:\,v(t)=0\,\text{in a neighborhood of}\,t=T \right\}.
			\end{equation}
			For this purpose choose $u\in\mathcal{W}_T$. Up to time inversion, translation and scaling we can assume that $u\in H^1(0,1;\widetilde{H}^s(\Omega))$ vanishes at $t=0$. Next, we take any $\eta\in C^{\infty}(\R)$ satisfying
			\[
			0\leq \eta\leq 1\quad\text{and}\quad \eta(t)=\begin{cases}
				0,&\quad t\leq 1,\\
				1,&\quad t\geq 2.
			\end{cases}
			\]
			For any $k\in\N$, we now introduce the sequence $\eta_k(t)=\eta(kt)$ and define $u_k(t)=\eta_k(t)u(t)$. Clearly, $u_k\in H^1(0,1;\widetilde{H}^s(\Omega))$ and $u_k(t)=0$ in a neighborhood of $t=0$. Hence, it remains to prove that $u_k\to u$ in $H^1(0,1;\widetilde{H}^s(\Omega))$. The convergence $u_k\to u$ in $L^2(0,1;\widetilde{H}^s(\Omega))$ is an immediate consequence of Lebesgue's dominated convergence theorem. Next, observe that by the product rule we have $\partial_t u_k=\eta'_k u+\eta_k\partial_t u$.
			
			Since $\partial_t u\in L^2(0,1;\widetilde{H}^s(\Omega))$, we have again $\eta_k\partial_t u\to \partial_t u$ in $L^2(0,1;\widetilde{H}^s(\Omega))$. Hence, if we can show $\eta'_k u\to 0$ in $L^2(0,1;\widetilde{H}^s(\Omega))$, then we have established $u_k\to u$ in $H^1(0,1;\widetilde{H}^s(\Omega))$. To see this note that
			\begin{equation}
				\label{eq: estimate for density}
				\begin{split}
					\|\eta'_k u\|^2_{L^2(0,1;H^s(\R^n))}&\leq\int_0^1|\eta'_k(t)|^2\|u(t)\|_{H^s(\R^n)}^2\,dt\\
					&=k^2\int_{1/k}^{2/k}|\eta'(kt)|^2\|u(t)\|_{H^s(\R^n)}^2\,dt\\
					&\leq k^2\int_{1/k}^{2/k}\|u(t)\|_{H^s(\R^n)}^2\,dt\\
					&=Ck^2\int_{1/k}^{2/k}t^2\frac{\|u(t)\|_{H^s(\R^n)}^2}{t^2}\,dt\\
					&\leq C\int_{1/k}^{2/k}\frac{\|u(t)\|_{H^s(\R^n)}^2}{t^2}\,dt.
				\end{split}
			\end{equation}
			Thus, if we can show that $\frac{\|u(t)\|_{H^s(\R^n)}^2}{t^2}\in L^1(0,1)$, then it follows from the absolute continuity of the Lebesgue integral that $\eta'_k u\to 0$ in $L^2(0,1;\widetilde{H}^s(\Omega))$. As $u\in H^1(0,1;\widetilde{H}^s(\Omega))$ one has 
			\begin{equation}
				\label{fundamental theorem of analysis}
				u(t)=u(s)+\int_s^t \partial_t u\,d\tau
			\end{equation}
			for $0\leq s\leq t\leq 1$. Moreover, that Hardy's inequality states that for any $1<p<\infty$ and any measurable function $f\colon\R_+\to \R_+$ one has
			\begin{equation}
				\label{eq: Hardys inequality}
				\int_0^{\infty}\left(\frac{1}{t}\int_0^tf(s)\,ds\right)\,dt\leq \left(\frac{p}{p-1}\right)^p\int_0^{\infty}(f(t))^p\,dt.
			\end{equation}
			Hence using \eqref{fundamental theorem of analysis}, $u(0)=0$ and \eqref{eq: Hardys inequality}, we get
			\[
			\begin{split}
				\int_0^1\left(\frac{\|u(t)\|_{H^s(\R^n)}}{t}\right)^2\,dt&\leq \int_0^1\left(\frac{1}{t}\int_0^t\|\partial_t u(\tau)\|_{H^s(\R^n)}d\tau\right)^2\,dt\\
				&\leq 4\int_0^1 \|\partial_t u(t)\|_{H^s(\R^n)}^2\,dt.
			\end{split}
			\]
			This establishes the required integrability of $\|u(t)\|_{H^s(\R^n)}^2/t^2$ and hence $\eta'_ku \to 0$ in $L^2(0,1;\widetilde{H}^s(\Omega))$. Thus, we can conclude that $\mathcal{V}_T$ is dense in $\mathcal{W}_T$. 
			
			Next we show that any $u\in \mathcal{V}_T$ can be approximated by elements in 
			\[
			\mathcal{D}_T=\left\{\sum_{j=1}^N \varphi_j w_j\,;\,w_j\in C_c^{\infty}(\Omega),\,\rho_j\in C_c^{\infty}([0,T)),\,1\leq j\leq N,\,N\in\N\right\}.
			\]
			This then establishes that $C_c^{\infty}([0,T)\times \Omega)$ is dense in $\mathcal{W}_T$. Let $u\in \mathcal{V}_T$. By assumption $\partial_t u\in L^2(0,T;\widetilde{H}^s(\Omega))$ and hence by standard results there exists a sequence $(v_n)_{n\in\N}$ in
			\[
			\mathscr{D}_T=\left\{\sum_{j=1}^N \rho_j \psi_j\,;\,\psi_j\in C_c^{\infty}(\Omega),\,\rho_j\in C_c^{\infty}((0,T)),\,1\leq j\leq N,\,N\in\N\right\}
			\]
			such that $v_n\to \partial_t v$ in $L^2(0,T;\widetilde{H}^s(\Omega))$. Now define $u_j\colon (0,T)\to\widetilde{H}^s(\Omega)$ by
			\[
			u_j(t)=-\int^T_t v_j(s)\,ds
			\]
			for $j\in\N$. Observe that since $u$ vanishes in a neighborhood of $t=T$, we have $u_j\in \mathcal{D}_T$. Note that there holds
			\[
			\partial_t u_j=v_j\to \partial_t u\,\text{in}\,L^2(0,T\,;\widetilde{H}^s(\Omega)).
			\]
			But since $u(T)=0$, the fundamental theorem of calculus guarantees
			\[
			u(t)=-\int^T_t\partial_t u\,ds
			\]
			and thus $u_j\to u$ in $L^2(0,T;\widetilde{H}^s(\Omega))$. In fact, first the convergence holds uniformly in $t$ and then by Lebesgue's dominated convergence theorem also in $L^2(0,T;\widetilde{H}^s(\Omega))$. This proves the assertion.
		\end{proof}

		\medskip 
		
		\subsection*{Acknowledgments}  Y.-H. Lin is partially supported by the National Science and Technology Council (NSTC) Taiwan, under the projects 111-2628-M-A49-002 and 112-2628-M-A49-003.

		\bibliography{refs} 
		
		\bibliographystyle{alpha}
		
	\end{document}